\documentclass[a4paper, reqno]{amsart}

\usepackage[T1]{fontenc}
\usepackage[latin1]{inputenc}
\usepackage{newlfont, verbatim, indentfirst, enumerate}

\usepackage{amsmath, amsthm, amssymb, amscd}
\usepackage{latexsym, amsfonts, amsbsy, mathrsfs}
\usepackage{array, pst-all, graphicx, rotating}
\usepackage{tikz} \usetikzlibrary{arrows} \usetikzlibrary{patterns}
\usepackage[unicode]{hyperref}

\allowdisplaybreaks
\numberwithin{equation}{section}

\theoremstyle{plain}
\newtheorem{theorem}{Theorem}[section]
\newtheorem{theo}[theorem]{Theorem}
\newtheorem{prop}[theorem]{Proposition}
\newtheorem{lem}[theorem]{Lemma}   
\newtheorem{cor}[theorem]{Corollary}

\theoremstyle{definition}
\newtheorem{defin}[theorem]{Definition}
\newtheorem{rem}[theorem]{Remark}

\newcommand {\emphatic} [1] {\emph{#1}}
\newcommand {\functor} [1] {\mathcal{#1}}  % any functor, 2-functor or pseudofunctor
\def \id {\operatorname{id}}  % identity of any given object in a fixed category
\def \Iso {\operatorname{Iso}}  % automorphism group of a given object

\def \EGpd {(\mathbf{\mathcal{\acute{E}}\,\mathcal{G}pd})}  % 2-category of \'etale groupoids
\def \SETW {\mathbf{W}}  % class of morphisms to be ``localized'' in a given bicategory
\def \SETWinv {\mathbf{W}^{-1}} % class of morphisms already ``localized'' in a given bicategory

\def \WEGpd {\mathbf{W}_{\mathbf{\mathcal{\acute{E}}\,\mathcal{G}pd}}}  % as above, for the class of all Morita equivalences between \'etale groupoids

\def \CATA {\mathbf{\mathscr{A}}}  % any category, 2-category of bicategory
\def \CATB {\mathbf{\mathscr{B}}}
\def \CATC {\mathbf{\mathscr{C}}}
\def \CATD {\mathbf{\mathscr{D}}}

\newcommand {\thetaa} [3] {\theta_{#1,#2,#3}}  % associativity 2-morphism of 3 composable morphisms in a given bicategory 
\newcommand {\thetab} [3] {\theta_{#1,#2,#3}^{-1}}  % inverse of the previous associativity 2-morphism
\newcommand {\Thetaa} [3] {\Theta_{#1,#2,#3}}  % associativity $2$-morphism in the bicategory of fractions of a given bicategory (the way to compute it is explained in~\cite[Appendix}{Pr}
\newcommand {\Thetab} [3] {\Theta_{#1,#2,#3}^{-1}}  % inverse of the previous associativity 2-morphism
\renewenvironment{itemize}{\begin{list}{$\bullet$}{\leftmargin=0.5cm}\parindent=0pt}{\end{list}}

\hypersetup{unicode=true, pdftoolbar=true, pdfmenubar=true, pdffitwindow=false,
  pdfstartview={FitH}, pdftitle={Weak fiber products in bicategories of fractions},
  pdfauthor={Matteo Tommasini}, pdfkeywords={Lie groupoids, bicategories of fractions, weak fiber
  products, weak pullbacks, stacks},
  pdfnewwindow=true,colorlinks=false, linkcolor=red, citecolor=red, filecolor=magenta, urlcolor=cyan}

\numberwithin{equation}{section}

\begin{document}

\title[Weak fiber products in bicategories of fractions]
{Weak fiber products in bicategories of fractions}

\author{Matteo Tommasini}

\address{\flushright Mathematics Research Unit\newline University of Luxembourg\newline
6, rue Richard Coudenhove-Kalergi\newline L-1359 Luxembourg\newline\newline
website: \href{http://matteotommasini.altervista.org/}
{\nolinkurl{http://matteotommasini.altervista.org/}}\newline\newline
email: \href{mailto:matteo.tommasini2@gmail.com}{\nolinkurl{matteo.tommasini2@gmail.com}}}

\date{\today}

\thanks{This research was performed at the Mathematics Research Unit of the University of Luxembourg,
thanks to the grant 4773242 by Fonds National de la Recherche Luxembourg.}

\begin{abstract}
We fix any pair $(\CATC,\SETW)$ consisting of a bicategory and a class of morphisms in it, admitting
a bicalculus of fractions, i.e.\ a ``localization'' of $\CATC$ with respect to the class $\SETW$.
In the resulting bicategory of fractions, we identify necessary and sufficient conditions for the
existence of weak fiber products.
\end{abstract}

\maketitle

\begingroup{\hypersetup{linkbordercolor=white}
\tableofcontents}\endgroup

\section*{Introduction}
In 1967 Pierre Gabriel and Michel Zisman proved in~\cite{GZ} that given a category $\CATC$ and a class
$\SETW$ of morphisms in it, satisfying $4$ technical conditions (called (\hyperref[CF1]{CF1}) --
(\hyperref[CF4]{CF4}), see Appendix~\ref{sec-04}), it is possible to construct a
``localization'' of $\CATC$ with respect to $\SETW$, i.e.\ a category $\CATC\left[\SETWinv\right]$
(called ``right category of fractions'') obtained from $\CATC$ by formally adding inverses for all the
morphisms in $\SETW$. To be more precise, objects of the category of fractions are the same as those
of $\CATC$; a morphism from $A$ to $B$ consists of an equivalence class of a triple $(A',
\operatorname{w},f)$ as follows

\begin{equation}\label{eq-136}
\begin{tikzpicture}[xscale=1.5,yscale=-1.2]
    \node (A0_0) at (0, 0) {$A$};
    \node (A0_1) at (1, 0) {$A'$};
    \node (A0_2) at (2, 0) {$B$,};
    \path (A0_1) edge [->]node [auto,swap] {$\scriptstyle{\operatorname{w}}$} (A0_0);
    \path (A0_1) edge [->]node [auto] {$\scriptstyle{f}$} (A0_2);
\end{tikzpicture}
\end{equation}
such that $\operatorname{w}$ belongs to $\SETW$ (we refer to Appendix~\ref{sec-04} for the
description of the equivalence relation used here). The technical conditions (\hyperref[CF]{CF})
mentioned before allow to prove that the compositions of such morphisms exists and that it satisfies
the usual properties of categories. Such a construction turned out to be very useful in several
branches of mathematics, for example homotopy theory and triangulated categories.\\

In 1997 Dorette Pronk generalized such a construction from categories to bicategories (see~\cite{Pr}).
To be more precise, given a bicategory $\CATC$ and a class $\SETW$ of morphisms in it,
satisfying  $5$ technical conditions (called (\hyperref[BF1]{BF1}) -- (\hyperref[BF5]{BF5}), see
Appendix~\ref{sec-03}), there is a ``right bicategory of fractions'' $\CATC\left[\SETWinv\right]$.
Such a bicategory in general is not unique, but any $2$ bicategories of fractions for the same
pair $(\CATC,\SETW)$ are equivalent using the axiom of choice.
Objects in $\CATC\left[\SETWinv\right]$ are the same as those of $\CATC$;
morphisms are given by triples $(A',\operatorname{w},f)$ as in \eqref{eq-136} (but not quotiented
by an equivalence relation, differently from the case of categories of fractions). $2$-morphisms
consist of classes of equivalence of quintuples of an object, a pair of morphisms and a pair of 
$2$-morphisms, satisfying some technical conditions (for more details,
we refer to Appendix~\ref{sec-03}).\\

In the case when $\CATC$ is a category (considered as a trivial bicategory), then the $5$
technical conditions (\hyperref[BF]{BF})
coincide with the $4$ technical conditions (\hyperref[CF]{CF}) and any resulting right bicategory
of fractions for $(\CATC,\SETW)$ is equivalent to the (trivial bicategory associated to) the 
right category of fractions for $(\CATC,\SETW)$.\\

Pronk introduced the notion of bicategory of fractions mainly in order to study certain bicategories
of stacks (we refer directly to~\cite{Pr} for details).
More recently, bicategories of fractions were used intensively mainly in relation with the
notion of butterflies; we refer to~\cite{AMMV}, \cite{MMV} and~\cite{R} for some recent interesting
development in this area.\\

The problem that we want to investigate in the present paper is the following: \emph{when
do weak fiber products exist in a bicategory of fractions}{? We recall that weak fiber products
are the natural generalization of (strong) fiber
products from categories to bicategories (in the case when the bicategory is a $2$-category, they
are also called $2$-fiber products; we refer to Definition~\ref{def-01} for the precise notion of
weak fiber product in any bicategory). Weak fiber products are one of the basic tools used whenever
one has to deal with a $2$-category or bicategory of stacks (on a given site). It is known that weak
fiber products of stacks (over a given site) exist because stackification commutes with $2$-fiber
products. However, very 
few is known in general about weak fiber products if
we restrict to a strict sub-$2$-category of stacks (for example, the sub-$2$-category
of differentiable stacks in the $2$-category of stacks over the site of smooth
manifolds, see e.g.\ \cite[Definition~8.1]{J}).
Frequently, such sub-$2$-categories can be described as (equivalent to) bicategories of fractions
(see e.g.~\cite[Corollary~43]{Pr} for a description of the $2$-category of differentiable stacks
as a bicategory of fractions). So it is interesting to understand under which conditions weak
fiber products exist in this framework.\\

If we try to understand the notion of weak fiber products in the case when we work in a
bicategory of fractions, we get soon stuck in a very complicated setup. Roughly
speaking (see Definition~\ref{def-01} for details), given any bicategory $\CATD$, any triple of
objects $A,B^1,B^2$ and any pair of morphisms $g^1:B^1\rightarrow A$ and $g^2:B^2\rightarrow A$,

\begin{enumerate}[(a)]
 \item a weak fiber product of $g^1$ and $g^2$ in
  $\CATD$ is the datum of $1$ object $C$, $2$ of morphisms $r^1:C\rightarrow B^1$, $r^2:C\rightarrow
  B^2$ and $1$ invertible $2$-morphism $\Omega:g^1\circ r^1\Rightarrow g^2\circ r^2$;
 \item in order to verify if a set $(C,r^1,r^2,\Omega)$ as above gives a weak fiber product, one has 
  to compare it against a set of $1$ object $D$, $4$ morphisms $s^1,s^2,t,t'$ and $3$
  $2$-morphisms $\Lambda,\Gamma^1,\Gamma^2$ (satisfying some technical conditions);
 \item the comparison of $(C,r^1,r^2,\Omega)$ against the set of data in (b) has to
  give back $1$ morphism $s$ and $3$ $2$-morphisms $\Lambda^1,\Lambda^2,\Gamma$ (satisfying
  some technical conditions).
\end{enumerate}

In the special case when $\CATD$ is a bicategory of fractions $\CATC\left[\SETWinv\right]$,
then the objects of $\CATD$ are the same as those of $\CATC$, the morphisms of $\CATD$ are triples
of an object and a pair of morphisms as in \eqref{eq-136} and the $2$-morphisms of $\CATD$ are
(classes of equivalence of) quintuples of an object, a pair of morphisms and a pair of $2$-morphisms
of $\CATC$. Therefore, (a) -- (c) above becomes:

\begin{enumerate}[(a)$'$]
 \item given any pair of morphisms with the same target in $\CATC\left[\SETWinv\right]$,
  a weak fiber product of them a priori consists of $4$ objects, $6$ morphisms
  and $2$ $2$-morphisms of $\CATC$;
 \item in order to verify if the set of data as in (a)$'$ is a weak fiber product in
  $\CATC\left[\SETWinv\right]$, a priori one has to
  compare it against a set of $8$ objects, $14$ morphisms and $6$ $2$-morphisms of $\CATC$
  (satisfying some technical conditions);
 \item the comparison of the data of (a)$'$ against the data of (b)$'$ has to
  give back $4$ objects, $8$ morphisms and $6$ $2$-morphisms of $\CATC$ (satisfying
  some technical conditions).
\end{enumerate}

This means that having fixed any pair of morphisms with the same target in a bicategory $\CATD$,

\begin{itemize}
 \item there are $16$ data of $\CATD$ (as in (a) -- (c))
  that we have either to construct (in order to define a weak fiber product) or to consider
  (in order to prove that what we constructed is actually a weak fiber product);
 \item if $\CATD=\CATC\left[\SETWinv\right]$, such data turn out to be given by $58$ data of
  $\CATC$.
\end{itemize}

As such, the problem of constructing a weak fiber product in a bicategory of fractions apparently
is very complicated. In the present paper we prove that such a problem can be considerably
simplified by reducing the $58$ data mentioned above to only $31$. To be more precise, first of all
we will show that it is sufficient to find $4$ data of $\CATC$ in order to define a weak fiber
product in $\CATC\left[\SETWinv\right]$ (instead of the $12$ data needed a priori in (a)$'$ above):

\begin{theo}\label{theo-02}
Let us fix any bicategory $\CATC$ and any class $\SETW$ of morphisms in it, satisfying axioms
\emphatic{(\hyperref[BF]{BF})}, and let us choose any bicategory of fractions
$\CATC\left[\SETWinv\right]$ associated to the
pair $(\CATC,\SETW)$. Given any pair of morphisms $f^1:B^1\rightarrow A$ and $f^2:B^2\rightarrow A$
in $\CATC$, the following facts are equivalent:

\begin{enumerate}[\emphatic{(}i\emphatic{)}]
 \item\label{i} for any pair of morphisms of $\SETW$ of the form $\operatorname{w}^1:B^1\rightarrow
  \overline{B}^1$, $\operatorname{w}^2:B^2\rightarrow\overline{B}^2$, the pair of morphisms
  
  \begin{equation}\label{eq-100}
  \begin{tikzpicture}[xscale=1.5,yscale=-0.8]
    \node (A0_2) at (2, 0) {$\overline{B}^1$};
    \node (A2_0) at (0, 2) {$\overline{B}^2$};
    \node (A2_2) at (2, 2) {$A$};
    
    \path (A0_2) edge [->]node [auto] {$\scriptstyle{(B^1,\operatorname{w}^1,f^1)}$} (A2_2);
    \path (A2_0) edge [->]node [auto,swap] {$\scriptstyle{(B^2,\operatorname{w}^2,f^2)}$} (A2_2);
  \end{tikzpicture}
  \end{equation}
  admits a weak fiber product in $\CATC\left[\SETWinv\right]$;
 
 \item\label{ii} there are an object $C$, a pair of morphisms $p^1:C\rightarrow B^1$, $p^2:C
  \rightarrow B^2$ and an invertible $2$-morphism $\omega:f^1\circ p^1\Rightarrow f^2\circ p^2$ in
  $\CATC$, such that the diagram
  
  \begin{equation}\label{eq-91}
  \begin{tikzpicture}[xscale=4.1,yscale=-0.8]
    \node (A0_0) at (0, 0) {$C$};
    \node (A0_2) at (2, 0) {$B^1$};
    \node (A2_0) at (0, 2) {$B^2$};
    \node (A2_2) at (2, 2) {$A$};
    
    \node (A1_1) [rotate=225] at (0.22, 1) {$\Longrightarrow$};
    \node (A1_2) at (1.08, 1) {$\Omega:=\left[C,\id_C,\id_C,i_{(\id_C\circ\id_C)\circ\id_C},
      \omega\ast i_{\id_C}\right]$};
    
    \path (A0_0) edge [->]node [auto,swap] {$\scriptstyle{(C,\id_C,p^2)}$} (A2_0);
    \path (A0_0) edge [->]node [auto] {$\scriptstyle{(C,\id_C,p^1)}$} (A0_2);
    \path (A0_2) edge [->]node [auto] {$\scriptstyle{(B^1,\id_{B^1},f^1)}$} (A2_2);
    \path (A2_0) edge [->]node [auto,swap] {$\scriptstyle{(B^2,\id_{B^2},f^2)}$} (A2_2);
  \end{tikzpicture}
  \end{equation}
  is a weak fiber product in $\CATC\left[\SETWinv\right]$.
\end{enumerate}
\end{theo}

Moreover, given any pair of morphisms $\operatorname{w}^1,\operatorname{w}^2$ in $\SETW$ as above,
a weak fiber product for \eqref{eq-100} can be obtained easily as a suitable modification of
diagram \eqref{eq-91} (we refer to Corollary~\ref{cor-03} for details).\\

In addition, we have the following result, where the $2$-morphisms $\theta_{\bullet}$ are the
associators of the bicategory $\CATC$ (they are all trivial if $\CATC$ is a $2$-category).

\begin{theo}\label{theo-04}
Let us fix any bicategory $\CATC$ and any class $\SETW$ of morphisms in it, satisfying axioms
\emphatic{(\hyperref[BF]{BF})}, and let us choose any bicategory of fractions
$\CATC\left[\SETWinv\right]$ associated to the
pair $(\CATC,\SETW)$. Moreover, let us fix any set of data in $\CATC$ as in the following diagram

\begin{equation}\label{eq-106}
\begin{tikzpicture}[xscale=2,yscale=-0.8]
    \node (A0_0) at (0, 0) {$C$};
    \node (A0_2) at (2, 0) {$B^1$};
    \node (A2_0) at (0, 2) {$B^2$};
    \node (A2_2) at (2, 2) {$A$};
    
    \node (A1_1) [rotate=225] at (0.9, 1) {$\Longrightarrow$};
    \node (A1_2) at (1.2, 1) {$\omega$};

    \path (A0_0) edge [->]node [auto,swap] {$\scriptstyle{p^2}$} (A2_0);
    \path (A0_0) edge [->]node [auto] {$\scriptstyle{p^1}$} (A0_2);
    \path (A0_2) edge [->]node [auto] {$\scriptstyle{f^1}$} (A2_2);
    \path (A2_0) edge [->]node [auto,swap] {$\scriptstyle{f^2}$} (A2_2);
\end{tikzpicture}
\end{equation}

Then the induced diagram \eqref{eq-91} is a weak fiber product in $\CATC\left[\SETWinv\right]$
if and only if the following $3$ conditions hold for each object $D$ of $\CATC$:
 
\begin{enumerate}[\emphatic{(}a\emphatic{)}]
 \item\label{a} given any pair of morphisms $q^m:D\rightarrow B^m$ for $m=1,2$
   and any invertible $2$-morphism $\lambda:f^1\circ q^1\Rightarrow f^2\circ q^2$ in $\CATC$,
   there are an object $E$, a morphism $\operatorname{v}:E\rightarrow D$ in $\SETW$, a morphism
   $q:E\rightarrow C$ and a pair of invertible $2$-morphisms $\lambda^m:q^m\circ\operatorname{v}
   \Rightarrow p^m\circ q$ for $m=1,2$ in $\CATC$, such that:

   \begin{gather}
   \nonumber \thetab{f^2}{p^2}{q}\odot\Big(\omega\ast i_{q}\Big)\odot\thetaa{f^1}{p^1}{q}\odot\Big(
    i_{f^1}\ast\lambda^1\Big)= \\
%%%
   \label{eq-53} =\Big(i_{f^2}\ast\lambda^2\Big)\odot\thetab{f^2}{q^2}{\operatorname{v}}\odot\Big(
    \lambda\ast i_{\operatorname{v}}\Big)\odot\thetaa{f^1}{q^1}{\operatorname{v}}; 
   \end{gather}
   
  \item\label{b} given any pair of morphisms $t,t':D\rightarrow C$ and any
   pair of invertible $2$-morphisms $\gamma^m:p^m\circ t\Rightarrow p^m\circ t'$ for $m=1,2$ in
   $\CATC$ such that
   
   \begin{gather}
   \nonumber \thetab{f^2}{p^2}{t'}\odot\Big(\omega\ast i_{t'}\Big)\odot\thetaa{f^1}{p^1}{t'}\odot
    \Big(i_{f^1}\ast\gamma^1\Big)=  \\
%%%
   \label{eq-13} =\Big(i_{f^2}\ast\gamma^2\Big)\odot\thetab{f^2}{p^2}{t}\odot\Big(\omega\ast i_t
    \Big)\odot\thetaa{f^1}{p^1}{t},  
   \end{gather}
   there are an object $F$, a morphism $\operatorname{u}:F\rightarrow D$ in $\SETW$ and an
   invertible $2$-morphism $\gamma:t\circ\operatorname{u}\Rightarrow t'\circ\operatorname{u}$ in
   $\CATC$, such that
   
   \begin{equation}\label{eq-15}
   \thetaa{p^m}{t'}{\operatorname{u}}\odot\Big(i_{p^m}\ast\gamma\Big)=\Big(\gamma^m\ast
   i_{\operatorname{u}}\Big)\odot\thetaa{p^m}{t}{\operatorname{u}}\quad\textrm{for }m=1,2;
   \end{equation}

  \item\label{c} given any set of data $(t,t',\gamma^1,\gamma^2,F,\operatorname{u},\gamma)$
   as in \emphatic{(}b\emphatic{)},
   if there is another choice of data $\widetilde{F}$, $\widetilde{\operatorname{u}}:\widetilde{F}
   \rightarrow D$ in $\SETW$ and $\widetilde{\gamma}:t\circ\widetilde{\operatorname{u}}\Rightarrow t'
   \circ\widetilde{\operatorname{u}}$ invertible, such that
   
   \begin{equation}\label{eq-34}
   \thetaa{p^m}{t'}{\widetilde{\operatorname{u}}}\odot\Big(i_{p^m}\ast\widetilde{\gamma}\Big)=\Big(
   \gamma^m\ast i_{\widetilde{\operatorname{u}}}\Big)\odot
   \thetaa{p^m}{t}{\widetilde{\operatorname{u}}}\quad\textrm{for }m=1,2,
   \end{equation}
   then there are an object $G$, a morphisms $\operatorname{z}:G\rightarrow F$ in $\SETW$, a morphism
   $\widetilde{\operatorname{z}}:G\rightarrow\widetilde{F}$ and an invertible $2$-morphism $\mu:
   \operatorname{u}\circ\operatorname{z}\Rightarrow\widetilde{\operatorname{u}}\circ
   \widetilde{\operatorname{z}}$, such that
   
   \begin{gather}
   \nonumber \thetaa{t'}{\widetilde{\operatorname{u}}}{\widetilde{\operatorname{z}}}\odot\Big(
    i_{t'}\ast\mu\Big)\odot\thetab{t'}{\operatorname{u}}{\operatorname{z}}\odot
    \Big(\gamma\ast i_{\operatorname{z}}\Big)= \\
%%%
   \label{eq-32} =\Big(\widetilde{\gamma}\ast i_{\widetilde{\operatorname{z}}}\Big)\odot\thetaa{t}
    {\widetilde{\operatorname{u}}}{\widetilde{\operatorname{z}}}\odot\Big(i_t\ast\mu
    \Big)\odot\thetab{t}{\operatorname{u}}{\operatorname{z}}.  
   \end{gather}
 \end{enumerate}
\end{theo}

As a consequence of Theorem~\ref{theo-04}, we have the following general principle. Suppose that
we are working in a given bicategory $\CATC$ and that for some reason not all the weak fiber
products exist in $\CATC$, or that not all the ``interesting'' fiber products exist there
(for example, the pullbacks along a certain class of ``good'' maps, etc). Then a possible way
to try to solve this problem is the following:

\begin{enumerate}[(1)]
 \item for each given pair of morphisms $f^m:B^m\rightarrow A$ for $m=1,2$ (or for each given pair
  $(f^1,f^2)$ that is ``interesting'' as above), try to identify
  a ``candidate'' for a weak fiber product in $\CATC$, i.e.\ a quadruple $(C,p^1,p^2,\omega)$ as
  in \eqref{eq-106};
 \item given any data as in (1) and any set of data $(D,q^1,q^2,\lambda,t,t',\gamma^1,\gamma^2)$ as
  in (a) and (b) above,
  try to find a set of data $(E,\operatorname{v},q,\lambda^1,\lambda^2,F,\operatorname{u},\gamma)$
  as in (a) and (b), with the only difference that we don't impose that 
  $\operatorname{v}$ and $\operatorname{u}$ belong to some class $\SETW$ (since for the moment there
  is no such class);
 \item try to identify a class $\SETW$ of morphisms in $\CATC$, such that:
  \begin{itemize}
   \item $\SETW$ contains all the morphisms $\operatorname{v}$ and
    $\operatorname{u}$ obtained from the previous procedure, for any set of data $(A,B^1,B^2,f^1,f^2)$
    as in (1) and for any $(D,q^1,q^2,\lambda,t,t',$ $\gamma^1,\gamma^2)$ as in (2),
   \item $\SETW$ satisfies conditions (\hyperref[BF]{BF}) for a bicategory of fractions;
  \end{itemize}
 \item verify if for any data as in (1), condition (c) holds with the associated ``candidate''
  $(C,p^1,p^2,\omega)$ (with the class $\SETW$ constructed in (3));
 \item if you are successful at each of the previous steps, this means that each pair of morphisms
  $(f^1,f^2)$ (or each ``interesting'' pair of morphisms $(f^1,f^2)$) has a weak fiber product
  if considered in the right bicategory of fractions $\CATC\left[\SETWinv\right]$.
\end{enumerate}

In other terms, if you are lucky then
you are able to construct the desired weak fiber products, provided that you allow some
morphisms of $\CATC$ to become internal equivalences. Note however that in general
there is no guarantee that the bicategory $\CATC\left[\SETWinv\right]$ obtained in this way
is ``interesting''. For example, if we have already managed to solve problem (1) and (2), but
a choice for $\SETW$ as in (3) is given by the entire class of morphisms of $\CATC$, then in the
bicategory of fractions obtained in this way all the morphisms are internal
equivalences; so in certain frameworks this procedure could lead to a bicategory that is not
useful or interesting to work with.\\

As a consequence of Theorem~\ref{theo-04}, we are also able to prove:

\begin{cor}\label{cor-01}
Let us fix any pair $(\CATC,\SETW)$ satisfying axioms \emphatic{(\hyperref[BF]{BF})}
and let us choose any bicategory of fractions $\CATC\left[\SETWinv\right]$
associated to the pair $(\CATC,\SETW)$. Let us fix any pair of morphisms $f^1:B^1\rightarrow A$ and
$f^2:B^2\rightarrow A$. Let us suppose that there is a set of data $(C,p^1,p^2,\omega)$ such
that \eqref{eq-106} is a weak fiber product in $\CATC$. Then conditions
\emphatic{(}\hyperref[a]{a}\emphatic{)}, \emphatic{(}\hyperref[b]{b}\emphatic{)}
and \emphatic{(}\hyperref[c]{c}\emphatic{)} above are satisfied. Therefore,
for each pair of morphisms in $\SETW$ of the form
$\operatorname{w}^1:B^1\rightarrow\overline{B}^1$ and $\operatorname{w}^2:B^2\rightarrow
\overline{B}^2$, there is a weak fiber product in $\CATC\left[\SETWinv\right]$ for the
pair of morphisms $(B^1,\operatorname{w}^1,f^1)$ and $(B^2,\operatorname{w}^2,f^2)$. In particular,
if the bicategory $\CATC$ is closed under weak fiber products, then also the bicategory $\CATC\left[
\SETWinv\right]$ is closed under weak fiber products.
\end{cor}

As a simple application of Theorems~\ref{theo-02} and~\ref{theo-04}, in the last part of this paper
we will examine the particular case when
$\CATC$ is a category (considered as a trivial bicategory) and the pair $(\CATC,\SETW)$ satisfies
conditions (\hyperref[CF]{CF}) for a right calculus of fractions. As we mentioned before, in
this case the pair $(\CATC,\SETW)$ satisfies also conditions (\hyperref[BF]{BF}) for a right
bicalculus of fractions and the right category of fractions associated to $(\CATC,\SETW)$ (considered
as a trivial bicategory) is equivalent to the right bicategory of fractions associated to $(\CATC,\SETW)$.
Moreover in this case weak fiber products are simply (strong) fiber products. Then we will prove the
following result.

\begin{prop}\label{prop-11}
Let us fix any pair $(\CATC,\SETW)$ satisfying axioms \emphatic{(\hyperref[CF]{CF})} for
a right calculus of fractions.
Given any pair of morphisms $f^1:B^1\rightarrow A$ and $f^2:B^2\rightarrow A$ in $\CATC$, the
following facts are equivalent:

\begin{enumerate}[\emphatic{(}i\emphatic{)}]
\setcounter{enumi}{2}
 \item\label{iii} for any pair of morphisms in $\SETW$ of the form $\operatorname{w}^1:B^1\rightarrow
  \overline{B}^1$ and $\operatorname{w}^2:B^2\rightarrow\overline{B}^2$, the pair of morphisms
  
  \begin{equation}\label{eq-116}
  \begin{tikzpicture}[xscale=1.5,yscale=-0.8]
    \node (A0_2) at (2, 0) {$\overline{B}^1$};
    \node (A2_0) at (0, 2) {$\overline{B}^2$};
    \node (A2_2) at (2, 2) {$A$};
    
    \path (A0_2) edge [->]node [auto] {$\scriptstyle{[B^1,\operatorname{w}^1,f^1]}$} (A2_2);
    \path (A2_0) edge [->]node [auto,swap] {$\scriptstyle{[B^2,\operatorname{w}^2,f^2]}$} (A2_2);
  \end{tikzpicture}
  \end{equation}
  admits a \emphatic{(}strong\emphatic{)} fiber product in the right category of fractions
  $\CATC\left[\SETWinv\right]$;
 
 \item\label{iv} there are an object $C$ in $\CATC$ and a pair of morphisms $p^1:C\rightarrow B^1$,
  $p^2:C\rightarrow B^2$, such $f^1\circ p^1=f^2\circ p^2$ and such that the diagram
  
  \begin{equation}\label{eq-83}
  \begin{tikzpicture}[xscale=2.9,yscale=-0.8]
    \node (A0_0) at (0, 0) {$C$};
    \node (A0_2) at (2, 0) {$B^1$};
    \node (A2_0) at (0, 2) {$B^2$};
    \node (A2_2) at (2, 2) {$A$};
    
    \node (A1_2) at (1, 1) {$\curvearrowright$};

    \path (A0_0) edge [->]node [auto,swap] {$\scriptstyle{[C,\id_C,p^2]}$} (A2_0);
    \path (A0_0) edge [->]node [auto] {$\scriptstyle{[C,\id_C,p^1]}$} (A0_2);
    \path (A0_2) edge [->]node [auto] {$\scriptstyle{[B^1,\id_{B^1},f^1]}$} (A2_2);
    \path (A2_0) edge [->]node [auto,swap] {$\scriptstyle{[B^2,\id_{B^2},f^2]}$} (A2_2);
  \end{tikzpicture}
  \end{equation}
  is a \emphatic{(}strong\emphatic{)} fiber product in the right category of
  fractions $\CATC\left[\SETWinv\right]$.
\end{enumerate}

Moreover, given any set of data $(C,p^1:C\rightarrow B^1,p^2:C\rightarrow B^2)$ such that
$f^1\circ p^1=f^2\circ p^2$, diagram \eqref{eq-83}
is a \emphatic{(}strong\emphatic{)} fiber product if and only if the following $2$ conditions hold:

\begin{enumerate}[\emphatic{(}a\emphatic{)}]
\setcounter{enumi}{3}
 \item\label{d} given any object $D$ and any pair of morphisms $q^m:D\rightarrow B^m$ for $m=1,2$,
  such that $f^1\circ q^1=f^2\circ q^2$ in $\CATC$, there are an object $E$, a morphism
  $\operatorname{v}:E\rightarrow D$ in $\SETW$ and a morphism $q:E\rightarrow C$, such that
  $q^m\circ\operatorname{v}=p^m\circ q$ for each $m=1,2$;

 \item\label{e} given any set of data $(D,q^1,q^2,E,\operatorname{v},q)$ as in
  \emphatic{(}d\emphatic{)}, if there is another choice of data $\widetilde{E},
  \widetilde{\operatorname{v}}:\widetilde{E}\rightarrow D$ in $\SETW$ and $\widetilde{q}:
  \widetilde{E}\rightarrow C$, such that $q^m\circ\widetilde{\operatorname{v}}=p^m\circ
  \widetilde{q}$ for each $m=1,2$, then there are an object $F$, a morphism
  $\operatorname{u}:F\rightarrow E$ in $\SETW$ and a morphism $\widetilde{\operatorname{u}}:F
  \rightarrow\widetilde{E}$, such that:
   \begin{itemize}
    \item $\operatorname{v}\circ\operatorname{u}=\widetilde{\operatorname{v}}\circ
     \widetilde{\operatorname{u}}$;
    \item $q\circ\operatorname{u}=\widetilde{q}\circ\widetilde{\operatorname{u}}$.
   \end{itemize}
 \end{enumerate}
\end{prop}

\section{Notations}\label{sec-07}
Through all this paper we will use the \emph{axiom of choice}, that we therefore assume
without further remarks. The reason for this is twofold: first of all, the construction of
bicategories of fractions in~\cite{Pr} in general requires the axiom of choice
(except for some special cases described in~\cite[Corollary~0.6]{T3}); moreover we will use
from time to time the universal property of bicategories of fractions,
that was proved in~\cite[Theorem~21]{Pr} implicitly using that axiom.\\
%In addition, in some proofs (for example the proof of Proposition~\ref{prop-07}), given a 
%(weak) equivalence of bicategories we need to find a quasi-inverse for it, and this is possible
%only if we assume the axiom of choice (so that the weak equivalence becomes a (strong) equivalence
%of bicategories.

We mainly refer to~\cite[\S~1]{PW} and~\cite[\S~1.5]{L} for a general overview on bicategories and
pseudofunctors. Given any bicategory $\CATC$, we denote its objects by $A,B,\ldots$,
its morphisms by $f,g,\cdots$ and its $2$-morphisms by $\alpha,\beta,\cdots$ (we will use $A_{\CATC},
f_{\CATC},\alpha_{\CATC},\cdots$ if we have to recall that they belong to $\CATC$ when we are using
more than one bicategory in the computations). Given any triple of morphisms $f:A\rightarrow B,g:B
\rightarrow C,h:C\rightarrow D$ in $\CATC$, we denote by $\thetaa{h}{g}{f}$ the associator
$h\circ(g\circ f)\Rightarrow(h\circ g)\circ f$ that is part of the structure of the
bicategory $\CATC$; we denote by
$\pi_f:f\circ\id_A\Rightarrow f$ and $\upsilon_f:\id_B\circ f\Rightarrow f$ the right and left unitors
for $\CATC$ relative to any morphism $f$ as above. Given another bicategory $\CATD$, we will denote
by $\Theta_{\bullet},\Pi_{\bullet}$ and $\Upsilon_{\bullet}$ its associators,
right and left unitors respectively. We denote by $\functor{F}=
(\functor{F}_0,\functor{F}_1,\functor{F}_2,$ $\Psi_{\bullet}^{\functor{F}},
\Sigma_{\bullet}^{\functor{F}})$ any pseudofunctor $\CATC\rightarrow\CATD$.
Here for each pair of morphisms $f,g$
as above, $\Psi^{\functor{F}}_{g,f}$ is the associator from $\functor{F}_1(g\circ f)$ to
$\functor{F}_1(g)\circ\functor{F}_1(f)$ and for each object $A$, $\Sigma^{\functor{F}}_A$ is the
unitor from $\functor{F}_1(\id_A)$ to $\id_{\functor{F}_0(A)}$.\\

We recall that a morphism $e:A\rightarrow B$ in a bicategory $\CATC$ is called an \emph{internal
equivalence} (or, simply, an \emph{equivalence}) of $\CATC$
if and only if there exists a triple $(d,\delta,\xi)$, where $d$ is a morphism
from $B$ to $A$ and $\delta:\id_A\Rightarrow d\circ e$ and $\xi:e\circ d
\Rightarrow\id_B$ are invertible $2$-morphisms in $\CATC$ (in the literature sometimes the name
``internal equivalence'' is used for denoting the whole quadruple $(e,d,\delta,\xi)$
instead of the morphism $e$ alone). In particular, also $d$ is an internal equivalence
%(it suffices to consider the triple $(e,\xi^{-1},\delta^{-1})$)
and it is usually called \emph{a
quasi-inverse} (or \emph{pseudo-inverse}) for $e$ (in general, the quasi-inverse of an internal
equivalence is not unique). An \emph{adjoint equivalence} is a quadruple $(e,d,\delta,
\xi)$ as above, such that

\[\upsilon_e\odot\Big(\xi\ast i_e\Big)\odot\thetaa{e}{d}{e}\odot\Big(i_e\ast\delta\Big)
\odot\pi^{-1}_e=i_e\]

and

\[\pi_d\odot\Big(i_d\ast\xi\Big)\odot\thetab{d}{e}{d}
\odot\Big(\delta\ast i_d\Big)\odot\upsilon_d^{-1}=i_d\]
(this more restrictive definition is actually the original definition of internal equivalence used
for example in~\cite[pag.~83]{Mac}). By~\cite[Proposition~1.5.7]{L} a morphism $e$ is (the first
component of) an internal equivalence if and only if it is the first component of a (possibly
different) adjoint equivalence.\\

\section{Weak fiber products in a bicategory}
Let us fix any bicategory $\CATD$ and any diagram in it as follows:

\begin{equation}\label{eq-60}
\begin{tikzpicture}[xscale=1.5,yscale=-0.8]
    \node (A0_0) at (0, 0) {$C$};
    \node (A0_2) at (2, 0) {$B^1$};
    \node (A2_0) at (0, 2) {$B^2$};
    \node (A2_2) at (2, 2) {$A$};
    
    \node (A1_1) [rotate=225] at (0.85, 1) {$\Longrightarrow$};
    \node (B1_1) at (1.15, 1) {$\Omega$};
    
    \path (A0_0) edge [->]node [auto,swap] {$\scriptstyle{r^2}$} (A2_0);
    \path (A0_0) edge [->]node [auto] {$\scriptstyle{r^1}$} (A0_2);
    \path (A0_2) edge [->]node [auto] {$\scriptstyle{g^1}$} (A2_2);
    \path (A2_0) edge [->]node [auto,swap] {$\scriptstyle{g^2}$} (A2_2);
\end{tikzpicture}
\end{equation}
with $\Omega$ invertible.
Given any object $D$ in $\CATD$, we define a $1$-category $\Iso_{\CATD}(D,C)$ whose objects are all
the $1$-morphisms from $D$ to $C$ in $\CATD$ and whose morphisms are all the invertible $2$-morphisms
between such $1$-morphisms (as such, $\Iso_{\CATD}(D,C)$ is an internal groupoid
in $(\operatorname{Sets})$). Moreover, we define also a
groupoid $\Iso_{\CATD}(D,g^1,g^2)$ as follows: its objects are all the triples $(s^1,s^2,\Lambda)$,
where $s^1:D\rightarrow B^1$, $s^2:D\rightarrow B^2$ are morphisms and $\Lambda$ is any invertible
$2$-morphism from $g^1\circ s^1$ to $g^2\circ s^2$ in $\CATD$.
A morphism from a triple $(s^1,s^2,\Lambda)$ to a
triple $(s^{\prime\, 1},s^{\prime\, 2},\Lambda')$ is any pair $(\Gamma^1,\Gamma^2)$ of invertible
$2$-morphisms $\Gamma^1:s^1\Rightarrow s^{\prime\, 1}$ and $\Gamma^2:s^2\Rightarrow s^{\prime\, 2}$, such
that

\[\Lambda'\odot\Big(i_{g^1}\ast\Gamma^1\Big)=\Big(i_{g^2}\ast\Gamma^2\Big)\odot\Lambda:\,\,g^1\circ
s^1\Longrightarrow g^2\circ s^{\prime\, 2}.\]

Then for each object $D$ in $\CATD$, diagram \eqref{eq-60} induces a functor

\[\functor{F}_D:\Iso_{\CATD}(D,C)\longrightarrow\Iso_{\CATD}(D,g^1,g^2)\]
defined on each object $s:D\rightarrow C$ in $\Iso_{\CATD}(D,C)$ by

\[\functor{F}_D(s):=\Big(r^1\circ s,r^2\circ s,\Thetab{g^2}{r^2}{s}\odot(\Omega\ast i_s)\odot
\Thetaa{g^1}{r^1}{s}\Big)\]
and on each invertible $2$-morphism $\Gamma:s\Rightarrow s'$ (i.e.\ each morphism in
$\Iso_{\CATD}(D,C)$ from $s$ to $s'$) by

\begin{gather*}
\functor{F}_D(\Gamma):=\Big(i_{r^1}\ast\Gamma,i_{r^2}\ast\Gamma\Big):\Big(r^1\circ s,r^2\circ s,
 \Thetab{g^2}{r^2}{s}\odot(\Omega\ast i_s)\odot\Thetaa{g^1}{r^1}{s}\Big)\longrightarrow \\
%%%
\longrightarrow\Big(r^1\circ s',r^2\circ s',\Thetab{g^2}{r^2}{s'}\odot(\Omega\ast i_{s'})\odot
 \Thetaa{g^1}{r^1}{s'}\Big)
\end{gather*}
(a direct check proves that $\functor{F}_D$ is actually a functor). Then one can give the following
definition (see for example~\cite[pag.~125]{MM} in the case when $\CATD$ is a $2$-category).

\begin{defin}\label{def-01}
Let us fix any bicategory $\CATD$ and any diagram as \eqref{eq-60} in it, with $\Omega$ invertible.
We say that such a diagram
\emph{has the universal property of weak fiber products} if the functor $\functor{F}_D$ described
above is an equivalence of categories (actually, of internal groupoids in $(\operatorname{Sets})$)
for each object $D$ in $\CATD$. In
this case, we say also that \eqref{eq-60} is \emph{a weak fiber product} (also called \emph{weak
pullback} or \emph{$2$-fiber product} when $\CATD$ is a $2$-category) of the pair $(g^1,g^2)$.
Equivalently, \eqref{eq-60} is a weak fiber product if and only if the following $2$ conditions hold
for every object $D$:

\begin{description}
 \item[A1$(D)$]\label{A1} \emph{$\functor{F}_D$ is essentially surjective}, i.e.\ for
  any set of data $(s^1,s^2,\Lambda)$ in $\CATD$ with $\Lambda$ invertible as follows
 
  \[
  \begin{tikzpicture}[xscale=1.5,yscale=-0.8]
    \node (A0_0) at (0, 0) {$D$};
    \node (A0_2) at (2, 0) {$B^1$};
    \node (A2_0) at (0, 2) {$B^2$};
    \node (A2_2) at (2, 2) {$A$,};
    
    \node (A1_1) [rotate=225] at (0.9, 1) {$\Longrightarrow$};
    \node (B1_1) at (1.2, 1) {$\Lambda$};
    
    \path (A0_0) edge [->]node [auto,swap] {$\scriptstyle{s^2}$} (A2_0);
    \path (A0_0) edge [->]node [auto] {$\scriptstyle{s^1}$} (A0_2);
    \path (A0_2) edge [->]node [auto] {$\scriptstyle{g^1}$} (A2_2);
    \path (A2_0) edge [->]node [auto,swap] {$\scriptstyle{g^2}$} (A2_2);
  \end{tikzpicture}
  \]
%%%
  there are a morphism $s:D\rightarrow C$ and a pair of invertible $2$-morphisms
  $\Lambda^m:s^m\Rightarrow r^m\circ s$ for $m=1,2$, such that
  
  \begin{equation}\label{eq-42}
  \Big(\Omega\ast i_s\Big)\odot\Thetaa{g^1}{r^1}{s}\odot\Big(i_{g^1}\ast
  \Lambda^1\Big)=\Thetaa{g^2}{r^2}{s}\odot\Big(i_{g^2}\ast\Lambda^2\Big)\odot\Lambda.
  \end{equation}
%%%
  For simplicity of exposition, we write below the $2$ diagrams associated to the left and to the
  right hand side of \eqref{eq-42}:
  
  \[
  \begin{tikzpicture}[xscale=3.2,yscale=-2.0]
    \node (A1_2) at (1.5, 1.5) {$B^1$};
    \node (A2_0) at (0, 2) {$D$};
    \node (A2_1) at (1.5, 2.5) {$C$};
    \node (A2_3) at (3, 2) {$A$,};

    \node (A1_1) at (0.7, 2.35) {$\Downarrow\,i_s$};
    \node (A1_3) at (2.3, 1.65) {$\Downarrow i_{g^1}$};
    \node (A2_2) at (2.3, 2.35) {$\Downarrow\,\Omega$};
    \node (A0_1) at (0.7, 1.65) {$\Downarrow\,\Lambda^1$};
    \node (A0_2) at (1.5, 2) {$\Downarrow\,\Thetaa{g^1}{r^1}{s}$};
    
    \path (A1_2) edge [->,bend right=30]node [auto] {$\scriptstyle{g^1}$} (A2_3);
    \path (A1_2) edge [->,bend left=5]node [auto,swap] {$\scriptstyle{g^1}$} (A2_3);
    \path (A2_0) edge [->,bend left=5]node [auto,swap] {$\scriptstyle{r^1\circ s}$} (A1_2);
    \path (A2_0) edge [->,bend right=30]node [auto] {$\scriptstyle{s^1}$} (A1_2);
    \path (A2_1) edge [->,bend right=5]node [auto] {$\scriptstyle{g^1\circ r^1}$} (A2_3);
    \path (A2_1) edge [->,bend left=30]node [auto,swap] {$\scriptstyle{g^2\circ r^2}$} (A2_3);
    \path (A2_0) edge [->,bend left=30]node [auto,swap] {$\scriptstyle{s}$} (A2_1);
    \path (A2_0) edge [->,bend right=5]node [auto] {$\scriptstyle{s}$} (A2_1);
    \end{tikzpicture}
  \]
  
  \[
  \begin{tikzpicture}[xscale=3.2,yscale=-2.2]
    \node (A1_1) at (1.5, 1.5) {$B^1$};
    \node (A2_0) at (0, 2) {$D$};
    \node (A2_2) at (3, 2) {$A$;};
    \node (A3_1) at (1.5, 2) {$B^2$};
    \node (C3_2) at (1.5, 2.5) {$B^2$};

    \node (A3_0) at (0.75, 2) {$\Downarrow\,\Lambda^2$};
    \node (A2_1) at (1.5, 1.75) {$\Downarrow\,\Lambda$};
    \node (A3_2) at (2.25, 2) {$\Downarrow\,i_{g^2}$};
    \node (C0_0) at (1.5, 2.25) {$\Downarrow\,\Thetaa{g^2}{r^2}{s}$};
    
    \path (A2_0) edge [->,bend right=15]node [auto] {$\scriptstyle{s^1}$} (A1_1);
    \path (A2_0) edge [->,bend right=15]node [auto] {$\scriptstyle{s^2}$} (A3_1);
    \path (A2_0) edge [->,bend left=15]node [auto,swap] {$\scriptstyle{r^2\circ s}$} (A3_1);
    \path (A1_1) edge [->,bend right=15]node [auto] {$\scriptstyle{g^1}$} (A2_2);
    \path (A3_1) edge [->,bend right=15]node [auto] {$\scriptstyle{g^2}$} (A2_2);
    \path (A3_1) edge [->,bend left=15]node [auto,swap] {$\scriptstyle{g^2}$} (A2_2);
    \path (A2_0) edge [->,bend left=15]node [auto,swap] {$\scriptstyle{s}$} (C3_2);
    \path (C3_2) edge [->,bend left=15]node [auto,swap] {$\scriptstyle{g^2\circ r^2}$} (A2_2);
  \end{tikzpicture}
  \]

 \item[A2$(D)$]\label{A2} \emph{$\functor{F}_D$ is fully faithful}, i.e.\ for
  any pair of morphisms $t,t':D\rightarrow C$ and for any
  pair of invertible $2$-morphisms $\Gamma^m:r^m\circ t\Rightarrow r^m\circ t'$ for $m=1,2$, such that

  \begin{gather}
  \nonumber \Thetab{g^2}{r^2}{t'}\odot\Big(\Omega\ast i_{t'}\Big)\odot\Thetaa{g^1}{r^1}{t'}\odot
   \Big(i_{g^1}\ast\Gamma^1\Big)= \\
%%%
  \label{eq-71} =\Big(i_{g^2}\ast\Gamma^2\Big)\odot\Thetab{g^2}{r^2}{t}\odot\Big(\Omega\ast i_t
   \Big)\odot\Thetaa{g^1}{r^1}{t},  
  \end{gather}
%%%
  there is a \emph{unique} invertible $2$-morphism $\Gamma:t\Rightarrow t'$, such that
  $i_{r^m}\ast\Gamma=\Gamma^m$ for each $m=1,2$. The pair of $2$-morphisms in \eqref{eq-71} is given
  as follows:

  
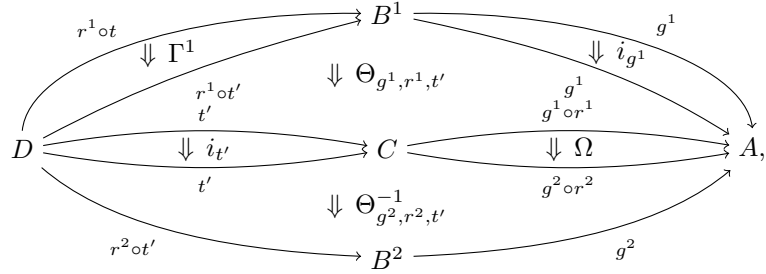
\[
  \begin{tikzpicture}[xscale=3.2,yscale=-1.2]
    \node (A2_0) at (0, 2) {$D$};
    \node (A2_2) at (1.5, 0.5) {$B^1$};
    \node (A2_3) at (3, 2) {$A$,};
    \node (A3_1) at (1.5, 2) {$C$};
    \node (A4_2) at (1.5, 3.2) {$B^2$};

    \node (A1_1) at (0.6, 0.95) {$\Downarrow\,\Gamma^1$};
    \node (A1_2) at (2.45, 0.95) {$\Downarrow\,i_{g^1}$};
    \node (A2_1) at (1.5, 1.2) {$\Downarrow\,\Thetaa{g^1}{r^1}{t'}$};
    \node (A3_2) at (2.25, 2) {$\Downarrow\,\Omega$};
    \node (A4_1) at (0.75, 2) {$\Downarrow\,i_{t'}$};
    \node (A5_1) at (1.5, 2.65) {$\Downarrow\,\Thetab{g^2}{r^2}{t'}$};

    \path (A2_0) edge [->,bend right=25]node [auto] {$\scriptstyle{t'}$} (A3_1);
    \path (A2_0) edge [->,bend left=25]node [auto,swap] {$\scriptstyle{t'}$} (A3_1);
    \path (A2_2) edge [->,bend right=15]node [auto,swap] {$\scriptstyle{g^1}$} (A2_3);
    \path (A2_2) edge [->,bend right=45]node [auto] {$\scriptstyle{g^1}$} (A2_3);
    \path (A3_1) edge [->,bend right=25]node [auto] {$\scriptstyle{g^1\circ r^1}$} (A2_3);
    \path (A3_1) edge [->,bend left=25]node [auto,swap] {$\scriptstyle{g^2\circ r^2}$} (A2_3);
    \path (A2_0) edge [->,bend left=30]node [auto,swap] {$\scriptstyle{r^2\circ t'}$} (A4_2);
    \path (A4_2) edge [->,bend left=30]node [auto,swap] {$\scriptstyle{g^2}$} (A2_3);
    \path (A2_0) edge [->,bend right=10]node [auto,swap] {$\scriptstyle{r^1\circ t'}$} (A2_2);
    \path (A2_0) edge [->,bend right=45]node [auto] {$\scriptstyle{r^1\circ t}$} (A2_2);
  \end{tikzpicture}
  \]
  
  
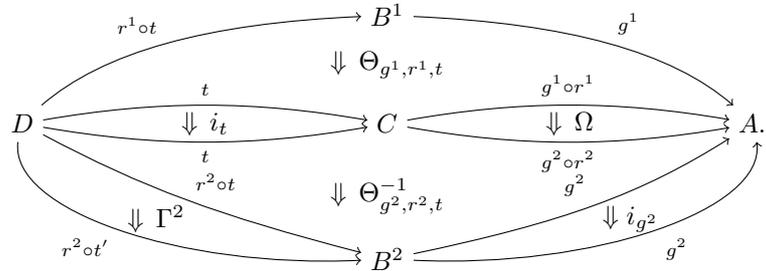
\[
  \begin{tikzpicture}[xscale=3.2,yscale=1.2]
    \node (A2_0) at (0, 2) {$D$};
    \node (A2_2) at (1.5, 0.5) {$B^2$};
    \node (A2_3) at (3, 2) {$A$.};
    \node (A3_1) at (1.5, 2) {$C$};
    \node (A4_2) at (1.5, 3.2) {$B^1$};

    \node (A1_1) at (0.55, 0.95) {$\Downarrow\,\Gamma^2$};
    \node (A1_2) at (2.5, 0.95) {$\Downarrow i_{g^2}$};
    \node (A2_1) at (1.5, 1.2) {$\Downarrow\,\Thetab{g^2}{r^2}{t}$};
    \node (A3_2) at (2.25, 2) {$\Downarrow\,\Omega$};
    \node (A4_1) at (0.75, 2) {$\Downarrow\,i_t$};
    \node (A5_1) at (1.5, 2.65) {$\Downarrow\,\Thetaa{g^1}{r^1}{t}$};

    \path (A2_0) edge [->,bend right=25]node [auto,swap] {$\scriptstyle{t}$} (A3_1);
    \path (A2_0) edge [->,bend left=25]node [auto] {$\scriptstyle{t}$} (A3_1);
    \path (A2_2) edge [->,bend right=15]node [auto] {$\scriptstyle{g^2}$} (A2_3);
    \path (A2_2) edge [->,bend right=50]node [auto,swap] {$\scriptstyle{g^2}$} (A2_3);
    \path (A3_1) edge [->,bend right=25]node [auto,swap] {$\scriptstyle{g^2\circ r^2}$} (A2_3);
    \path (A3_1) edge [->,bend left=25]node [auto] {$\scriptstyle{g^1\circ r^1}$} (A2_3);
    \path (A2_0) edge [->,bend left=30]node [auto] {$\scriptstyle{r^1\circ t}$} (A4_2);
    \path (A4_2) edge [->,bend left=30]node [auto] {$\scriptstyle{g^1}$} (A2_3);
    \path (A2_0) edge [->,bend right=10]node [auto] {$\scriptstyle{r^2\circ t}$} (A2_2);
    \path (A2_0) edge [->,bend right=50]node [auto,swap] {$\scriptstyle{r^2\circ t'}$} (A2_2);
  \end{tikzpicture}
  \]
\end{description}
\end{defin}

\begin{rem}\label{rem-02}
Equivalently, \eqref{eq-60} is a weak fiber product in the bicategory $\CATD$ if and only if the
following conditions are satisfied:

\begin{itemize}
 \item for each triple $(D,s^1:D\rightarrow B^1,s^2:D\rightarrow B^2)$ in $\CATD$ the
  following property holds:
  
  \begin{description}
   \item[B1$(D,s^1,s^2)$]\label{B1}
    for any invertible $2$-morphism $\Lambda:g^1\circ s^1\Rightarrow g^2\circ s^2$,
    there are a morphism $s:D\rightarrow C$ and a pair of invertible $2$-morphisms $\Lambda^m:s^m
    \Rightarrow r^m\circ s$ for $m=1,2$, such that \eqref{eq-42} holds;
  \end{description}

 \item for each triple $(D,t:D\rightarrow C,t':D\rightarrow C)$ in $\CATD$ the following
  property holds:
  \begin{description}
   \item[B2$(D,t,t')$]\label{B2}
    for any pair of invertible $2$-morphisms $\Gamma^m:r^m\circ t\Rightarrow r^m
    \circ t'$ for $m=1,2$, such that \eqref{eq-71} holds, there is a unique invertible
    $2$-morphism $\Gamma:t\Rightarrow t'$, such that $i_{r^m}\ast\Gamma=\Gamma^m$ for each $m=1,2$.
  \end{description}
\end{itemize}

As we will see in Proposition~\ref{prop-03} and~\ref{prop-04}, in general it is sufficient
to verify that condition \textbf{\hyperref[B1]{B1}},
respectively \textbf{\hyperref[B2]{B2}}, holds for a (smaller) subset of triples $(D,s^1,s^2)$,
respectively $(D,t,t')$.
\end{rem}

\begin{rem}\label{rem-01}
Given any category $\CATD$, we denote by $\CATD^2$ the trivial bicategory obtained from
$\CATD$, i.e.\ the bicategory whose objects and morphisms are the same as those of $\CATD$ and whose
$2$-morphisms are only the $2$-identities. Then it is easy to see that a $2$-commutative square in
$\CATD^2$ is a weak fiber product if and only if the same square is a (strong) fiber product in
$\CATD$. In other terms, weak fiber products generalize the notion of (strong) fiber products from
categories to $2$-categories.
\end{rem}

In the remaining part of this section we are going to state some useful results
about weak fiber products in any bicategory $\CATD$. All such lemmas will play a crucial role
when $\CATD$ will be a bicategory of fractions $\CATC\left[\SETWinv\right]$.

\begin{prop}\label{prop-01}
Let us suppose that \eqref{eq-60} \emphatic{(}with $\Omega$ invertible\emphatic{)}
is a weak fiber product in a bicategory $\CATD$.
Moreover, let us also fix any set of objects, morphisms and $2$-morphism as follows for each $m=1,2$:

\begin{gather}
\nonumber e^m:\,\overline{B}^m\longrightarrow B^m,\quad\quad d^m:B^m\longrightarrow\overline{B}^m, \\
%%%
\label{eq-102}\Delta^m:\,\id_{\overline{B}^m}\Longrightarrow d^m\circ e^m,\quad\quad\Xi^m:\,e^m\circ
 d^m\Longrightarrow\id_{B^m},
\end{gather}
such that the quadruple $(e^m,d^m,\Delta^m,\Xi^m)$ is an \emph{adjoint}
equivalence in $\CATD$ for each $m=1,2$. Moreover, let us define

\begin{gather}
\nonumber \overline{\Omega}:=\Thetab{g^2\circ e^2}{d^2}{r^2}\odot\Big(\Thetaa{g^2}{e^2}{d^2}
 \ast i_{r^2}\Big)\odot\Big(\Big(i_{g^2}\ast\left(\Xi^2\right)^{-1}\Big)\ast i_{r^2}\Big)\odot
 \Big(\Pi_{g^2}^{-1}\ast i_{r^2}\Big)\odot \\
\nonumber\odot\,\Omega\odot\Big(\Pi_{g^1}\ast i_{r^1}\Big)\odot\Big(
 \Big(i_{g^1}\ast\Xi^1\Big)\ast i_{r^1}\Big)
 \odot\Big(\Thetab{g^1}{e^1}{d^1}\ast i_{r^1}\Big)\odot\Thetaa{g^1\circ e^1}{d^1}{r^1}: \\
%%%
\label{eq-36} \phantom{\Big(}(g^1\circ e^1)\circ(d^1\circ r^1)\Longrightarrow
 (g^2\circ e^2)\circ(d^2\circ r^2)
\end{gather}
\emphatic{(}where $\Theta_{\bullet}$ and $\Pi_{\bullet}$ are the associators and right unitors
for $\CATD$\emphatic{)}. Then the diagram

\begin{equation}\label{eq-02}
\begin{tikzpicture}[xscale=1.5,yscale=-0.8]
    \node (A0_0) at (0, 0) {$C$};
    \node (A0_2) at (2, 0) {$\overline{B}^1$};
    \node (A2_0) at (0, 2) {$\overline{B}^2$};
    \node (A2_2) at (2, 2) {$A$};
    
    \node (A1_1) [rotate=225] at (0.85, 1) {$\Longrightarrow$};
    \node (A1_2) at (1.15, 1) {$\overline{\Omega}$};
    
    \path (A0_0) edge [->]node [auto,swap] {$\scriptstyle{d^2\circ r^2}$} (A2_0);
    \path (A0_0) edge [->]node [auto] {$\scriptstyle{d^1\circ r^1}$} (A0_2);
    \path (A0_2) edge [->]node [auto] {$\scriptstyle{g^1\circ e^1}$} (A2_2);
    \path (A2_0) edge [->]node [auto,swap] {$\scriptstyle{g^2\circ e^2}$} (A2_2);
\end{tikzpicture}
\end{equation}
is a weak fiber product in $\CATD$.
\end{prop}

Since each internal equivalence is the first component of an adjoint equivalence
(see~\cite[Proposition~1.5.7]{L}), then this result implies at once that:

\begin{cor}\label{cor-04}
Let us fix any pair of morphisms $g^m:B^m\rightarrow A$ for $m=1,2$ that admit a weak fiber product
in a bicategory $\CATD$; then for every pair of internal
equivalences $e^m:\overline{B}^m\rightarrow B^m$ for $m=1,2$, the morphisms $g^m\circ e^m:
\overline{B}^m\rightarrow A$ for $m=1,2$ have a weak fiber product in $\CATD$.
\end{cor}

\begin{proof}[Proof of Proposition~\ref{prop-01}.]
\emph{For simplicity of exposition, we will
give a complete proof only in the case when $\CATD$ is a $2$-category. In the general case, one has
to add associators and unitors of $\CATD$ and use the coherence conditions on the bicategory $\CATD$
wherever it is necessary. Apart from that, the proofs are exactly the same.}\\

Since the quadruple
$(e^m,d^m,\Delta^m,$ $\Xi^m)$ is an adjoint equivalence, then for each $m=1,2$ we have:

\begin{equation}\label{eq-07}
\Big(\Xi^m\ast i_{e^m}\Big)\odot\Big(i_{e^m}\ast\Delta^m\Big)=i_{e^m}\quad\textrm{and}
\quad\Big(i_{d^m}\ast\Xi^m\Big)\odot\Big(\Delta^m\ast i_{d^m}\Big)=i_{d^m}.
\end{equation}

Let us fix any object $\overline{D}$ in $\CATD$ and 
let us prove property \textbf{\hyperref[A1]{A1}}$(\overline{D})$
for diagram \eqref{eq-02}, so let us fix
any set of data $(\overline{s}^1,\overline{s}^2,\overline{\Lambda})$ in $\CATD$ as
follows, with $\overline{\Lambda}$ invertible

\begin{equation}\label{eq-90}
\begin{tikzpicture}[xscale=1.5,yscale=-0.8]
    \node (A0_0) at (0, 0) {$\overline{D}$};
    \node (A0_2) at (2, 0) {$\overline{B}^1$};
    \node (A2_0) at (0, 2) {$\overline{B}^2$};
    \node (A2_2) at (2, 2) {$A$.};
    
    \node (A1_1) [rotate=225] at (0.9, 1) {$\Longrightarrow$};
    \node (A1_2) at (1.2, 1) {$\overline{\Lambda}$};
    
    \path (A2_0) edge [->]node [auto,swap] {$\scriptstyle{g^2\circ e^2}$} (A2_2);
    \path (A0_0) edge [->]node [auto,swap] {$\scriptstyle{\overline{s}^2}$} (A2_0);
    \path (A0_2) edge [->]node [auto] {$\scriptstyle{g^1\circ e^1}$} (A2_2);
    \path (A0_0) edge [->]node [auto] {$\scriptstyle{\overline{s}^1}$} (A0_2);
\end{tikzpicture}
\end{equation}

Since $\CATD$ is a $2$-category, 
we can consider $\overline{\Lambda}$ as defined from $g^1\circ(e^1\circ\overline{s}^1)$
to $g^2\circ(e^2\circ\overline{s}^2)$. Using property \textbf{\hyperref[A1]{A1}}$(\overline{D})$
for diagram \eqref{eq-60},
there are a morphism $\overline{s}:\overline{D}\rightarrow C$ and a pair of
invertible $2$-morphisms $\Lambda^m:e^m\circ\overline{s}^m\Rightarrow r^m\circ\overline{s}$ for $m=1,
2$, such that

\begin{equation}\label{eq-06}
\Big(\Omega\ast i_{\overline{s}}\Big)\odot\Big(i_{g^1}\ast\Lambda^1\Big)=\Big(i_{g^2}\ast\Lambda^2
\Big)\odot\overline{\Lambda}.
\end{equation}

For each $m=1,2$ we define an invertible $2$-morphism

\[\overline{\Lambda}^m:=\Big(i_{d^m}\ast\Lambda^m\Big)\odot\Big(\Delta^m\ast i_{\overline{s}^m}
\Big):\,\,\overline{s}^m\Longrightarrow(d^m\circ r^m)\circ\overline{s}.\]

Then using the definitions of $\overline{\Lambda}^1,\overline{\Lambda}^2$ and $\overline{\Omega}$
(where we omit associators and unitors of $\CATD$ since we are assuming that $\CATD$ is a
$2$-category), we get a series of identities as follows:

\begin{gather}
\nonumber \Big(\overline{\Omega}\ast i_{\overline{s}}\Big)\odot\Big(i_{g^1\circ e^1}\ast
 \overline{\Lambda}^1\Big)= \\
%%%
\nonumber =\Big(i_{g^2}\ast\left(\Xi^2\right)^{-1}\ast i_{r^2\circ\overline{s}}\Big)\odot\Big(
 \Omega\ast i_{\overline{s}}\Big)\odot\Big(i_{g^1}\ast\Xi^1\ast i_{r^1\circ\overline{s}}\Big)\odot \\
\nonumber \odot\Big(i_{g^1\circ e^1\circ d^1}\ast\Lambda^1\Big)\odot\Big(i_{g^1\circ e^1}\ast
 \Delta^1\ast i_{\overline{s}^1}\Big)\stackrel{(\ast)}{=} \\
%%%
\nonumber \stackrel{(\ast)}{=}\Big(i_{g^2}\ast\left(\Xi^2\right)^{-1}\ast i_{r^2\circ\overline{s}}
 \Big)\odot\Big(\Omega\ast i_{\overline{s}}\Big)\odot\Big(i_{g^1}\ast\Lambda^1\Big)\odot \\
\nonumber \odot\Big(i_{g^1}\ast\Xi^1\ast i_{e^1\circ\overline{s}^1}\Big)\odot\Big(i_{g^1\circ e^1}\ast
 \Delta^1\ast i_{\overline{s}^1}\Big)\stackrel{(\ast\ast)}{=} \\
%%%
\nonumber \stackrel{(\ast\ast)}{=}\Big(i_{g^2}\ast\left(\Xi^2\right)^{-1}\ast i_{r^2\circ
 \overline{s}}\Big)\odot\Big(\Omega\ast i_{\overline{s}}\Big)\odot\Big(i_{g^1}\ast\Lambda^1\Big)
 \stackrel{\eqref{eq-06}}{=} \\
%%%
\nonumber \stackrel{\eqref{eq-06}}{=}\Big(i_{g^2}\ast\left(\Xi^2\right)^{-1}\ast i_{r^2\circ
 \overline{s}}\Big)\odot\Big(i_{g^2}\ast\Lambda^2\Big)\odot\overline{\Lambda}\stackrel{(\ast)}{=} \\
%%% 
\nonumber \stackrel{(\ast)}{=}\Big(i_{g^2\circ e^2\circ d^2}\ast\Lambda^2\Big)\odot\Big(i_{g^2}\ast
 \left(\Xi^2\right)^{-1}\ast i_{e^2\circ\overline{s}^2}\Big)\odot\overline{\Lambda}
 \stackrel{(\ast\ast)}{=} \\
%%%
\label{eq-58} \stackrel{(\ast\ast)}{=}\Big(i_{g^2\circ e^2\circ d^2}\ast\Lambda^2\Big)\odot\Big(
 i_{g^2\circ e^2} \ast\Delta^2\ast i_{\overline{s}^2}\Big)\odot\overline{\Lambda}=\Big(
 i_{g^2\circ e^2}\ast\overline{\Lambda}^2\Big)\odot\overline{\Lambda},  
\end{gather}
where the identities of the form $\stackrel{(\ast)}{=}$ are a consequence of the interchange law in
$\CATD$ (see~\cite[Proposition~1.3.5]{B})
and the identities denoted by $\stackrel{(\ast\ast)}{=}$ are obtained using \eqref{eq-07}. Then
identity \eqref{eq-58} proves that diagram \eqref{eq-02} satisfies property
\textbf{\hyperref[A1]{A1}}$(\overline{D})$.\\

Now let us prove also property \textbf{\hyperref[A2]{A2}}$(\overline{D})$
for diagram \eqref{eq-02}, so let us fix
any pair of morphisms $\overline{t},\overline{t}':\overline{D}\rightarrow C$ and any
pair of invertible $2$-morphisms $\overline{\Gamma}^m:(d^m\circ r^m)\circ\overline{t}\Rightarrow(d^m
\circ r^m)\circ\overline{t}'$ for $m=1,2$, such that

\begin{equation}\label{eq-10}
\Big(\overline{\Omega}\ast i_{\overline{t}'}\Big)\odot\Big(i_{g^1\circ e^1}\ast\overline{\Gamma}^1
\Big)=\Big(i_{g^2\circ e^2}\ast\overline{\Gamma}^2\Big)\odot\Big(\overline{\Omega}\ast
i_{\overline{t}}\Big).
\end{equation}

Then for each $m=1,2$ we define an invertible $2$-morphism

\begin{equation}\label{eq-38}
\Gamma^m:=\Big(\Xi^m\ast i_{r^m\circ\overline{t}'}\Big)\odot\Big(i_{e^m}\ast\overline{\Gamma}^m
\Big)\odot\Big(\left(\Xi^m\right)^{-1}\ast i_{r^m\circ\overline{t}}\Big):\,\,r^m\circ\overline{t}
\Longrightarrow r^m\circ\overline{t}'.
\end{equation}

Then using the interchange law on $\CATD$, we have:

\begin{gather}
\nonumber \Big(\Omega\ast i_{\overline{t}'}\Big)\odot\Big(i_{g^1}\ast\Gamma^1\Big)
 \stackrel{\eqref{eq-38}}{=}  \\
%%%
\nonumber \stackrel{\eqref{eq-38}}{=}\Big(\Omega\ast i_{\overline{t}'}\Big)\odot
 \Big(i_{g^1}\ast\Xi^1\ast i_{r^1\circ
 \overline{t}'}\Big)\odot\Big(i_{g^1\circ e^1}\ast\overline{\Gamma}^1\Big)\odot\Big(i_{g^1}\ast
 \left(\Xi^1\right)^{-1}\ast i_{r^1\circ\overline{t}}\Big)\stackrel{\eqref{eq-36}}{=} \\
%%%
\nonumber \stackrel{\eqref{eq-36}}{=}\Big(i_{g^2}\ast\Xi^2\ast i_{r^2\circ\overline{t}'}\Big)
 \odot\Big(\overline{\Omega}\ast
 i_{\overline{t}'}\Big)\odot\Big(i_{g^1\circ e^1}\ast\overline{\Gamma}^1\Big)\odot\Big(i_{g^1}\ast
 \left(\Xi^1\right)^{-1}\ast i_{r^1\circ\overline{t}}\Big)\stackrel{\eqref{eq-10}}{=} \\
%%%
\nonumber \stackrel{\eqref{eq-10}}{=}\Big(i_{g^2}\ast\Xi^2\ast i_{r^2\circ\overline{t}'}\Big)\odot
 \Big(i_{g^2\circ e^2}\ast\overline{\Gamma}^2\Big)\odot\Big(\overline{\Omega}\ast i_{\overline{t}}
 \Big)\odot\Big(i_{g^1}\ast\left(\Xi^1\right)^{-1}\ast i_{r^1\circ\overline{t}}\Big)
 \stackrel{\eqref{eq-36}}{=} \\
%%%
\nonumber \stackrel{\eqref{eq-36}}{=}\Big(i_{g^2}\ast\Xi^2\ast i_{r^2\circ\overline{t}'}\Big)
 \odot\Big(i_{g^2\circ e^2}\ast
 \overline{\Gamma}^2\Big)\odot\Big(i_{g^2}\ast\left(\Xi^2\right)^{-1}\ast i_{r^2\circ\overline{t}}
 \Big)\odot \\
\nonumber \odot\Big(\Omega\ast i_{\overline{t}}\Big)\odot\Big(i_{g^1}\ast\Xi^1\ast i_{r^1\circ
 \overline{t}}\Big)\odot\Big(i_{g^1}\ast\left(\Xi^1\right)^{-1}\ast i_{r^1\circ\overline{t}}\Big)
  \stackrel{\eqref{eq-38}}{=}\\
%%%
\label{eq-11} \stackrel{\eqref{eq-38}}{=}\Big(i_{g^2}\ast\Gamma^2\Big)\odot\Big(\Omega
 \ast i_{\overline{t}}\Big).  
\end{gather}

Since property \textbf{\hyperref[A2]{A2}}$(\overline{D})$ holds for diagram \eqref{eq-60},
then \eqref{eq-11}
implies that there is a unique invertible $2$-morphism $\overline{\Gamma}:\overline{t}\Rightarrow
\overline{t}'$, such that

\begin{equation}\label{eq-12}
i_{r^m}\ast\overline{\Gamma}=\Gamma^m\quad\textrm{for }m=1,2.
\end{equation}

Then for each $m=1,2$, by interchange law we have:

\begin{gather}
\nonumber i_{d^m\circ r^m}\ast\overline{\Gamma}\stackrel{\eqref{eq-12}}{=}i_{d^m}\ast\Gamma^m
 \stackrel{\eqref{eq-38}}{=}\\
%%%
\nonumber \stackrel{\eqref{eq-38}}{=}\Big(i_{d^m}\ast\Xi^m\ast i_{r^m\circ\overline{t}'}\Big)
 \odot\Big(i_{d^m\circ e^m}
 \ast\overline{\Gamma}^m\Big)\odot\Big(i_{d^m}\ast\left(\Xi^m\right)^{-1}\ast i_{r^m\circ
 \overline{t}}\Big)\stackrel{\eqref{eq-07}}{=} \\
\nonumber  \stackrel{\eqref{eq-07}}{=}\Big(\left(\Delta^m\right)^{-1}\ast
 i_{d^m\circ r^m\circ\overline{t}'}\Big)\odot\Big(
 i_{d^m\circ e^m}\ast\overline{\Gamma}^m\Big)\odot\Big(\Delta^m\ast i_{d^m\circ r^m\circ\overline{t}}
 \Big)= \\
%%%
\label{eq-14} =\overline{\Gamma}^m\odot\Big(\left(\Delta^m\right)^{-1}\ast i_{d^m\circ r^m\circ
 \overline{t}}\Big)\odot\Big(\Delta^m\ast i_{d^m\circ r^m\circ\overline{t}}\Big)=
 \overline{\Gamma}^m.  
\end{gather}

In order to conclude the proof, we need only to prove that $\overline{\Gamma}$ is the unique
invertible $2$-morphism from $\overline{t}$ to $\overline{t}'$, such that \eqref{eq-14} holds for
each $m=1,2$. So let us suppose that there is another invertible $2$-morphism $\overline{\Gamma}':
\overline{t}\Rightarrow\overline{t}'$, such that $i_{d^m\circ r^m}\ast\overline{\Gamma}'=
\overline{\Gamma}^m$ for each $m=1,2$. Then using again the interchange law,
for each $m=1,2$ we have:

\begin{gather*}
\Gamma^m\stackrel{\eqref{eq-38}}{=}\Big(\Xi^m\ast i_{r^m\circ\overline{t}'}\Big)\odot
 \Big(i_{e^m}\ast\overline{\Gamma}^m\Big)
 \odot\Big(\left(\Xi^m\right)^{-1}\ast i_{r^m\circ\overline{t}}\Big)= \\
%%%
=\Big(\Xi^m\ast i_{r^m\circ\overline{t}'}\Big)\odot\Big(i_{e^m\circ d^m\circ r^m}\ast
 \overline{\Gamma}'\Big)\odot\Big(\left(\Xi^m\right)^{-1}\ast i_{r^m\circ\overline{t}}\Big)= \\
%%%
=\Big(i_{r^m}\ast\overline{\Gamma}'\Big)\odot\Big(\Xi^m\ast i_{r^m\circ\overline{t}}\Big)\odot
 \Big(\left(\Xi^m\right)^{-1}\ast i_{r^m\circ\overline{t}}\Big)=i_{r^m}\ast\overline{\Gamma}'.  
\end{gather*}

Since $\overline{\Gamma}$ is the unique invertible $2$-morphism from $\overline{t}$ to
$\overline{t}'$ such that \eqref{eq-12} holds, we conclude that
$\overline{\Gamma}=\overline{\Gamma}'$. So we have proved that property
\textbf{\hyperref[A2]{A2}}$(\overline{D})$ holds for diagram \eqref{eq-02}.
\end{proof}

\begin{lem}\label{lem-16}
Let us suppose that \eqref{eq-60} \emphatic{(}with $\Omega$ invertible\emphatic{)} is a weak 
fiber product in a bicategory $\CATD$ and let us fix any internal equivalence
$e:A\rightarrow\overline{A}$ in $\CATD$. Then the induced square

\begin{equation}\label{eq-134}
\begin{tikzpicture}[xscale=3.6,yscale=-0.8]
    \node (A0_0) at (0, 0) {$C$};
    \node (A0_2) at (2, 0) {$B^1$};
    \node (A2_0) at (0, 2) {$B^2$};
    \node (A2_2) at (2, 2) {$\overline{A}$};

    \node (A1_1) [rotate=225] at (0.24, 1) {$\Longrightarrow$};
    \node (A1_2) at (1.1, 1) {$\overline{\Omega}:=\Thetaa{e}{g^2}{r^2}\odot
      \Big(i_e\ast\Omega\Big)\odot\Thetab{e}{g^1}{r^1}$};

     \path (A0_0) edge [->]node [auto,swap] {$\scriptstyle{r^2}$} (A2_0);
    \path (A0_0) edge [->]node [auto] {$\scriptstyle{r^1}$} (A0_2);
    \path (A0_2) edge [->]node [auto] {$\scriptstyle{e\circ g^1}$} (A2_2);
    \path (A2_0) edge [->]node [auto,swap] {$\scriptstyle{e\circ g^2}$} (A2_2);
\end{tikzpicture}
\end{equation}
is a weak fiber product in $\CATD$.
\end{lem}

See Appendix~\ref{sec-05} for a proof.

\begin{lem}\label{lem-03}
Let us fix any diagram as \eqref{eq-60} \emphatic{(}with $\Omega$ invertible\emphatic{)}
in a bicategory $\CATD$. Moreover, let us fix any pair of morphisms
$\overline{g}^1,\overline{g}^2$ and any pair of invertible $2$-morphisms $\Omega^1$ and $\Omega^2$
as follows:

\[
\begin{tikzpicture}[xscale=1.5,yscale=-1.2]
    \node (A0_2) at (2, 0) {$B^1$};
    \node (A2_0) at (0, 2) {$B^2$};
    \node (A2_2) at (2, 2) {$A$.};
    
    \node (B0_0) at (2, 0.9) {$\Rightarrow$};
    \node (C0_0) at (2, 1.2) {$\Omega^1$};
    \node (B2_2) at (1, 2) {$\Downarrow\,\Omega^2$};
    
    \path (A0_2) edge [->,bend left=25]node [auto,swap] {$\scriptstyle{g^1}$} (A2_2);
    \path (A2_0) edge [->,bend right=25]node [auto] {$\scriptstyle{g^2}$} (A2_2);
    \path (A0_2) edge [->,bend right=25]node [auto] {$\scriptstyle{\overline{g}^1}$} (A2_2);
    \path (A2_0) edge [->,bend left=25]node [auto,swap] {$\scriptstyle{\overline{g}^2}$} (A2_2);
\end{tikzpicture}
\]

Then \eqref{eq-60} is a weak fiber product if and only if the following diagram is a weak fiber
product

\begin{equation}\label{eq-21}
\begin{tikzpicture}[xscale=3.8,yscale=-0.8]
    \node (A0_0) at (0, 0) {$C$};
    \node (A0_2) at (2, 0) {$B^1$};
    \node (A2_0) at (0, 2) {$B^2$};
    \node (A2_2) at (2, 2) {$A$.};
    
    \node (A1_1) [rotate=225] at (0.2, 1) {$\Longrightarrow$};
    \node (A1_2) at (1.1, 1) {$\overline{\Omega}:=\Big(\Omega^2\ast
      i_{r^2}\Big)\odot\Omega\odot\Big(
      \left(\Omega^1\right)^{-1}\ast i_{r^1}\Big)$};
    
    \path (A0_0) edge [->]node [auto,swap] {$\scriptstyle{r^2}$} (A2_0);
    \path (A0_0) edge [->]node [auto] {$\scriptstyle{r^1}$} (A0_2);
    \path (A0_2) edge [->]node [auto] {$\scriptstyle{\overline{g}^1}$} (A2_2);
    \path (A2_0) edge [->]node [auto,swap] {$\scriptstyle{\overline{g}^2}$} (A2_2);
\end{tikzpicture}
\end{equation}
\end{lem}

See Appendix~\ref{sec-05} for a proof.

\begin{theo}\label{theo-01}
Let us fix any bicategory $\CATD$, any pair of morphisms $g^1:B^1\rightarrow A$, $g^2:B^2
\rightarrow A$ and any triple of internal equivalences

\[e:A\longrightarrow\overline{A},\quad e^1:\overline{B}^1\longrightarrow B^1\quad\textrm{and}\quad
e^2:\overline{B}^2\longrightarrow B^2.\]

Then the following facts are equivalent:

\begin{enumerate}[\emphatic{(}a\emphatic{)}]
 \item the pair $(g^1,g^2)$ has a weak fiber product;
 \item the pair $(e\circ(g^1\circ e^1),e\circ(g^2\circ e^2))$ has a weak fiber product.
\end{enumerate}

Moreover, if for each $m=1,2$ we fix a triple $(d^m,\Delta^m,\Xi^m)$ such that the quadruple
$(e^m,d^m,\Delta^m,\Xi^m)$ is an adjoint equivalence
%\emphatic{(}the existence of such a triple is a consequence
%of~\cite[Proposition~1.5.7]{L}\emphatic{)}
and if we assume that 
a weak fiber product for \emphatic{(}a\emphatic{)} is given by diagram \eqref{eq-60},
then a weak fiber product for \emphatic{(}b\emphatic{)} is given by the following diagram

\[
\begin{tikzpicture}[xscale=1.5,yscale=-0.8]
    \node (A0_0) at (0, 0) {$C$};
    \node (A0_2) at (2, 0) {$\overline{B}^1$};
    \node (A2_0) at (0, 2) {$\overline{B}^2$};
    \node (A2_2) at (2, 2) {$\overline{A}$,};
    
    \node (A1_1) [rotate=225] at (0.9, 1) {$\Longrightarrow$};
    \node (A1_2) at (1.2, 1) {$\overline{\Omega}$};
 
    \path (A0_0) edge [->]node [auto,swap] {$\scriptstyle{d^2\circ r^2}$} (A2_0);
    \path (A0_0) edge [->]node [auto] {$\scriptstyle{d^1\circ r^1}$} (A0_2);
    \path (A0_2) edge [->]node [auto] {$\scriptstyle{e\circ(g^1\circ e^1)}$} (A2_2);
    \path (A2_0) edge [->]node [auto,swap] {$\scriptstyle{e\circ(g^2\circ e^2)}$} (A2_2);
\end{tikzpicture}
\]
where:

\begin{gather*}
\overline{\Omega}:=\Thetaa{e}{g^2\circ e^2}{d^2\circ r^2}\odot\Big\{i_e\ast\Big[
 \Thetab{g^2\circ e^2}{d^2}{r^2}\odot\Big(\Thetaa{g^2}{e^2}{d^2}\ast i_{r^2}\Big)\odot \\
\odot\Big(\Big(i_{g^2}\ast\left(\Xi^2\right)^{-1}\Big)\ast i_{r^2}\Big)\odot
 \Big(\Pi_{g^2}^{-1}\ast i_{r^2}\Big)\odot\Omega\odot\Big(\Pi_{g^1}\ast i_{r^1}\Big)\odot \\
\odot\Big(\Big(i_{g^1}\ast\Xi^1\Big)\ast i_{r^1}\Big)\odot\Big(\Thetab{g^1}{e^1}{d^1}\ast i_{r^1}
 \Big)\odot\Thetaa{g^1\circ e^1}{d^1}{r^1}\Big]\Big\}\odot
 \Thetab{e}{g^1\circ e^1}{d^1\circ r^1}: \\
%%%
\phantom{(}(e\circ(g^1\circ e^1))\circ(d^1\circ r^1)\Longrightarrow(e\circ(g^2\circ e^2))
 \circ(d^2\circ r^2).
\end{gather*}
\end{theo}

\begin{proof}
The implication $\textrm{(a)}\Rightarrow\textrm{(b)}$ and the last part of the Theorem are given
by Proposition~\ref{prop-01} and Lemma~\ref{lem-16}, so we need only to prove 
$\textrm{(b)}\Rightarrow\textrm{(a)}$.\\

As usual, for simplicity of exposition we assume that $\CATD$ is a $2$-category.
Let us suppose that $e\circ g^1\circ e^1$ and $e\circ g^2\circ e^2$ have a
weak fiber product. Let us fix a pair of triples $(d^m,\Delta^m,\Xi^m)$ for $m=1,2$ as in
\eqref{eq-102}, such that the quadruple $(e^m,d^m,\Delta^m,\Xi^m)$ is an adjoint equivalence for
each $m=1,2$. In particular, both $d^1$ and $d^2$ are internal equivalences. Since
the pair $(e\circ g^1\circ e^1,e\circ g^2\circ e^2)$ has a weak fiber product, then by
Corollary~\ref{cor-04} also the pair of morphisms

\[(\overline{g}^1:=e\circ g^1\circ e^1\circ d^1:B^1\rightarrow\overline{A},\,\,
\overline{g}^2:=e\circ g^2\circ e^2\circ d^2:B^2
\rightarrow\overline{A})\]
has a weak fiber product. Then for each $m=1,2$ we define an invertible $2$-morphism $\Omega^m:=
i_{e\circ g^m}\ast(\Xi^m)^{-1}$ from $e\circ g^m$ to $\overline{g}^m$. Then by Lemma~\ref{lem-03}
we conclude that the pair of morphisms $(e\circ g^1,e\circ g^2)$ has a weak fiber product.\\

Now $e$ is an internal equivalence, so there are an internal equivalence $d:\overline{A}\rightarrow
A$ and an invertible $2$-morphism $\Delta:\id_A\Rightarrow d\circ e$. By Lemma~\ref{lem-16} we get
that the pair of morphisms $(d\circ e\circ g^1,d\circ e\circ g^2)$ has a weak fiber product.
Then by Lemma~\ref{lem-03} applied to the pair of invertible $2$-morphisms
$\Delta\ast i_{g^m}:g^m\Rightarrow d\circ e\circ g^m$ for $m=1,2$, we get that the pair of morphisms
$(g^1,g^2)$ has a weak fiber product, so we have proved that (b) implies (a).
\end{proof}

\begin{lem}\label{lem-04}
Let us suppose that \eqref{eq-60} \emphatic{(}with $\Omega$ invertible\emphatic{)}
is a weak fiber product in a bicategory $\CATD$ and let us suppose that $e:\overline{C}
\rightarrow C$ is an internal equivalence in $\CATD$. Then the induced square

\begin{equation}\label{eq-24}
\begin{tikzpicture}[xscale=3.4,yscale=-0.8]
    \node (A0_0) at (0, 0) {$\overline{C}$};
    \node (A0_2) at (2, 0) {$B^1$};
    \node (A2_0) at (0, 2) {$B^2$};
    \node (A2_2) at (2, 2) {$A$};
    
    \node (A1_1) [rotate=225] at (0.2, 1) {$\Longrightarrow$};
    \node (A1_2) at (1.1, 1) {$\overline{\Omega}:=\Thetab{g^2}{r^2}{e}\odot
      \Big(\Omega\ast i_e\Big)\odot\Thetaa{g^1}{r^1}{e}$};
    
    \path (A0_0) edge [->]node [auto,swap] {$\scriptstyle{r^2\circ e}$} (A2_0);
    \path (A0_0) edge [->]node [auto] {$\scriptstyle{r^1\circ e}$} (A0_2);
    \path (A0_2) edge [->]node [auto] {$\scriptstyle{g^1}$} (A2_2);
    \path (A2_0) edge [->]node [auto,swap] {$\scriptstyle{g^2}$} (A2_2);
\end{tikzpicture}
\end{equation}
is a weak fiber product in $\CATD$.
\end{lem}

See Appendix~\ref{sec-05} for a proof.\\

In Remark~\ref{rem-02} we described a set of conditions equivalent to conditions
\textbf{\hyperref[A1]{A1}} and \textbf{\hyperref[A2]{A2}}. 
In the following $2$ propositions we will show that given
a diagram as \eqref{eq-60}, it is sufficient to verify property \textbf{\hyperref[B1]{B1}} for it on
a (in general smaller) set of triples $(D,s^1,s^2)$; analogously it is sufficient to verify property
\textbf{\hyperref[B2]{B2}} on a (in general smaller) set of triples $(D,t,t')$. This 
will be very useful in order to simplify the computations
when $\CATD$ is a bicategory of fractions $\CATC\left[\SETWinv\right]$.

\begin{prop}\label{prop-03}
Let us fix any diagram as \eqref{eq-60} in a bicategory $\CATD$, with $\Omega$ invertible,
and any triple $(D,s^1:D\rightarrow B^1,s^2:D\rightarrow B^2)$. Moreover, let us fix
any other pair of morphisms $\overline{s}^m:D\rightarrow B^m$ and any
pair of invertible $2$-morphisms $\Omega^m:s^m\Rightarrow\overline{s}^m$ for
$m=1,2$. Then the following facts are equivalent:

\begin{enumerate}[\emphatic{(}a\emphatic{)}]
 \item condition \emphatic{\textbf{\hyperref[B1]{B1}}}$(D,s^1,s^2)$ for \eqref{eq-60} holds;
 \item condition \emphatic{\textbf{\hyperref[B1]{B1}}}$(D,\overline{s}^1,\overline{s}^2)$ for
  \eqref{eq-60} holds.
\end{enumerate}

Moreover, given any object $\overline{D}$ and any internal equivalence $e:\overline{D}
\rightarrow D$, property \emphatic{(}a\emphatic{)} is equivalent to:

\begin{enumerate}[\emphatic{(}a\emphatic{)}]
\setcounter{enumi}{2}
 \item condition \emphatic{\textbf{\hyperref[B1]{B1}}}$(\overline{D},s^1\circ e,s^2\circ e)$ for
  \eqref{eq-60} holds.
\end{enumerate}
\end{prop}

\begin{proof}
As usual, for simplicity of exposition let us suppose that $\CATD$ is a $2$-category.
Let us firstly prove that (a) implies (b), so let us fix any invertible $2$-morphism
$\overline{\Lambda}:g^1\circ\overline{s}^1\Rightarrow g^2\circ\overline{s}^2$. Then we define
an invertible $2$-morphism

\begin{equation}\label{eq-147}
\Lambda:=\Big(i_{g^2}\ast\left(\Omega^2\right)^{-1}\Big)\odot\overline{\Lambda}\odot\Big(i_{g^1}
\ast\Omega^1\Big):\,g^1\circ s^1\Longrightarrow g^2\circ s^2.
\end{equation}

Since we are assuming (a), then there are a morphism $\overline{s}:D\rightarrow C$ and a pair
of invertible $2$-morphisms $\Lambda^m:s^m\Rightarrow r^m\circ\overline{s}$ for $m=1,2$, such that

\begin{equation}\label{eq-148}
\Big(\Omega\ast i_{\overline{s}}\Big)\odot\Big(i_{g^1}\ast\Lambda^1\Big)=\Big(i_{g^2}\ast
\Lambda^2\Big)\odot\Lambda.
\end{equation}

Then for each $m=1,2$ we define $\overline{\Lambda}^m:=\Lambda^m\odot(\Omega^m)^{-1}:\overline{s}^m
\Rightarrow r^m\circ\overline{s}$, so:

\begin{gather*}
\Big(\Omega\ast i_{\overline{s}}\Big)\odot\Big(i_{g^1}\ast\overline{\Lambda}^1\Big)=
 \Big(\Omega\ast i_{\overline{s}}\Big)\odot\Big(i_{g^1}\ast\Lambda^1\Big)\odot
 \Big(i_{g^1}\ast\left(\Omega^1\right)^{-1}\Big)\stackrel{\eqref{eq-148}}{=} \\
%%%
\stackrel{\eqref{eq-148}}{=} \Big(i_{g^2}\ast
\Lambda^2\Big)\odot\Lambda\odot
 \Big(i_{g^1}\ast\left(\Omega^1\right)^{-1}\Big)\stackrel{\eqref{eq-147}}{=}
\Big(i_{g^2}\ast\overline{\Lambda}^2\Big)\odot\overline{\Lambda}.
\end{gather*}

Therefore \textbf{\hyperref[B1]{B1}}$(D,\overline{s}^1,\overline{s}^2)$ holds, i.e.\ (b) is
satisfied.\\

Since $\Omega^1$ and $\Omega^2$ are invertible by hypothesis, then an analogous proof shows that
(b) implies (a).\\

Now let us fix any object $\overline{D}$ and any internal equivalence $e:\overline{D}\rightarrow D$
and let us prove that (a) implies (c).
Since $e$ is an internal equivalence, we choose an internal equivalence
$d:D\rightarrow\overline{D}$ and invertible
$2$-morphisms $\Delta:\id_{\overline{D}}\Rightarrow d\circ e$ and $\Xi:e\circ d\Rightarrow\id_D$,
such that

\begin{equation}\label{eq-107}
\Big(\Xi\ast i_e\Big)\odot\Big(i_e\ast\Delta\Big)=i_e\quad\textrm{and}\quad\Big(i_d\ast\Xi\Big)\odot
\Big(\Delta\ast i_d\Big)=i_d.
\end{equation}

In order to prove that (c) holds, let us fix any
invertible $2$-morphism $\overline{\Lambda}:g^1\circ
s^1\circ e\Rightarrow g^2\circ s^2\circ e$. Then we define an invertible $2$-morphism

\begin{equation}\label{eq-62}
\Lambda:=\Big(i_{g^2\circ s^2}\ast\Xi\Big)\odot\Big(\overline{\Lambda}\ast i_d\Big)\odot\Big(i_{g^1
\circ s^1}\ast\Xi^{-1}\Big):\,g^1\circ s^1\Longrightarrow g^2\circ s^2.
\end{equation}

Since condition \textbf{\hyperref[B1]{B1}}$(D,s^1,s^2)$ holds for \eqref{eq-60}, then there
are a morphism $s:D\rightarrow C$ and a pair of invertible $2$-morphisms $\Lambda^m:s^m
\Rightarrow r^m\circ
s$ for $m=1,2$, such that \eqref{eq-42} holds. Then we set $\overline{s}:=s\circ e:\overline{D}
\rightarrow C$ and

\begin{equation}\label{eq-63}
\overline{\Lambda}^m:=\Lambda^m\ast i_e:\,\,s^m\circ e\Longrightarrow r^m\circ s\circ e=r^m
\circ\overline{s}\quad\textrm{for }m=1,2.
\end{equation}

Then we have:

\begin{gather*}
\Big(\Omega\ast i_{\overline{s}}\Big)\odot\Big(i_{g^1}\ast\overline{\Lambda}^1\Big)
 \stackrel{\eqref{eq-63}}{=}\Big[\Big
 (\Omega\ast i_s\Big)\odot\Big(i_{g^1}\ast\Lambda^1\Big)\Big]\ast i_e\stackrel{\eqref{eq-42}}{=} \\
%%%
\stackrel{\eqref{eq-42}}{=}\Big[\Big(i_{g^2}\ast\Lambda^2\Big)\odot\Lambda\Big]\ast i_e
 \stackrel{\eqref{eq-63}}{=}\Big(i_{g^2}
 \ast\overline{\Lambda}^2\Big)\odot\Big(\Lambda\ast i_e\Big)\stackrel{\eqref{eq-62},\eqref{eq-107}}{=}
 \Big(i_{g^2}\ast\overline{\Lambda}^2\Big)\odot\overline{\Lambda}. 
\end{gather*}

So we have proved that condition
\textbf{\hyperref[B1]{B1}}$(\overline{D},s^1\circ e,s^2\circ e)$ holds for \eqref{eq-60}, i.e.\
we have proved that (c) holds.\\

Conversely, let us suppose that (c) holds and let us prove that (a) holds. Let us choose any
internal equivalence $d$ and any pair of invertible $2$-morphisms $\Delta$ and $\Xi$ as above.
By proceeding as in the proof of (a)$\Rightarrow$(c) already given, we have that (c) implies that
\textbf{\hyperref[B1]{B1}}$(D,s^1\circ e\circ d,s^2\circ e\circ d)$ holds for \eqref{eq-60}.
Then using the equivalence
of (b) with (a) and the pair of invertible $2$-morphisms $\Omega^m:=i_{s^m}\ast\Xi$ for $m=1,2$,
we conclude that \textbf{\hyperref[B1]{B1}}$(D,s^1,s^2)$ holds for \eqref{eq-60}, i.e.\ (a)
is satisfied. This suffices to conclude.
\end{proof}

\begin{prop}\label{prop-04}
Let us fix any diagram as \eqref{eq-60} in a bicategory $\CATD$, with $\Omega$ invertible,
and any triple $(D,t:D\rightarrow C,t':D\rightarrow C)$. Moreover, let us fix another pair
of morphisms $\overline{t},\overline{t}':D\rightarrow C$ and any pair of invertible $2$-morphisms
$\Phi:t\Rightarrow\overline{t}$ and $\Phi':t'\Rightarrow\overline{t}'$. Then the following facts
are equivalent:

\begin{enumerate}[\emphatic{(}a\emphatic{)}]
 \item condition \emphatic{\textbf{\hyperref[B2]{B2}}}$(D,t,t')$ holds for \eqref{eq-60};
 \item condition \emphatic{\textbf{\hyperref[B2]{B2}}}$(D,\overline{t},\overline{t}')$ holds
  for \eqref{eq-60}.
\end{enumerate}

Moreover, given any object $\overline{D}$ and any internal equivalence $e:\overline{D}
\rightarrow D$, property \emphatic{(}a\emphatic{)} is equivalent to:

\begin{enumerate}[\emphatic{(}a\emphatic{)}]
\setcounter{enumi}{2} 
 \item condition \emphatic{\textbf{\hyperref[B2]{B2}}}$(\overline{D},t\circ e,t'\circ e)$ holds
  for \eqref{eq-60}.
\end{enumerate}
\end{prop}

\begin{proof}
Let us firstly prove that (a) implies (b), so let us fix any pair of invertible $2$-morphisms
$\overline{\Gamma}^m:r^m\circ\overline{t}\Rightarrow r^m\circ\overline{t}'$ for $m=1,2$, such
that

\begin{equation}\label{eq-149}
\Big(\Omega\ast i_{\overline{t}'}\Big)\odot\Big(i_{g^1}\ast\overline{\Gamma}^1\Big)=
\Big(i_{g^2}\ast\overline{\Gamma}^2\Big)\odot\Big(\Omega\ast i_{\overline{t}}\Big).
\end{equation}

Then for each $m=1,2$ we define an invertible $2$-morphism

\begin{equation}\label{eq-150}
\Gamma^m:=\Big(i_{r^m}\ast\left(\Phi'\right)^{-1}\Big)\odot\overline{\Gamma}^m\odot\Big(i_{r^m}
\ast\Phi\Big):\,r^m\circ t\Longrightarrow r^m\circ t'.
\end{equation}

Then by interchange law we have:

\begin{gather*}
\Big(\Omega\ast i_{t'}\Big)\odot\Big(i_{g^1}\ast\Gamma^1\Big)\stackrel{\eqref{eq-150}}{=} \\
%%%
\stackrel{\eqref{eq-150}}{=} \Big(\Omega\ast i_{t'}\Big)
 \odot\Big(i_{g^1\circ r^1}\ast\left(\Phi'\right)^{-1}\Big)\odot\Big(i_{g^1}\ast\overline{\Gamma}^1
 \Big)\odot\Big(i_{g^1\circ r^1}\ast\Phi\Big)= \\
%%%
=\Big(i_{g^2\circ r^2}\ast\left(\Phi'\right)^{-1}\Big)\odot\Big(\Omega\ast i_{\overline{t}'}\Big)
 \odot\Big(i_{g^1}\ast\overline{\Gamma}^1\Big)\odot\Big(i_{g^1\circ r^1}\ast\Phi\Big)
 \stackrel{\eqref{eq-149}}{=} \\
%%%
\stackrel{\eqref{eq-149}}{=}\Big(i_{g^2\circ r^2}\ast\left(\Phi'\right)^{-1}\Big)\odot
 \Big(i_{g^2}\ast\overline{\Gamma}^2\Big)\odot\Big(\Omega\ast i_{\overline{t}}\Big)
 \odot\Big(i_{g^1\circ r^1}\ast\Phi\Big)= \\
%%% 
=\Big(i_{g^2\circ r^2}\ast\left(\Phi'\right)^{-1}\Big)\odot\Big(i_{g^2}\ast
 \overline{\Gamma}^2\Big)\odot\Big(i_{g^2\circ r^2}\ast\Phi\Big)\odot\Big(\Omega\ast i_t\Big)
 \stackrel{\eqref{eq-150}}{=} \\
%%%
\stackrel{\eqref{eq-150}}{=}\Big(i_{g^2}\ast\Gamma^2\Big)\odot\Big(\Omega\ast i_t\Big).
\end{gather*}

Since we are assuming (a), then there is a unique invertible $2$-morphism $\Gamma:t\Rightarrow t'$
such that $i_{r^m}\ast\Gamma=\Gamma^m$ for each $m=1,2$. Then we define an invertible $2$-morphism
$\overline{\Gamma}:=\Phi'\odot\Gamma\odot\Phi^{-1}:\overline{t}\Rightarrow\overline{t}'$. Therefore,
for each $m=1,2$ we have:

\begin{gather*}
\overline{\Gamma}^m\stackrel{\eqref{eq-150}}{=}\Big(i_{r^m}\ast\Phi'\Big)\odot\Gamma^m\odot
 \Big(i_{r^m}\ast\Phi^{-1}\Big)= \\
%%%
=\Big(i_{r^m}\ast\Phi'\Big)\odot\Big(i_{r^m}\ast\Gamma\Big)\odot
 \Big(i_{r^m}\ast\Phi^{-1}\Big)= i_{r^m}\ast\overline{\Gamma}.
\end{gather*}

Now let us suppose that there is another invertible $2$-morphism $\overline{\Gamma}':\overline{t}
\Rightarrow\overline{t}'$, such that $\overline{\Gamma}^m=i_{r^m}\ast\overline{\Gamma}'$ for each
$m=1,2$. Then we define an invertible $2$-morphism $\Gamma':=(\Phi')^{-1}\odot\overline{\Gamma}'
\odot\Phi: t\Rightarrow t'$. Therefore, for each $m=1,2$ we have:

\begin{gather*}
\Gamma^m\stackrel{\eqref{eq-150}}{=}\Big(i_{r^m}\ast\left(\Phi'\right)^{-1}\Big)\odot
 \overline{\Gamma}^m\odot\Big(i_{r^m}\ast\Phi\Big)= \\
%%%
=\Big(i_{r^m}\ast\left(\Phi'\right)^{-1}\Big)\odot\Big(i_{r^m}\ast\overline{\Gamma}'\Big)
 \odot\Big(i_{r^m}\ast\Phi\Big)=i_{r^m}\ast\Gamma'.
\end{gather*}

Since $\Gamma$ is the unique invertible $2$-morphism such that $i_{r^m}\ast\Gamma=\Gamma^m$ for
each $m=1,2$, then we get that $\Gamma'=\Gamma$. Therefore, we conclude that $\overline{\Gamma}'=
\overline{\Gamma}$, so we have proved that \textbf{\hyperref[B2]{B2}}$(D,\overline{t},
\overline{t}')$ holds for \eqref{eq-60}, i.e.\ that (b) holds. Since $\Phi$ and $\Phi'$ are
invertible, an analogous proof shows that (b) implies (a).\\

Now let us fix any object $\overline{D}$ and any internal equivalence $e:\overline{D}\rightarrow
D$ and let us prove that (a) implies (c). So let us fix any pair of invertible $2$-morphisms
$\overline{\Gamma}^m:r^m\circ t\circ e\Rightarrow r^m\circ t'\circ e$ for $m=1,2$, such that

\begin{equation}\label{eq-41}
\Big(\Omega\ast i_{t'\circ e}\Big)\odot\Big(i_{g^1}\ast\overline{\Gamma}^1\Big)=\Big(i_{g^2}\ast
\overline{\Gamma}^2\Big)\odot\Big(\Omega\ast i_{t\circ e}\Big).
\end{equation}

Let us choose any internal equivalence $d$ and any pair of invertible $2$-morphisms $\Delta$ and
$\Xi$ as in the proof of Proposition~\ref{prop-03} (in particular, let us assume that \eqref{eq-107}
holds). For each $m=1,2$, let us define an invertible $2$-morphism

\begin{equation}\label{eq-151}
\Gamma^m:=\Big(i_{r^m\circ t'}\ast\Xi\Big)\odot\Big(\overline{\Gamma}^m\ast i_d\Big)\odot\Big(
i_{r^m\circ t}\ast\Xi^{-1}\Big):\,\,r^m\circ t\Longrightarrow r^m\circ t'.
\end{equation}

Then by interchange law we get:

\begin{gather}
\nonumber \Big(\Omega\ast i_{t'}\Big)\odot\Big(i_{g^1}\ast\Gamma^1\Big)\stackrel{\eqref{eq-151}}{=} \\
%%%
\nonumber \stackrel{\eqref{eq-151}}{=}\Big(\Omega\ast i_{t'}\Big)\odot\Big(i_{g^1\circ r^1\circ t'}
 \ast\Xi\Big)\odot\Big(i_{g^1}\ast
 \overline{\Gamma}^1\ast i_d\Big)\odot\Big(i_{g^1\circ r^1\circ t}\ast\Xi^{-1}\Big)= \\
%%%
\nonumber =\Big(i_{g^2\circ r^2\circ t'}\ast\Xi\Big)\odot\Big(\Omega\ast i_{t'\circ e\circ d}\Big)\odot\Big(
 i_{g^1}\ast\overline{\Gamma}^1\ast i_d \Big)\odot\Big(i_{g^1\circ r^1\circ t}\ast\Xi^{-1}\Big)
 \stackrel{\eqref{eq-41}}{=} \\
%%%
\nonumber \stackrel{\eqref{eq-41}}{=}\Big(i_{g^2\circ r^2\circ t'}\ast\Xi\Big)\odot\Big(i_{g^2}\ast
 \overline{\Gamma}^2\ast i_d\Big)
 \odot\Big(\Omega\ast i_{t\circ e\circ d}\Big)\odot\Big(i_{g^1\circ r^1\circ t}\ast\Xi^{-1}\Big)= \\
%%%
\nonumber =\Big(i_{g^2\circ r^2\circ t'}\ast\Xi\Big)\odot\Big(i_{g^2}\ast\overline{\Gamma}^2\ast
 i_d\Big)\odot\Big(i_{g^2\circ r^2\circ t}\ast\Xi^{-1}\Big)\odot\Big(\Omega\ast i_t\Big)
 \stackrel{\eqref{eq-151}}{=} \\
%%%
\label{eq-120} \stackrel{\eqref{eq-151}}{=}\Big(i_{g^2}\ast\Gamma^2\Big)\odot\Big(\Omega\ast i_t
 \Big).  
\end{gather}

By (a), condition \textbf{\hyperref[B2]{B2}}$(D,t,t')$ holds for diagram \eqref{eq-60}, so by
\eqref{eq-120} there
is a unique invertible $2$-morphism $\Gamma:t\Rightarrow t'$, such that $i_{r^m}\ast\Gamma=\Gamma^m$
for each $m=1,2$. We set $\overline{\Gamma}:=\Gamma\ast i_e:t\circ e\Rightarrow t'\circ e$. Then 
using the interchange law, for each $m=1,2$ we have:

\begin{gather*}
i_{r^m}\ast\overline{\Gamma}=i_{r^m}\ast\Gamma\ast i_e=\Gamma^m\ast i_e
 \stackrel{\eqref{eq-151}}{=}\\
%%%
\stackrel{\eqref{eq-151}}{=}\Big(i_{r^m\circ t'}\ast\Xi\ast i_e\Big)\odot\Big(\overline{\Gamma}^m
 \ast i_{d\circ e}\Big)\odot\Big(i_{r^m\circ t}\ast\Xi^{-1}\ast i_e\Big)
 \stackrel{\eqref{eq-107}}{=} \\
%%%
\stackrel{\eqref{eq-107}}{=}\Big(i_{r^m\circ t'\circ e}\ast\Delta^{-1}\Big)\odot
 \Big(\overline{\Gamma}^m\ast i_{d\circ e}\Big)
 \odot\Big(i_{r^m\circ t\circ e}\ast\Delta\Big)=\overline{\Gamma}^m.
\end{gather*}

Then we need only to prove the uniqueness of $\overline{\Gamma}$.
Let us suppose that $\overline{\Gamma}'$ is another invertible $2$-morphism from $t\circ e$ to
$t'\circ e$, such that $i_{r^m}\ast\overline{\Gamma}'=\overline{\Gamma}^m$ for each $m=1,2$. We set

\begin{equation}\label{eq-152}
\Gamma':=\Big(i_{t'}\ast\Xi\Big)\odot\Big(\overline{\Gamma}'\ast i_d\Big)\odot\Big(i_t\ast
\Xi^{-1}\Big):\,\,t\Longrightarrow t'.
\end{equation}

Then for each $m=1,2$ we have:

\begin{gather*}
i_{r^m}\ast\Gamma'\stackrel{\eqref{eq-152}}{=}\Big(i_{r^m\circ t'}\ast\Xi\Big)\odot\Big(i_{r^m}
 \ast\overline{\Gamma}'\ast i_d\Big)\odot\Big(i_{r^m\circ t}\ast\Xi^{-1}\Big)= \\
%%%
=\Big(i_{r^m\circ t'}\ast\Xi\Big)\odot\Big(\overline{\Gamma}^m\ast i_d\Big)\odot\Big(i_{r^m\circ
 t}\ast\Xi^{-1}\Big)\stackrel{\eqref{eq-151}}{=}\Gamma^m.  
\end{gather*}

By uniqueness of $\Gamma$, we get that $\Gamma=\Gamma'$. Therefore,

\[\overline{\Gamma}=\Gamma\ast i_e=\Gamma'\ast i_e\stackrel{\eqref{eq-152}}{=}\Big(i_{t'}\ast
\Xi\ast i_e\Big)\odot\Big(
\overline{\Gamma}'\ast i_{d\circ e}\Big)\odot\Big(i_t\ast\Xi^{-1}\ast i_e\Big)
\stackrel{\eqref{eq-107}}{=}\overline{\Gamma}'.\]

So \textbf{\hyperref[B2]{B2}}$(\overline{D},t\circ e,t'\circ e)$ holds for diagram \eqref{eq-60},
i.e.\ (c) holds. Now let us assume that (c) holds and let us prove that (a) holds. Since $d$ is an
internal equivalence, using the proof that (a) implies (c) we get that 
\textbf{\hyperref[B2]{B2}}$(\overline{D},t\circ e\circ d,t'\circ e\circ d)$ holds for \eqref{eq-60}.
Then using the equivalence of (a) and (b) and the pair of invertible $2$-morphisms
$\Phi:=i_t\ast\Xi$ and $\Phi':=i_{t'}\ast\Xi$, we get that
\textbf{\hyperref[B2]{B2}}$(\overline{D},t,t')$ holds for \eqref{eq-60}, i.e.\ (a) is satisfied.
\end{proof}

\section{Weak fiber products in equivalent bicategories}
Given any pair $(\CATC,\SETW)$ satisfying conditions (\hyperref[BF]{BF}) for a right 
bicalculus of fractions (see Appendix~\ref{sec-03}), in general the associated right bicategory of
fractions is \emph{unique
only up to equivalences of bicategories}, since the ``standard'' construction 
(as described in~\cite[\S~2.2 and 2.3]{Pr}) depends on a set of choices \hyperref[C]{C}$(\SETW)$
involving axioms (\hyperref[BF]{BF}). In the next pages we will need to
do all the computations of weak fiber products in a \emph{chosen} right bicategory of fractions for
the pair $(\CATC,\SETW)$ and then use this result in order to get a similar result
for any other right bicategory of fractions for $(\CATC,\SETW)$.\\

Therefore, the aim of this section is to prove that \emph{weak fiber products are preserved by
equivalences of bicategories}.
We recall from~\cite[(1.33)]{St} that given any pair of bicategories
$\CATA$ and $\CATB$, a pseudofunctor $\functor{F}:\CATA\rightarrow\CATB$ is a \emph{weak equivalence
of bicategories} (also known as \emph{weak biequivalence}) if and only if the following conditions
hold:

\begin{enumerate}[({X}1)]
 \item\label{X1} for each object $A_{\CATA}$, there are an object $A_{\CATB}$ and an internal
  equivalence from $\functor{F}_0(A_{\CATA})$ to $A_{\CATB}$ in $\CATB$;
 \item\label{X2} for each pair of objects $A_{\CATA},B_{\CATA}$, the functor $\functor{F}(A_{\CATA},
  B_{\CATA})$ is an equivalence of categories from $\CATA(A_{\CATA},B_{\CATA})$ to
  $\CATB(\functor{F}_0(A_{\CATA}),\functor{F}_0(B_{\CATA}))$.
\end{enumerate}

\emph{Since we are assuming the axiom of choice, then each weak equivalence of bicategories is a}
(\emph{strong}) \emph{equivalence of bicategories} (see~\cite[\S~1]{PW}), i.e.\ it admits a 
quasi-inverse. Conversely, each strong equivalence of bicategories is also a weak equivalence.
So from now on we will simply write
``equivalence of bicategories'' for any weak, equivalently strong, equivalence of
bicategories.

\begin{lem}\label{lem-14}
Let us suppose that $\functor{F}:\CATA\rightarrow\CATB$ is an equivalence of bicategories;
moreover, let us fix any weak fiber product in $\CATA$ as follows:

\begin{equation}\label{eq-93}
\begin{tikzpicture}[xscale=1.5,yscale=-0.8]
    \node (A0_0) at (0, 0) {$C_{\CATA}$};
    \node (A0_2) at (2, 0) {$B^1_{\CATA}$};
    \node (A2_0) at (0, 2) {$B^2_{\CATA}$};
    \node (A2_2) at (2, 2) {$A_{\CATA}$.};

    \node (A0_1) [rotate=225] at (0.8, 1) {$\Longrightarrow$};
    \node (A1_1) at (1.2, 1) {$\Omega_{\CATA}$};

    \path (A2_0) edge [->]node [auto,swap] {$\scriptstyle{g^2_{\CATA}}$} (A2_2);
    \path (A0_0) edge [->]node [auto,swap] {$\scriptstyle{r^2_{\CATA}}$} (A2_0);
    \path (A0_2) edge [->]node [auto] {$\scriptstyle{g^1_{\CATA}}$} (A2_2);
    \path (A0_0) edge [->]node [auto] {$\scriptstyle{r^1_{\CATA}}$} (A0_2);
\end{tikzpicture}
\end{equation}

Then the induced diagram

\begin{equation}\label{eq-94}
\begin{tikzpicture}[xscale=4,yscale=-0.8]
    \node (A0_0) at (0, 0) {$\functor{F}_0(C_{\CATA})$};
    \node (A0_2) at (2, 0) {$\functor{F}_0(B^1_{\CATA})$};
    \node (A2_0) at (0, 2) {$\functor{F}_0(B^2_{\CATA})$};
    \node (A2_2) at (2, 2) {$\functor{F}_0(A_{\CATA})$};

    \node (A0_1) [rotate=225] at (0.2, 1) {$\Longrightarrow$};
    \node (A1_1) at (1.07, 1) {$\Omega_{\CATB}:=\Psi^{\functor{F}}_{g^2_{\CATA},r^2_{\CATA}}\odot
      \functor{F}_2(\Omega_{\CATA})\odot\Big(\Psi^{\functor{F}}_{g^1_{\CATA},r^1_{\CATA}}\Big)^{-1}$};
    
    \path (A2_0) edge [->]node [auto,swap] {$\scriptstyle{\functor{F}_1(g^2_{\CATA})}$} (A2_2);
    \path (A0_0) edge [->]node [auto,swap] {$\scriptstyle{\functor{F}_1(r^2_{\CATA})}$} (A2_0);
    \path (A0_2) edge [->]node [auto] {$\scriptstyle{\functor{F}_1(g^1_{\CATA})}$} (A2_2);
    \path (A0_0) edge [->]node [auto] {$\scriptstyle{\functor{F}_1(r^1_{\CATA})}$} (A0_2);
\end{tikzpicture}
\end{equation}
is a weak fiber product in $\CATB$ \emphatic{(}here the $2$-morphisms $\Psi^{\functor{F}}_{\bullet}$
are the associators for $\functor{F}$\emphatic{)}.
\end{lem}

See Appendix~\ref{sec-05} for a proof.

\begin{prop}\label{prop-07}
Let us suppose that $\functor{F}:\CATA\rightarrow\CATB$ is an equivalence of bicategories; moreover,
let us fix any triple of objects $A_{\CATA},B^1_{\CATA},B^2_{\CATA}$ and any pair of morphisms
$g^m_{\CATA}:B^m_{\CATA}\rightarrow A_{\CATA}$ for $m=1,2$. Then the pair $(g^1_{\CATA},g^2_{\CATA})$
has a weak fiber product in $\CATA$ if and only if the pair $(\functor{F}_1(g^1_{\CATA}),
\functor{F}_1(g^2_{\CATA}))$ has a weak fiber product in $\CATB$.
\end{prop}

\begin{proof}
One of the $2$ implication is simply Lemma~\ref{lem-14}. So we need only to prove the opposite
implication. So let us suppose that the pair $(\functor{F}_1(g^1_{\CATA}),\functor{F}_1
(g^2_{\CATA}))$ has a weak fiber product in $\CATB$. Since $\functor{F}$ is an equivalence of
bicategories, then by~\cite[\S~2.2]{L}
there are an equivalence of bicategories $\functor{G}:\CATB\rightarrow\CATA$
and a pseudonatural equivalence of pseudofunctors $\mu:\functor{G}\circ\functor{F}\Rightarrow
\id_{\CATA}$. Since $\functor{G}$ is an equivalence of bicategories,
then by Lemma~\ref{lem-14} (with the roles of
$\CATA$ and $\CATB$ reversed, and $\functor{F}$ replaced by $\functor{G}$) we have that the pair
of morphisms:

\[(\functor{G}_1\circ\functor{F}_1(g^1_{\CATA}),\,\,\functor{G}_1\circ\functor{F}_1(g^2_{\CATA}))\]
has a weak fiber product in $\CATA$. Moreover, since $\mu$ is a pseudonatural equivalence of
pseudofunctors, for each $m=1,2$ we have an invertible $2$-morphism in $\CATA$:

\[\mu(g^m_{\CATA}):\,\,\functor{G}_1\circ\functor{F}_1(g^m_{\CATA})\Longrightarrow g^m_{\CATA}.\]

So by Lemma~\ref{lem-03} we conclude that the pair $(g^1_{\CATA},g^2_{\CATA})$ has a weak
fiber product in $\CATA$.
\end{proof}

\section{Weak fiber products in bicategories of fractions}
In this section and in the following ones, $\CATC$ will be a fixed bicategory and $\SETW$ will be a
fixed class of morphisms in it, such that the pair $(\CATC,\SETW)$ satisfies conditions
(\hyperref[BF]{BF}) for a right bicalculus of fractions (see Appendix~\ref{sec-03}).
We recall that the construction of a bicategory of fractions in~\cite{Pr} depends on a set of choices
\hyperref[C]{C}$(\SETW)$ involving axioms (\hyperref[BF]{BF}). For more details on the construction
of bicategories of fractions, we refer directly to Appendix~\ref{sec-03}. Then we are ready to give
the following:

\begin{proof}[Proof of Theorem~\ref{theo-02}.]
First of all, let us assume (ii) and let us prove that (i) holds. By (ii), there is a weak
fiber product in $\CATC\left[\SETWinv\right]$ for the pair of morphisms $(B^m,\id_{B^m},
f^m)$ for $m=1,2$. Now for each $m=1,2$ we consider the morphism $e^m:=(B^m,
\operatorname{w}^m,\id_{B^m}):\overline{B}^m\rightarrow B^m$. By~\cite[Proposition~20]{Pr}
each $e^m$ is an internal equivalence in $\CATC\left[\SETWinv\right]$ because $\operatorname{w}^m$
belongs to $\SETW$. So by
Corollary~\ref{cor-04} for $\CATD:=\CATC\left[\SETWinv\right]$, the morphisms $g^m\circ e^m$ for
$m=1,2$ have a weak fiber product. Now for each $m=1,2$ we consider an invertible $2$-morphism 
in $\CATC\left[\SETWinv\right]$ as follows:

\begin{gather*}
\Omega^m:=\Big[B^m,\id_{B^m},\id_{B^m},\pi_{\operatorname{w}^m\circ\id_{B^m}},
 \pi_{f^m\circ\id_{B^m}}\Big]: \\
%%%
g^m\circ e^m=\Big(B^m,\operatorname{w}^m\circ\id_{B^m},f^m
  \circ\id_{B^m}\Big)\Longrightarrow\Big(B^m,\operatorname{w}^m,f^m\Big)
\end{gather*}
(where the $2$-morphisms $\pi_{\bullet}$ are the right unitors of $\CATC$).
Using Lemma~\ref{lem-03}, we conclude that the pair of morphisms in \eqref{eq-100} has a
weak fiber product, i.e.\ (i) holds.\\

Now let us assume (i) and let us prove that (ii) holds. If we choose $\operatorname{w}^m:=\id_{B^m}$
for each $m=1,2$ and we use the definition of morphisms in a bicategory of fractions, then there 
are $3$ objects $T,T^1,T^2$, a pair of morphisms $\operatorname{v}^m:T^m\rightarrow T$ in $\SETW$
for $m=1,2$, a pair of morphisms $t^m:T^m\rightarrow B^m$ for $m=1,2$ in $\CATC$ and an invertible
$2$-morphism $\Omega^1$ in $\CATC\left[\SETWinv\right]$, such that the following diagram is a
weak fiber product in $\CATC\left[\SETWinv\right]$:

\begin{equation}\label{eq-35}
\begin{tikzpicture}[xscale=1.5,yscale=-0.8]
    \node (A0_0) at (0, 0) {$T$};
    \node (A0_2) at (2, 0) {$B^1$};
    \node (A2_0) at (0, 2) {$B^2$};
    \node (A2_2) at (2, 2) {$A$.};
    
    \node (A1_1) [rotate=225] at (0.8, 1) {$\Longrightarrow$};
    \node (B1_1) at (1.2, 1) {$\Omega^1$};
    
    \path (A2_0) edge [->]node [auto,swap] {$\scriptstyle{(B^2,\id_{B^2},f^2)}$} (A2_2);
    \path (A0_0) edge [->]node [auto,swap]
      {$\scriptstyle{(T^2,\operatorname{v}^2,t^2)}$} (A2_0);
    \path (A0_2) edge [->]node [auto] {$\scriptstyle{(B^1,\id_{B^1},f^1)}$} (A2_2);
    \path (A0_0) edge [->]node [auto]
      {$\scriptstyle{(T^1,\operatorname{v}^1,t^1)}$} (A0_2);
\end{tikzpicture}
\end{equation}

For simplicity of exposition, from now on we assume that $\CATC$ is a $2$-category instead of a
bicategory. Note that even under this restriction, in general $\CATC\left[\SETWinv\right]$ is
only a bicategory, with trivial unitors but possibly non-trivial associators. So we will have
anyway to explicitly write the associators $\Theta_{\bullet}$
for $\CATC\left[\SETWinv\right]$.\\

As we mentioned above, the bicategory $\CATC\left[\SETWinv\right]$ is not unique, but it depends on
a set of choices \hyperref[C]{C}$(\SETW)$ (for any pair of morphisms $(f,\operatorname{v})$
with the same target and
such that $\operatorname{v}$ belongs to $\SETW$). Different sets of choices give equivalent
bicategories. Therefore, the proof that (i) implies (ii) will consist of the following $2$ steps:

\begin{enumerate}[(a)]
 \item first of all, we prove that (ii) holds in any bicategory $\CATC\left[\SETWinv\right]$ obtained
  by fixing a set of choices \hyperref[C]{C}$(\SETW)$ satisfying condition (\hyperref[C3]{C3})
  (see Appendix~\ref{sec-01});
 \item then we use Step (a) in order to prove that (ii) holds for any other set of choices
  \hyperref[C]{C}$\,'(\SETW)$.
\end{enumerate}

So for the moment, we assume that the set of choices \hyperref[C]{C}$(\SETW)$ satisfies
condition (\hyperref[C3]{C3}). In other terms, we assume that the fixed choices in 
\hyperref[C]{C}$(\SETW)$ are trivial for each pair $(\operatorname{v},\operatorname{v})$ (with
$\operatorname{v}$ in $\SETW$).\\

Let us suppose that the fixed choices \hyperref[C]{C}$(\SETW)$ give data as in the
following diagram, with $\operatorname{x}^1$ in $\SETW$ and $\rho$ invertible:

\begin{equation}\label{eq-129}
\begin{tikzpicture}[xscale=1.5,yscale=-0.8]
    \node (A0_1) at (1, 0) {$R$};
    \node (A1_0) at (0, 2) {$T^1$};
    \node (A1_2) at (2, 2) {$T^2$.};
    \node (A2_1) at (1, 2) {$T$};
    
    \node (A1_1) at (1, 1) {$\rho$};
    \node (B1_1) at (1, 1.4) {$\Rightarrow$};
    
    \path (A1_2) edge [->]node [auto] {$\scriptstyle{\operatorname{v}^2}$} (A2_1);
    \path (A0_1) edge [->]node [auto] {$\scriptstyle{\operatorname{x}^2}$} (A1_2);
    \path (A1_0) edge [->]node [auto,swap] {$\scriptstyle{\operatorname{v}^1}$} (A2_1);
    \path (A0_1) edge [->]node [auto,swap] {$\scriptstyle{\operatorname{x}^1}$} (A1_0);
\end{tikzpicture}
\end{equation}

By condition (\hyperref[C3]{C3}) the composition of the following
morphisms in $\CATC\left[\SETWinv\right]$

\[
\begin{tikzpicture}[xscale=1.5,yscale=-1.2]
    \node (A0_0) at (-0.1, 0) {$T^1$};
    \node (A0_1) at (1, 0) {$T^1$};
    \node (A0_2) at (2, 0) {$T$};
    \path (A0_1) edge [->]node [auto,swap] {$\scriptstyle{\id_{T^1}}$} (A0_0);
    \path (A0_1) edge [->]node [auto] {$\scriptstyle{\operatorname{v}^1}$} (A0_2);
    \node (C0_0) at (3, 0) {$\textrm{and}$};
    
    \node (B0_0) at (4, 0) {$T$};
    \node (B0_1) at (5, 0) {$T^1$};
    \node (B0_2) at (6, 0) {$B^1$};
    \path (B0_1) edge [->]node [auto,swap] {$\scriptstyle{\operatorname{v}^1}$} (B0_0);
    \path (B0_1) edge [->]node [auto] {$\scriptstyle{t^1}$} (B0_2);
\end{tikzpicture}
\]
is given by

\[
\begin{tikzpicture}[xscale=1.9,yscale=-1.2]
    \node (A0_0) at (-0.1, 0) {$T^1$};
    \node (A0_1) at (1, 0) {$T^1$};
    \node (A0_2) at (2, 0) {$B^1$;};
    
    \path (A0_1) edge [->]node [auto,swap] {$\scriptstyle{\id_{T^1}}$} (A0_0);
    \path (A0_1) edge [->]node [auto] {$\scriptstyle{t^1}$} (A0_2);
\end{tikzpicture}
\]
by \eqref{eq-129} the composition of

\[
\begin{tikzpicture}[xscale=1.5,yscale=-1.2]
    \node (A0_0) at (-0.1, 0) {$T^1$};
    \node (A0_1) at (1, 0) {$T^1$};
    \node (A0_2) at (2, 0) {$T$};
    \path (A0_1) edge [->]node [auto,swap] {$\scriptstyle{\id_{T^1}}$} (A0_0);
    \path (A0_1) edge [->]node [auto] {$\scriptstyle{\operatorname{v}^1}$} (A0_2);
    \node (C0_0) at (3, 0) {$\textrm{and}$};
    
    \node (B0_0) at (4, 0) {$T$};
    \node (B0_1) at (5, 0) {$T^2$};
    \node (B0_2) at (6, 0) {$B^2$};
    \path (B0_1) edge [->]node [auto,swap] {$\scriptstyle{\operatorname{v}^2}$} (B0_0);
    \path (B0_1) edge [->]node [auto] {$\scriptstyle{t^2}$} (B0_2);
\end{tikzpicture}
\]
is given by

\[
\begin{tikzpicture}[xscale=1.9,yscale=-1.2]
    \node (A0_0) at (-0.1, 0) {$T^1$};
    \node (A0_1) at (1, 0) {$R$};
    \node (A0_2) at (2, 0) {$B^2$.};
    
    \path (A0_1) edge [->]node [auto,swap] {$\scriptstyle{\operatorname{x}^1}$} (A0_0);
    \path (A0_1) edge [->]node [auto] {$\scriptstyle{t^2\circ\operatorname{x}^2}$} (A0_2);
\end{tikzpicture}
\]

Therefore, if we apply Lemma~\ref{lem-04} to $\CATD:=\CATC\left[\SETWinv\right]$, to diagram
\eqref{eq-35} and to the internal equivalence $(T^1,\id_{T^1},
\operatorname{v}^1)$ (see again~\cite[Proposition~20]{Pr}), we get that there is a weak fiber product
in $\CATC\left[\SETWinv\right]$ of the form

\begin{equation}\label{eq-37}
\begin{tikzpicture}[xscale=1.5,yscale=-0.8]
    \node (A0_0) at (0, 0) {$T^1$};
    \node (A0_2) at (2, 0) {$B^1$};
    \node (A2_0) at (0, 2) {$B^2$};
    \node (A2_2) at (2, 2) {$A$.};

    \node (A1_1) [rotate=225] at (0.8, 1) {$\Longrightarrow$};
    \node (B1_1) at (1.2, 1) {$\Omega^2$};
    
    \path (A2_0) edge [->]node [auto,swap] {$\scriptstyle{(B^2,\id_{B^2},f^2)}$} (A2_2);
    \path (A0_0) edge [->]node [auto,swap] {$\scriptstyle{(R,\,
      \operatorname{x}^1,t^2\circ\operatorname{x}^2)}$} (A2_0);
    \path (A0_2) edge [->]node [auto] {$\scriptstyle{(B^1,\id_{B^1},f^1)}$} (A2_2);
    \path (A0_0) edge [->]node [auto] {$\scriptstyle{(T^1,\id_{T^1},t^1)}$} (A0_2);
\end{tikzpicture}
\end{equation}

Now we apply again Lemma~\ref{lem-04} to diagram \eqref{eq-37} and to the internal equivalence
$(R,\id_R,\operatorname{x}^1)$. So using again condition (\hyperref[C3]{C3}),
there is a weak fiber product in $\CATC\left[\SETWinv\right]$ of the form

\begin{equation}\label{eq-40}
\begin{tikzpicture}[xscale=1.8,yscale=-0.8]
    \node (A0_0) at (0, 0) {$R$};
    \node (A0_2) at (2, 0) {$B^1$};
    \node (A2_0) at (0, 2) {$B^2$};
    \node (A2_2) at (2, 2) {$A$.};
    
    \node (A1_1) [rotate=225] at (0.8, 1) {$\Longrightarrow$};
    \node (B1_1) at (1.2, 1) {$\Omega^3$};
    
    \path (A2_0) edge [->]node [auto,swap]{$\scriptstyle{(B^2,\id_{B^2},f^2)}$} (A2_2);
    \path (A0_0) edge [->]node [auto,swap] {$\scriptstyle{(R,\id_R,
      t^2\circ\operatorname{x}^2)}$} (A2_0);
    \path (A0_2) edge [->]node [auto] {$\scriptstyle{(B^1,\id_{B^1},f^1)}$} (A2_2);
    \path (A0_0) edge [->]node [auto] {$\scriptstyle{(R,\id_R,
      t^1\circ\operatorname{x}^1)}$} (A0_2);
\end{tikzpicture}
\end{equation}

Since \eqref{eq-40} is a weak fiber product, then $\Omega^3$ is invertible in
$\CATC\left[\SETWinv\right]$. Its source is the morphism $(R,\id_R,f^1\circ t^1\circ
\operatorname{x}^1)$, while its target is the morphism $(R,\id_R,f^2\circ t^2\circ
\operatorname{x}^2)$. By~\cite[Lemma~6.1]{T3} applied to $\alpha:=i_{\id_R}$ and to $\Omega^3$,
there are an object
$C$, a morphism $\operatorname{z}:C\rightarrow R$ in $\SETW$ and a $2$-morphism

\[\omega:\,f^1\circ t^1\circ\operatorname{x}^1\circ\operatorname{z}\Longrightarrow f^2
\circ t^2\circ\operatorname{x}^2\circ\operatorname{z}\]
in $\CATC$, such that $\Omega^3=[C,\operatorname{z},\operatorname{z},i_{\operatorname{z}},
\omega]$. Using~\cite[Proposition~0.8]{T3}, up to replacing $C$ and $\operatorname{z}$, we can
assume that $\omega$ is invertible in $\CATC$ since $\Omega^3$ is invertible in
$\CATC\left[\SETWinv\right]$. Now by Lemma~\ref{lem-04} applied to the weak
fiber product \eqref{eq-40} and to the internal equivalence $(C,\id_C,\operatorname{z})$
(see again~\cite[Proposition~20]{Pr}), we have
a weak fiber product in $\CATC\left[\SETWinv\right]$ of the form

\begin{equation}\label{eq-156}
\begin{tikzpicture}[xscale=1.8,yscale=-0.8]
    \node (A0_0) at (0, 0) {$C$};
    \node (A0_2) at (2, 0) {$B^1$};
    \node (A2_0) at (0, 2) {$B^2$};
    \node (A2_2) at (2, 2) {$A$,};
    
    \node (A1_1) [rotate=225] at (0.8, 1) {$\Longrightarrow$};
    \node (A1_1) at (1.2, 1) {$\Omega^4$};
    
    \path (A2_0) edge [->]node [auto,swap] {$\scriptstyle{(B^2,\id_{B^2},f^2)}$} (A2_2);
    \path (A0_0) edge [->]node [auto,swap] {$\scriptstyle{(C,\id_C,p^2)}$} (A2_0);
    \path (A0_2) edge [->]node [auto] {$\scriptstyle{(B^1,\id_{B^1},f^1)}$} (A2_2);
    \path (A0_0) edge [->]node [auto] {$\scriptstyle{(C,\id_C,p^1)}$} (A0_2);
\end{tikzpicture}
\end{equation}
where for each $m=1,2$ we set $p^m:=t^m\circ\operatorname{x}^m\circ\operatorname{z}$ and where

\begin{gather*}
\Omega^4:=\Thetab{(B^2,\id_{B^2},f^2)}{(R,\id_R,t^2\circ\operatorname{x}^2)}{(C,\id_C,
 \operatorname{z})}\odot\Big(\Omega^3\ast i_{(C,\id_C,\operatorname{z})}\Big)\odot \\
\odot\,\Thetaa{(B^1,\id_{B^1},f^1)}{(R,\id_R,t^1\circ\operatorname{x}^1)}{(C,\id_C,
 \operatorname{z})}:\,\Big(C,\id_C,f^1\circ p^1\Big)\Longrightarrow\Big(C,\id_C,f^2\circ p^2\Big).  
\end{gather*}

By Lemma~\ref{lem-10}, the associators $\Theta_{\bullet}$ in the previous lines are both
trivial, so

\begin{equation}\label{eq-157}
\Omega^4=\Omega^3\ast i_{(C,\id_C,\operatorname{z})}=
\Big[C,\id_C,\id_C,i_{\id_C},\omega\Big]:\Big(C,\id_C,f^1\circ p^1\Big)\Rightarrow\Big(
C,\id_C,f^2\circ p^2\Big).
\end{equation}

Therefore, we have completely proved Step (a). Now let us fix any other set of choices
\hyperref[C]{C}$\,'(\SETW)$ and let us denote by $\CATC'\left[\SETWinv\right]$ the associated
bicategory of fractions. This bicategory has the same objects, morphisms and $2$-morphisms 
as those of $\CATC\left[\SETWinv\right]$, but compositions of morphisms and $2$-morphisms
are (possibly) different (therefore, we cannot conclude directly that \eqref{eq-156} is 
a weak fiber product also in $\CATC'\left[\SETWinv\right]$). By~\cite[Corollary~3.6]{T4},
there is a pseudofunctor

\[\functor{Q}:\,\CATC\left[\SETWinv\right]\longrightarrow\CATC'\left[\SETWinv\right]\]
that is the identity on objects, morphisms and $2$-morphisms (hence, $\functor{Q}$ is an equivalence
of bicategories). Since $\functor{Q}$ is a pseudofunctor, then its associators
$\Psi^{\functor{Q}}_{\bullet}$
(that are induced by the set of choices \hyperref[C]{C}$(\SETW)$ and \hyperref[C]{C}$\,'(\SETW)$) are
invertible. So for each $m=1,2$ we can consider the invertible $2$-morphism

\begin{gather*}
\Gamma^m:=\Psi^{\functor{Q}}_{(B^m,\id_{B^m},f^m),(C,\id_C,p^m)}:\Big(C,\id_C,f^m\circ p^m\Big)=
 \functor{Q}\Big(C,\id_C,f^m\circ p^m\Big)= \\
=\functor{Q}\Big(\Big(B^m,\id_{B^m},f^m\Big)\circ_{\CATC
 \left[\SETWinv\right]}\Big(C,\id_C,p^m\Big)\Big)\Longrightarrow \\
%%%
\Longrightarrow\functor{Q}\Big(B^m,\id_{B^m},f^m\Big)\circ_{\CATC'
 \left[\SETWinv\right]}\functor{Q}\Big(C,\id_C,p^m\Big)= \\
=\Big(B^m,\id_{B^m},f^m\Big)\circ_{\CATC'\left[\SETWinv\right]}\Big(C,\id_C,p^m\Big)
 =\Big(C,\id_C,f^m\circ p^m\Big)
\end{gather*}
(in the lines above $\circ_{\CATC\left[\SETWinv\right]}$ is the composition in
$\CATC\left[\SETWinv\right]$, and analogously for $\circ_{\CATC'\left[\SETWinv\right]}$).
If we apply~\cite[Lemma~6.1]{T3} for $\alpha:=i_{\id_C}$ and for
$(\Gamma^1)^{-1}$, we get an object $C^1$,
a morphism $\operatorname{z}^1:C^1\rightarrow C$ in $\SETW$ and a $2$-morphism

\[\alpha^1:\,f^1\circ p^1\circ\operatorname{z}^1\Longrightarrow
f^1\circ p^1\circ\operatorname{z}^1\]
in $\CATC$, such that

\[(\Gamma^1)^{-1}=\Big[C^1,\operatorname{z}^1,\operatorname{z}^1,i_{\operatorname{z}^1},\alpha^1
\Big].\]

If we apply~\cite[Lemma~6.1]{T3} for $\alpha:=i_{\operatorname{z}^1}$ and for
$\Gamma^2$, there are an object $\overline{C}$,
a morphism $\operatorname{z}^2:\overline{C}\rightarrow C^1$, such that $\operatorname{z}^1\circ
\operatorname{z}^2$ belongs to $\SETW$, and a $2$-morphism

\[\alpha^2:\,f^2\circ p^2\circ\operatorname{z}^1\circ\operatorname{z}^2\Longrightarrow
f^2\circ p^2\circ\operatorname{z}^1\circ\operatorname{z}^2\]
in $\CATC$, such that

\[\Gamma^2=\Big[\overline{C},\operatorname{z}^1\circ\operatorname{z}^2,\operatorname{z}^1\circ
\operatorname{z}^2,i_{\operatorname{z}^1\circ\operatorname{z}^2},\alpha^2\Big].\]

Since \eqref{eq-156} is a weak fiber product in $\CATC\left[\SETWinv\right]$,
then using Lemma~\ref{lem-14} we get that
the following diagram is a weak fiber product in $\CATC'\left[\SETWinv\right]$:

\begin{equation}\label{eq-108}
\begin{tikzpicture}[xscale=2.6,yscale=-0.8]
    \node (A0_0) at (0, 0) {$C$};
    \node (A0_2) at (2, 0) {$B^1$};
    \node (A2_0) at (0, 2) {$B^2$};
    \node (A2_2) at (2, 2) {$A$.};
    
    \node (A1_1) [rotate=225] at (0.25, 1) {$\Longrightarrow$};
    \node (B1_1) at (1.15, 1) {$\Omega^5:=\Gamma^2\odot\Omega^4\odot(\Gamma^1)^{-1}$};
    
    \path (A0_0) edge [->]node [auto,swap] {$\scriptstyle{(C,\id_C,p^2)}$} (A2_0);
    \path (A0_0) edge [->]node [auto] {$\scriptstyle{(C,\id_C,p^1)}$} (A0_2);
    \path (A0_2) edge [->]node [auto] {$\scriptstyle{(B^1,\id_{B^1},f^1)}$} (A2_2);
    \path (A2_0) edge [->]node [auto,swap] {$\scriptstyle{(B^2,\id_{B^2},f^2)}$} (A2_2);
\end{tikzpicture}
\end{equation}

Now for each $m=1,2$ we set $\overline{p}^m:=p^m\circ\operatorname{z}^1\circ\operatorname{z}^2:
\overline{C}\rightarrow B^m$ and

\[\overline{\omega}:=\alpha^2\odot\Big(\omega\ast i_{\operatorname{z}^1\circ\operatorname{z}^2}
\Big)\odot\Big(\alpha^1\ast i_{\operatorname{z}^2}\Big):\,f^1\circ\overline{p}^1
\Longrightarrow f^2\circ\overline{p}^2.\]

Then a direct check proves that

\[\Omega^5=\Big[\overline{C},\operatorname{z}^1\circ
\operatorname{z}^2,\operatorname{z}^1\circ\operatorname{z}^2,i_{\operatorname{z}^1\circ
\operatorname{z}^2},\overline{\omega}\Big].\]

Now by Lemma~\ref{lem-04} applied to \eqref{eq-108} and to the
internal equivalence $(\overline{C},\id_{\overline{C}},\operatorname{z}^1\circ\operatorname{z}^2)$
(see~\cite[Proposition~20]{Pr}), the following diagram is a weak fiber product in
$\CATC'\left[\SETWinv\right]$:

\[
\begin{tikzpicture}[xscale=2.2,yscale=-0.8]
    \node (A0_0) at (0, 0) {$\overline{C}$};
    \node (A0_2) at (2, 0) {$B^1$};
    \node (A2_0) at (0, 2) {$B^2$};
    \node (A2_2) at (2, 2) {$A$,};
    
    \node (A1_1) [rotate=225] at (0.9, 1) {$\Longrightarrow$};
    \node (A1_2) at (1.2, 1) {$\Omega^6$};
    
    \path (A0_0) edge [->]node [auto,swap] {$\scriptstyle{(\overline{C},
      \id_{\overline{C}},\overline{p}^2)}$} (A2_0);
    \path (A0_0) edge [->]node [auto] {$\scriptstyle{(\overline{C},\id_{\overline{C}},
      \overline{p}^1)}$} (A0_2);
    \path (A0_2) edge [->]node [auto] {$\scriptstyle{(B^1,\id_{B^1},f^1)}$} (A2_2);
    \path (A2_0) edge [->]node [auto,swap] {$\scriptstyle{(B^2,\id_{B^2},f^2)}$} (A2_2);
\end{tikzpicture}
\]
where

\begin{gather*}
\Omega^6:=\Thetaa{(B^2,\id_{B^2},f^2)}{(C,\id_C,p^2)}
  {(\overline{C},\id_{\overline{C}},\operatorname{z}^1\circ\operatorname{z}^2)}\odot \\
%%%
\odot\Big(\Omega^5\ast
 i_{(\overline{C},\id_{\overline{C}},\operatorname{z}^1\circ\operatorname{z}^2)}\Big)\odot
 \Thetab{(B^1,\id_{B^1},f^1)}{(C,\id_C,p^1)}{(\overline{C},\id_{\overline{C}},\operatorname{z}^1\circ
 \operatorname{z}^2)}.
\end{gather*}

By Lemma~\ref{lem-10}, the associators $\Theta_{\bullet}$ above are both trivial, so

\[\Omega^6=\Omega^5\ast
i_{(\overline{C},\id_{\overline{C}},\operatorname{z}^1\circ\operatorname{z}^2)}=
\Big[\overline{C},\id_{\overline{C}},\id_{\overline{C}},
i_{\id_{\overline{C}}},\overline{\omega}\Big].\]

So the quadruple $(\overline{C},\overline{p}^1,\overline{p}^2,\overline{\omega})$ proves that
Step (b) is satisfied in $\CATC'\left[\SETWinv\right]$.
\end{proof}

\begin{rem}
The previous Theorem proves that for each set of choices \hyperref[C]{C}$(\SETW)$
there is a set of data $(C,p^1,p^2,\omega)$ (a priori depending on \hyperref[C]{C}$(\SETW)$),
inducing a weak fiber product \eqref{eq-91} in the bicategory
$\CATC\left[\SETWinv\right]$. A priori we don't know whether such a set
of data induces a weak fiber product also in the bicategory of fractions associated 
to a different set of choices \hyperref[C]{C}$\,'(\SETW)$ or not. Actually,
given any set of data $(C,p^1,p^2,\omega)$, the following facts are equivalent:

\begin{itemize}
 \item diagram \eqref{eq-91} induced by $(C,p^1,p^2,\omega)$ is a weak fiber product in
  $\CATC\left[\SETWinv\right]$;
 \item the same diagram is a weak fiber product in
  $\CATC'\left[\SETWinv\right]$ for any set of choices \hyperref[C]{C}$\,'(\SETW)$.
\end{itemize}

This will be an obvious consequence of the fact that conditions (a), (b) and (c) in
Theorem~\ref{theo-04} (that we are going to prove below)
do not depend on a set of choices \hyperref[C]{C}$(\SETW)$, but only on the
pair $(\CATC,\SETW)$, hence they are verified for every bicategory of fractions $\CATC\left[
\SETWinv\right]$ associated to $(\CATC,\SETW)$.
\end{rem}

As we mentioned in the Introduction,
Theorem~\ref{theo-02} gives an explicit form for a weak fiber product \eqref{eq-91} (whenever
it exists) in the case when the pair of fixed morphism in $\CATC\left[\SETWinv\right]$
has the special form $(B^m,\id_{B^m},f^m)$
for $m=1,2$. The same Theorem shows that whenever such a special pair of morphisms have a weak
fiber product, then also those pairs with $(\id_{B^1},\id_{B^2})$ replaced by any pair of morphisms
in $\SETW$ have a weak
fiber product. However in Theorem~\ref{theo-02} we did not give any explicit description of
a weak fiber product in that case. The next Corollary fills this gap (as in the previous
pages, the $2$-morphisms $\theta_{\bullet}$ and $\upsilon_{\bullet}$ are the associators, respectively
the left unitors of $\CATC$).

\begin{cor}\label{cor-03}
Let us fix any pair $(\CATC,\SETW)$ satisfying conditions \emphatic{(\hyperref[BF]{BF})},
any bicategory of fractions $\CATC\left[\SETWinv\right]$ \emphatic{(}i.e.\ any set of choices
\emphatic{\hyperref[C]{C}}$(\SETW)$\emphatic{)}, any
pair of morphisms $f^m:B^m\rightarrow A$ for $m=1,2$ and any pair of morphisms
$\operatorname{w}^m:B^m\rightarrow\overline{B}^m$ in $\SETW$ for $m=1,2$. Moreover, let us fix
any object $C$, any pair of morphisms $p^m:C\rightarrow B^m$ for $m=1,2$ and any invertible
$2$-morphism $\omega:f^1\circ p^1\Rightarrow f^2\circ p^2$ in $\CATC$, such that
diagram \eqref{eq-91} is a weak fiber product in $\CATC\left[\SETWinv\right]$. In addition,
let us suppose that for each $m=1,2$ the fixed choices \emphatic{\hyperref[C]{C}}$(\SETW)$
give data as in upper part of the following
diagram, with $\operatorname{v}^m$ in $\SETW$ and $\sigma^m$ invertible:

\begin{equation}\label{eq-146}
\begin{tikzpicture}[xscale=2.0,yscale=-0.8]
    \node (A0_1) at (1, 0) {$C^m$};
    \node (A2_0) at (0, 2) {$C$};
    \node (A2_1) at (1, 2) {$\overline{B}^m$};
    \node (A2_2) at (2, 2) {$B^m$};

    \node (A1_1) at (1, 1) {$\sigma^m$};
    \node (B1_1) at (1, 1.4) {$\Rightarrow$};
    
    \path (A0_1) edge [->]node [auto,swap] {$\scriptstyle{\operatorname{v}^m}$} (A2_0);
    \path (A2_2) edge [->]node [auto] {$\scriptstyle{\operatorname{w}^m}$} (A2_1);
    \path (A2_0) edge [->]node [auto,swap] {$\scriptstyle{\operatorname{w}^m\circ p^m}$} (A2_1);
    \path (A0_1) edge [->]node [auto] {$\scriptstyle{q^m}$} (A2_2);
\end{tikzpicture}
\end{equation}
\emphatic{(}this implies that $(B^m,\operatorname{w}^m,f^m)\circ(C,\id_C,\operatorname{w}^m
\circ p^m)=(C^m,\id_C\circ\operatorname{v}^m,f^m\circ q^m)$ for each $m=1,2$\emphatic{)}.
Then let us choose \emph{any set of data} as follows \emphatic{(}the existence of such data
is a consequence of axioms \emphatic{(\hyperref[BF]{BF})}, see \emphatic{Appendix~\ref{sec-03})}:

\begin{enumerate}[\emphatic{(}i\emphatic{)}]
 \item for each $m=1,2$, an object $C^{\prime m}$, a morphism $\operatorname{u}^m:C^{\prime m}
  \rightarrow C^m$ in $\SETW$ and an invertible $2$-morphism $\tau^m:(p^m\circ\operatorname{v}^m)
   \circ\operatorname{u}^m\Rightarrow q^m\circ\operatorname{u}^m$, such that:
   
   \[i_{\operatorname{w}^m}\ast\tau^m=\thetab{\operatorname{w}^m}{q^m}{\operatorname{u}^m}\odot
   \Big(\Big(\sigma^m\odot\thetaa{\operatorname{w}^m}{p^m}{\operatorname{v}^m}\Big)\ast
   i_{\operatorname{u}^m}\Big)\odot
   \thetaa{\operatorname{w}^m}{p^m\circ\operatorname{v}^m}{\operatorname{u}^m};\]
   
  \item an object $C^{\prime\prime}$, a pair of morphisms $\operatorname{z}^m:C^{\prime\prime}
   \rightarrow C^{\prime m}$ for $m=1,2$, with $\operatorname{z}^1$ in $\SETW$, and an invertible
   $2$-morphism $\mu:\operatorname{v}^1\circ(\operatorname{u}^1\circ\operatorname{z}^1)\Rightarrow
   \operatorname{v}^2\circ(\operatorname{u}^2\circ\operatorname{z}^2)$.
\end{enumerate}

Then the following diagram is a weak fiber product in $\CATC\left[\SETWinv\right]$

\begin{equation}\label{eq-139}
\begin{tikzpicture}[xscale=2.0,yscale=-0.8]
    \node (A0_0) at (0, 0) {$C$};
    \node (A0_2) at (2, 0) {$\overline{B}^1$};
    \node (A2_0) at (0, 2) {$\overline{B}^2$};
    \node (A2_2) at (2, 2) {$A$,};
    
    \node (A1_1) [rotate=225] at (0.9, 1) {$\Longrightarrow$};
    \node (A1_2) at (1.2, 1) {$\overline{\Omega}$};
    
    \path (A0_0) edge [->]node [auto,swap]
      {$\scriptstyle{(C,\id_C,\operatorname{w}^2\circ p^2)}$} (A2_0);
    \path (A0_0) edge [->]node [auto] {$\scriptstyle{(C,\id_C,\operatorname{w}^1\circ p^1)}$} (A0_2);
    \path (A0_2) edge [->]node [auto] {$\scriptstyle{(B^1,\operatorname{w}^1,f^1)}$} (A2_2);
    \path (A2_0) edge [->]node [auto,swap] {$\scriptstyle{(B^2,\operatorname{w}^2,f^2)}$} (A2_2);
\end{tikzpicture}
\end{equation}
where

\begin{gather*}
\overline{\Omega}:=\Big[C^{\prime\prime},\operatorname{u}^1\circ\operatorname{z}^1,
 \operatorname{u}^2\circ\operatorname{z}^2,\Big(\upsilon_{\operatorname{v}^2}^{-1}\ast
 i_{\operatorname{u}^2\circ\operatorname{z}^2}\Big)\odot\mu\odot\Big(\upsilon_{\operatorname{v}^1}
 \ast i_{\operatorname{u}^1\circ\operatorname{z}^1}\Big),\delta\Big]: \\
%%%
\Big(C^1,\id_C\circ\operatorname{v}^1,f^1\circ q^1\Big)\Longrightarrow
 \Big(C^2,\id_C\circ\operatorname{v}^2,f^2\circ q^2\Big)
\end{gather*}
and $\delta:(f^1\circ q^1)\circ(\operatorname{u}^1\circ\operatorname{z}^1)\Rightarrow
(f^2\circ q^2)\circ(\operatorname{u}^2\circ\operatorname{z}^2)$ is defined as the following
composition \emphatic{(}associators of $\CATC$ omitted for simplicity\emphatic{)}:

\[
\begin{tikzpicture}[xscale=2.2,yscale=-0.8]
    \node (A0_2) at (3, 0) {$C^1$};
    \node (A1_1) at (1, 0) {$C^{\prime 1}$};
    \node (A1_2) at (2, 1) {$C^1$};
    \node (A1_4) at (4, 1) {$B^1$};
    \node (A2_0) at (0, 2) {$C^{\prime\prime}$};
    \node (A2_3) at (3, 2) {$C$};
    \node (A2_5) at (5, 2) {$A$.};
    \node (A3_1) at (1, 4) {$C^{\prime 2}$};
    \node (A3_2) at (2, 3) {$C^2$};
    \node (A3_4) at (4, 3) {$B^2$};
    \node (A4_2) at (3, 4) {$C^2$};

    \node (A1_3) at (2.7, 0.8) {$\Downarrow\,(\tau^1)^{-1}$};
    \node (A3_3) at (2.7, 3.2) {$\Downarrow\,\tau^2$};
    \node (A2_4) at (4, 2) {$\Downarrow\,\omega$};
    \node (A2_1) at (1, 2) {$\Downarrow\,\mu$};
    
    \path (A2_3) edge [->]node [auto,swap] {$\scriptstyle{p^2}$} (A3_4);
    \path (A1_1) edge [->]node [auto] {$\scriptstyle{\operatorname{u}^1}$} (A0_2);
    \path (A1_2) edge [->]node [auto,swap] {$\scriptstyle{\operatorname{v}^1}$} (A2_3);
    \path (A2_0) edge [->]node [auto] {$\scriptstyle{\operatorname{z}^1}$} (A1_1);
    \path (A3_1) edge [->]node [auto] {$\scriptstyle{\operatorname{u}^2}$} (A3_2);
    \path (A2_0) edge [->]node [auto,swap] {$\scriptstyle{\operatorname{z}^2}$} (A3_1);
    \path (A3_2) edge [->]node [auto] {$\scriptstyle{\operatorname{v}^2}$} (A2_3);
    \path (A1_1) edge [->]node [auto,swap] {$\scriptstyle{\operatorname{u}^1}$} (A1_2);
    \path (A1_4) edge [->]node [auto] {$\scriptstyle{f^1}$} (A2_5);
    \path (A3_4) edge [->]node [auto,swap] {$\scriptstyle{f^2}$} (A2_5);
    \path (A0_2) edge [->]node [auto] {$\scriptstyle{q^1}$} (A1_4);
    \path (A3_1) edge [->]node [auto,swap] {$\scriptstyle{\operatorname{u}^2}$} (A4_2);
    \path (A2_3) edge [->]node [auto] {$\scriptstyle{p^1}$} (A1_4);
    \path (A4_2) edge [->]node [auto,swap] {$\scriptstyle{q^2}$} (A3_4);
\end{tikzpicture}
\]
\end{cor}

\begin{proof}
For simplicity of exposition, we give a complete proof only in the case when $\CATC$ is a
$2$-category.
For each $m=1,2$, let us suppose that the fixed choices \hyperref[C]{C}$(\SETW)$ give a set
of data as in the upper part of the following diagram, with $\operatorname{t}^m$ in $\SETW$
and $\xi^m$ invertible:

\begin{equation}\label{eq-180}
\begin{tikzpicture}[xscale=2.2,yscale=-0.8]
    \node (A0_1) at (1, 0) {$\widetilde{B}^m$};
    \node (A2_0) at (0, 2) {$B^m$};
    \node (A2_1) at (1, 2) {$\overline{B}^m$};
    \node (A2_2) at (2, 2) {$B^m$.};
    
    \node (A1_1) at (1, 1) {$\xi^m$};
    \node (B1_1) at (1, 1.4) {$\Rightarrow$};
    
    \path (A0_1) edge [->]node [auto,swap] {$\scriptstyle{\operatorname{t}^m}$} (A2_0);
    \path (A2_2) edge [->]node [auto] {$\scriptstyle{\operatorname{w}^m}$} (A2_1);
    \path (A2_0) edge [->]node [auto,swap] {$\scriptstyle{\operatorname{w}^m}$} (A2_1);
    \path (A0_1) edge [->]node [auto] {$\scriptstyle{\operatorname{s}^m}$} (A2_2);
\end{tikzpicture}
\end{equation}

Note that the choices \hyperref[C]{C}$(\SETW)$ here are arbitrary, so we cannot use
condition (\hyperref[C3]{C3}) for the previous diagram (see Appendix~\ref{sec-01}).
For each $m=1,2$, we apply axioms (\hyperref[BF4a]{BF4a}) and (\hyperref[BF4b]{BF4b}) to the
invertible $2$-morphism $\xi^m$; so there are an object $B^{\prime m}$, a morphism
$\operatorname{r}^m:B^{\prime m}\rightarrow\widetilde{B}^m$ in $\SETW$ and an invertible $2$-morphism
$\varepsilon^m:\operatorname{t}^m\circ\operatorname{r}^m\Rightarrow\operatorname{s}^m\circ
\operatorname{r}^m$, such that

\begin{equation}\label{eq-181}
i_{\operatorname{w}^m}\ast\varepsilon^m=\xi^m\ast i_{\operatorname{r}^m}.
\end{equation}

For each $m=1,2$, we consider the following morphisms in $\CATC\left[\SETWinv\right]$

\[
\begin{tikzpicture}[xscale=1.8,yscale=-1.2]
    \node (A0_0) at (-0.3, 0) {$e^m:=\Big(\overline{B}^m$};
    \node (A0_1) at (1, 0) {$B^m$};
    \node (A0_2) at (2, 0) {$B^m\Big)$};
    \path (A0_1) edge [->]node [auto,swap] {$\scriptstyle{\operatorname{w}^m}$} (A0_0);
    \path (A0_1) edge [->]node [auto] {$\scriptstyle{\id_{B^m}}$} (A0_2);
    \node (B0_0) at (2.5, 0) {$\textrm{and}$};
    
    \node (C0_0) at (3.3, 0) {$d^m:=\Big(B^m$};
    \node (C0_1) at (4.6, 0) {$B^m$};
    \node (C0_2) at (5.6, 0) {$\overline{B}^m\Big)$.};
    \path (C0_1) edge [->]node [auto,swap] {$\scriptstyle{\id_{B^m}}$} (C0_0);
    \path (C0_1) edge [->]node [auto] {$\scriptstyle{\operatorname{w}^m}$} (C0_2);    
\end{tikzpicture}
\]

Using \eqref{eq-180}, for each $m=1,2$ we get

\[
\begin{tikzpicture}[xscale=1.5,yscale=-1.2]
    \node (A0_0) at (-0.6, 0) {$d^m\circ e^m=\Big(\overline{B}^m$};
    \node (A0_1) at (1, 0) {$B^m$};
    \node (A0_2) at (2, 0) {$\overline{B}^m\Big)$};
    \path (A0_1) edge [->]node [auto,swap] {$\scriptstyle{\operatorname{w}^m}$} (A0_0);
    \path (A0_1) edge [->]node [auto] {$\scriptstyle{\operatorname{w}^m}$} (A0_2);
    \node (B0_0) at (2.65, 0) {$\textrm{and}$};
    
    \node (C0_0) at (3.9, 0) {$e^m\circ d^m=\Big(B^m$};
    \node (C0_1) at (5.6, 0) {$\widetilde{B}^m$};
    \node (C0_2) at (6.6, 0) {$B^m\Big)$.};
    \path (C0_1) edge [->]node [auto,swap] {$\scriptstyle{\operatorname{t}^m}$} (C0_0);
    \path (C0_1) edge [->]node [auto] {$\scriptstyle{\operatorname{s}^m}$} (C0_2);    
\end{tikzpicture}
\]

Then for each $m=1,2$ we define an invertible $2$-morphism $\Delta^m:\id_{\overline{B}^m}
\Rightarrow d^m\circ e^m$ in $\CATC\left[\SETWinv\right]$ as the
$2$-morphism represented by the following diagram:

\[
\begin{tikzpicture}[xscale=2.2,yscale=-0.8]
    \node (A0_2) at (2, 0) {$\overline{B}^m$};
    \node (A2_0) at (0, 2) {$\overline{B}^m$};
    \node (A2_2) at (2, 2) {$B^m$};
    \node (A2_4) at (4, 2) {$\overline{B}^m$;};
    \node (A4_2) at (2, 4) {$B^m$};

    \node (A2_1) at (1.2, 2) {$\Downarrow\,i_{\operatorname{w}^m}$};
    \node (A2_3) at (2.8, 2) {$\Downarrow\,i_{\operatorname{w}^m}$};
    
    \path (A0_2) edge [->]node [auto,swap] {$\scriptstyle{\id_{\overline{B}^m}}$} (A2_0);
    \path (A0_2) edge [->]node [auto] {$\scriptstyle{\id_{\overline{B}^m}}$} (A2_4);
    \path (A4_2) edge [->]node [auto] {$\scriptstyle{\operatorname{w}^m}$} (A2_0);
    \path (A4_2) edge [->]node [auto,swap] {$\scriptstyle{\operatorname{w}^m}$} (A2_4);
    \path (A2_2) edge [->]node [auto,swap] {$\scriptstyle{\operatorname{w}^m}$} (A0_2);
    \path (A2_2) edge [->]node [auto] {$\scriptstyle{\id_{B^m}}$} (A4_2);
\end{tikzpicture}
\]

Moreover, for each $m=1,2$ we define an invertible $2$-morphism $\Xi^m:e^m\circ
d^m\Rightarrow\id_{B^m}$ in $\CATC\left[\SETWinv\right]$ as the $2$-morphism represented
by the following diagram:

\begin{equation}\label{eq-109}
\begin{tikzpicture}[xscale=2.2,yscale=-0.8]
    \node (A0_2) at (2, 0) {$\widetilde{B}^m$};
    \node (A2_0) at (0, 2) {$B^m$};
    \node (A2_2) at (2, 2) {$B^{\prime m}$};
    \node (A2_4) at (4, 2) {$B^m$.};
    \node (A4_2) at (2, 4) {$B^m$};

    \node (A2_1) at (1.2, 2) {$\Downarrow\,\varepsilon^m$};
    \node (A2_3) at (2.8, 2) {$\Downarrow\,i_{\operatorname{s}^m\circ\operatorname{r}^m}$};
    
    \path (A0_2) edge [->]node [auto,swap] {$\scriptstyle{\operatorname{t}^m}$} (A2_0);
    \path (A0_2) edge [->]node [auto] {$\scriptstyle{\operatorname{s}^m}$} (A2_4);
    \path (A4_2) edge [->]node [auto] {$\scriptstyle{\id_{B^m}}$} (A2_0);
    \path (A4_2) edge [->]node [auto,swap] {$\scriptstyle{\id_{B^m}}$} (A2_4);
    \path (A2_2) edge [->]node [auto,swap] {$\scriptstyle{\operatorname{r}^m}$} (A0_2);
    \path (A2_2) edge [->]node [auto] {$\scriptstyle{\operatorname{s}^m
      \circ\operatorname{r}^m}$} (A4_2);
\end{tikzpicture}
\end{equation}

Following the proof of~\cite[Proposition~20]{Pr}, the quadruple $(e^m,d^m,\Delta^m,\Xi^m)$ is an
\emph{adjoint} equivalence in $\CATC\left[\SETWinv\right]$ for each $m=1,2$.
For each such $m$, let us set:

\[
\begin{tikzpicture}[xscale=1.8,yscale=-0.8]
    \node (A0_0) at (0, 0) {$g^m:=\Big(B^m$};
    \node (A0_1) at (1.25, 0) {$B^m$};
    \node (A0_2) at (2.2, 0) {$A\Big)$};
    \path (A0_1) edge [->]node [auto,swap] {$\scriptstyle{\id_{B^m}}$} (A0_0);
    \path (A0_1) edge [->]node [auto] {$\scriptstyle{f^m}$} (A0_2);
    \node (B0_0) at (2.95, 0) {$\textrm{and}$};
    
    \node (C0_0) at (3.9, 0) {$r^m:=\Big(C$};
    \node (C0_1) at (5.05, 0) {$C$};
    \node (C0_2) at (6, 0) {$B^m\Big)$.};
    \path (C0_1) edge [->]node [auto,swap] {$\scriptstyle{\id_C}$} (C0_0);
    \path (C0_1) edge [->]node [auto] {$\scriptstyle{p^m}$} (C0_2);   
\end{tikzpicture}
\]

Since we assumed that $\CATC$ is a $2$-category, then the $2$-morphism $\Omega$ of 
$\CATC\left[\SETWinv\right]$ appearing in \eqref{eq-91} is given by

\begin{equation}\label{eq-193}
\Omega=\Big[C,\id_C,\id_C,i_{\id_C},\omega\Big]:\,g^1\circ r^1\Longrightarrow g^2\circ r^2.
\end{equation}

Then we define an invertible $2$-morphism

\[\overline{\Omega}:\,(g^1\circ e^1)\circ(d^1\circ r^1)\Longrightarrow
(g^2\circ e^2)\circ(d^2\circ r^2)\]
as the following composition, where the $2$-morphisms $\Theta_{\bullet}$ are the associators of
$\CATC\left[\SETWinv\right]$:

\begin{equation}\label{eq-182}
\begin{tikzpicture}[xscale=4.5,yscale=-0.8]
    \node (A3_0) at (0, 3) {$C$};
    \node (A3_1) at (2, 3) {$A$.};
    
    \node (D0_1) at (1.3, 0) {$\Downarrow\,\Thetaa{g^1\circ e^1}{d^1}{r^1}$};
    \node (D1_1) at (1.3, 1) {$\Downarrow\,\Thetab{g^1}{e^1}{d^1}\ast i_{r^1}$};
    \node (D2_1) at (1.3, 2) {$\Downarrow\,(i_{g^1}\ast\Xi^1)\ast i_{r^1}$};
    \node (D3_1) at (1.3, 3) {$\Downarrow\,\Omega$};
    \node (D4_1) at (1.3, 4) {$\Downarrow\,(i_{g^2}\ast(\Xi^2)^{-1})\ast i_{r^2}$};
    \node (D5_1) at (1.3, 5) {$\Downarrow\,\Thetaa{g^2}{e^2}{d^2}\ast i_{r^2}$};
    \node (D6_1) at (1.3, 6) {$\Downarrow\,\Thetab{g^2\circ e^2}{d^2}{r^2}$};
    
    \foreach \i in {-1,...,6} {\draw[rounded corners,->] (A3_0) to (0.2,\i+0.5)
      to (1.8,\i+0.5) to (A3_1);}
    
    \node (B1_1) at (0.7, -0.75) {$\scriptstyle{(g^1\circ e^1)\circ(d^1\circ r^1)}$};
    \node (B2_2) at (0.7, 0.25) {$\scriptstyle{((g^1\circ e^1)\circ d^1)\circ r^1}$};
    \node (B3_3) at (0.7, 1.25) {$\scriptstyle{(g^1\circ(e^1\circ d^1))\circ r^1}$};
    \node (B4_4) at (0.63, 2.25) {$\scriptstyle{(g^1\circ\id_{B^1})\circ r^1=g^1\circ r^1}$};
    \node (B5_5) at (0.63, 3.25) {$\scriptstyle{(g^2\circ\id_{B^2})\circ r^2=g^2\circ r^2}$};
    \node (B6_6) at (0.7, 4.25) {$\scriptstyle{(g^2\circ(e^2\circ d^2))\circ r^2}$};
    \node (B7_7) at (0.7, 5.25) {$\scriptstyle{((g^2\circ e^2)\circ d^2)\circ r^2}$};
    \node (B8_8) at (0.7, 6.25) {$\scriptstyle{(g^2\circ e^2)\circ(d^2\circ r^2)}$};
\end{tikzpicture}
\end{equation}

Since we are working in the case when $\CATC$ is a $2$-category, then it is easy to prove
that the unitors of $\CATC\left[\SETWinv\right]$ are trivial. Therefore, the $2$-morphism
$\overline{\Omega}$ above coincides with the $2$-morphism $\overline{\Omega}$ defined in
Proposition~\ref{prop-01} for $\CATD:=\CATC\left[\SETWinv\right]$.
By hypothesis, diagram \eqref{eq-91} is a weak fiber product in
$\CATC\left[\SETWinv\right]$, so by Proposition~\ref{prop-01} we get that also the
following diagram is a weak fiber product in $\CATC\left[\SETWinv\right]$:

\[
\begin{tikzpicture}[xscale=2.6,yscale=-0.8]
    \node (A0_0) at (0, 0) {$C$};
    \node (A0_2) at (2, 0) {$\overline{B}^1$};
    \node (A2_0) at (0, 2) {$\overline{B}^2$};
    \node (A2_2) at (2, 2) {$A$.};
    
    \node (A1_1) [rotate=225] at (0.9, 1) {$\Longrightarrow$};
    \node (A1_2) at (1.1, 1) {$\overline{\Omega}$};
    
    \path (A0_0) edge [->]node [auto,swap] {$\scriptstyle{d^2\circ r^2
      =(C,\id_C,\operatorname{w}^2\circ p^2)}$} (A2_0);
    \path (A0_0) edge [->]node [auto] {$\scriptstyle{d^1\circ r^1
      =(C,\id_C,\operatorname{w}^1\circ p^1)}$} (A0_2);
    \path (A0_2) edge [->]node [auto] {$\scriptstyle{g^1\circ e^1=
      (B^1,\operatorname{w}^1,f^1)}$} (A2_2);
    \path (A2_0) edge [->]node [auto,swap] {$\scriptstyle{g^2\circ e^2=
      (B^2,\operatorname{w}^2,f^2)}$} (A2_2);
\end{tikzpicture}
\]

Then in order to prove the claim we need only to compute all the $2$-morphisms in \eqref{eq-182}
and to prove that their composition is equal to the $2$-morphism in \eqref{eq-139}.\\

Since we are assuming that $\CATC$ is a $2$-category, then for each $m=1,2$ we have

\begin{gather}
\nonumber (g^m\circ e^m)\circ(d^m\circ r^m)=\Big(B^m,\operatorname{w}^m,f^m\Big)\circ\Big(C,\id_C,
 \operatorname{w}^m\circ p^m\Big)\stackrel{\eqref{eq-146}}{=} \\
%%%
\label{eq-186} \stackrel{\eqref{eq-146}}{=}\Big(C^m,\operatorname{v}^m,f^m\circ q^m\Big).
\end{gather}

Let us suppose that for each $m=1,2$ the fixed choices \hyperref[C]{C}$(\SETW)$ give data
as in the upper part of the following diagram, with $\operatorname{k}^m$ in $\SETW$ and $\eta^m$
invertible:

\begin{equation}\label{eq-185}
\begin{tikzpicture}[xscale=2.2,yscale=-0.8]
    \node (A0_1) at (1, 0) {$F^m$};
    \node (A2_0) at (0, 2) {$C$};
    \node (A2_1) at (1, 2) {$B^m$};
    \node (A2_2) at (2, 2) {$\widetilde{B}^m$.};
    
    \node (A1_1) at (1, 1) {$\eta^m$};
    \node (B1_1) at (1, 1.4) {$\Rightarrow$};

    \path (A0_1) edge [->]node [auto,swap] {$\scriptstyle{\operatorname{k}^m}$} (A2_0);
    \path (A2_2) edge [->]node [auto] {$\scriptstyle{\operatorname{t}^m}$} (A2_1);
    \path (A2_0) edge [->]node [auto,swap] {$\scriptstyle{p^m}$} (A2_1);
    \path (A0_1) edge [->]node [auto] {$\scriptstyle{h^m}$} (A2_2);
\end{tikzpicture}
\end{equation}

Then for each $m=1,2$ we have:

\begin{gather}
\nonumber ((g^m\circ e^m)\circ d^m)\circ r^m=\Big(\Big(B^m,\operatorname{w}^m,f^m\Big)
 \circ\Big(B^m,\id_{B^m},\operatorname{w}^m\Big)\Big)\circ\Big(C,\id_C,p^m\Big)
 \stackrel{\eqref{eq-180}}{=} \\
%%%
\nonumber \stackrel{\eqref{eq-180}}{=}\Big(\widetilde{B}^m,\operatorname{t}^m,f^m\circ
 \operatorname{s}^m\Big)\circ\Big(C,\id_C,p^m\Big)\stackrel{\eqref{eq-185}}{=} \\
%%%
\label{eq-187} \stackrel{\eqref{eq-185}}{=}\Big(F^m,\operatorname{k}^m,f^m\circ
 \operatorname{s}^m\circ h^m\Big).
\end{gather}

Now for each $m=1,2$, we want to compute the associators $\Thetaa{g^m\circ e^m}{d^m}{r^m}$
from (\ref{eq-186}) to (\ref{eq-187}) appearing in \eqref{eq-182}.
As a preliminary step, for each $m=1,2$
we use axiom (\hyperref[BF3]{BF3}) in order to get a set of data as in the
upper part of the following diagram, with $\operatorname{a}^m$ in $\SETW$ and
$\gamma^m$ invertible:

\[
\begin{tikzpicture}[xscale=2.4,yscale=-0.8]
    \node (A0_1) at (1, 0) {$G^m$};
    \node (A2_0) at (0, 2) {$C''$};
    \node (A2_1) at (1, 2) {$C$};
    \node (A2_2) at (2, 2) {$F^m$.};

    \node (A1_1) at (1, 1) {$\gamma^m$};
    \node (B1_1) at (1, 1.4) {$\Rightarrow$};

    \path (A0_1) edge [->]node [auto,swap] {$\scriptstyle{\operatorname{a}^m}$} (A2_0);
    \path (A2_2) edge [->]node [auto] {$\scriptstyle{\operatorname{k}^m}$} (A2_1);
    \path (A2_0) edge [->]node [auto,swap] {$\scriptstyle{\operatorname{v}^m\circ
      \operatorname{u}^m\circ\operatorname{z}^m}$} (A2_1);
    \path (A0_1) edge [->]node [auto] {$\scriptstyle{\operatorname{b}^m}$} (A2_2);
\end{tikzpicture}
\]

Then we use (\hyperref[BF4a]{BF4a}) and (\hyperref[BF4b]{BF4b}) in order to get an object $T^m$,
a morphism $\operatorname{j}^m:T^m\rightarrow G^m$ in $\SETW$ and an invertible $2$-morphism

\[\rho^m:\,q^m\circ\operatorname{u}^m\circ\operatorname{z}^m\circ
\operatorname{a}^m\circ\operatorname{j}^m\Longrightarrow\operatorname{s}^m\circ
h^m\circ\operatorname{b}^m\circ\operatorname{j}^m,\]
such that $i_{\operatorname{w}^m}\ast\rho^m$ coincides with the following composition:

\begin{equation}\label{eq-191}
\begin{tikzpicture}[xscale=2.3,yscale=-1.2]
    \node (A0_2) at (2, 0) {$C^m$};
    \node (A0_4) at (4, 0) {$B^m$};
    \node (A1_0) at (0.3, 1) {$T^m$};
    \node (A1_1) at (1, 1) {$G^m$};
    \node (A1_3) at (3, 1) {$C$};
    \node (A1_4) at (4, 1) {$B^m$};
    \node (A1_5) at (5, 1) {$\overline{B}^m$.};
    \node (A2_2) at (1.8, 1) {$F^m$};
    \node (A2_3) at (3, 2) {$\widetilde{B}^m$};
    \node (A2_4) at (4, 2) {$B^m$};
    
    \node (A2_5) at (4, 1.5) {$\Downarrow\,\xi^m$};
    \node (A2_1) at (2.9, 1.4) {$\Downarrow\,\eta^m$};
    \node (A0_1) at (2, 0.55) {$\Downarrow\,\gamma^m$};
    \node (A0_3) at (3.5, 0.4) {$\Downarrow\,(\sigma^m)^{-1}$};

    \path (A2_4) edge [->]node [auto,swap] {$\scriptstyle{\operatorname{w}^m}$} (A1_5);
    \path (A1_4) edge [->]node [auto] {$\scriptstyle{\operatorname{w}^m}$} (A1_5);
    \path (A1_1) edge [->]node [auto] {$\scriptstyle{\operatorname{u}^m\circ
      \operatorname{z}^m\circ\operatorname{a}^m}$} (A0_2);
    \path (A0_4) edge [->]node [auto] {$\scriptstyle{\operatorname{w}^m}$} (A1_5);
    \path (A1_3) edge [->]node [auto] {$\scriptstyle{p^m}$} (A1_4);
    \path (A1_0) edge [->]node [auto] {$\scriptstyle{\operatorname{j}^m}$} (A1_1);
    \path (A2_2) edge [->]node [auto,swap] {$\scriptstyle{\operatorname{k}^m}$} (A1_3);
    \path (A1_1) edge [->]node [auto,swap] {$\scriptstyle{\operatorname{b}^m}$} (A2_2);
    \path (A2_3) edge [->]node [auto,swap] {$\scriptstyle{\operatorname{s}^m}$} (A2_4);
    \path (A0_2) edge [->]node [auto] {$\scriptstyle{q^m}$} (A0_4);
    \path (A0_2) edge [->]node [auto] {$\scriptstyle{\operatorname{v}^m}$} (A1_3);
    \path (A2_2) edge [->]node [auto,swap] {$\scriptstyle{h^m}$} (A2_3);
    \path (A2_3) edge [->]node [auto] {$\scriptstyle{\operatorname{t}^m}$} (A1_4);
\end{tikzpicture}
\end{equation}

Then we compute the associator from (\ref{eq-186}) to (\ref{eq-187})
using~\cite[Proposition~0.1]{T3} for $\underline{f}:=r^m$,
$\underline{g}:=d^m$ and $\underline{h}:=g^m\circ e^m$. Using the previous choices, we have
that the $2$-morphisms appearing in~\cite[Proposition~0.1(0.4)]{T3} are given as follows:

\[\delta:=i_{p^m},\quad\sigma:=\sigma^m,\quad\xi:=\xi^m,\quad\eta:=\eta^m.\]

Then in~\cite[Proposition~0.1]{T3} we choose

\[\gamma:=\gamma^m\ast i_{\operatorname{j}^m},\quad
\omega:=\Big(\eta^m\ast i_{\operatorname{b}^m\circ\operatorname{j}^m}\Big)\odot\Big(i_{p^m}
\ast\gamma^m\ast i_{\operatorname{j}^m}\Big),\quad\rho:=\rho^m.\]

Then we get that the associator $\Thetaa{\underline{g}^m\circ \underline{e}^m}
{\underline{d}^m}{\underline{r}^m}$ from \eqref{eq-186} to
\eqref{eq-187} is represented by the following diagram:

\begin{equation}\label{eq-190}
\begin{tikzpicture}[xscale=2.9,yscale=-0.75]
    \node (A0_2) at (2, 0) {$C^m$};
    \node (A2_0) at (0, 2) {$C$};
    \node (A2_2) at (2, 2) {$T^m$};
    \node (A2_4) at (4, 2) {$A$.};
    \node (A4_2) at (2, 4) {$F^m$};
    
    \node (A2_3) at (2.8, 2) {$\Downarrow\,i_{f^m}\ast\rho^m$};
    \node (A2_1) at (1.2, 2) {$\Downarrow\,\gamma^m\ast i_{\operatorname{j}^m}$};

    \path (A0_2) edge [->]node [auto,swap] {$\scriptstyle{\operatorname{v}^m}$} (A2_0);
    \path (A4_2) edge [->]node [auto] {$\scriptstyle{\operatorname{k}^m}$} (A2_0);
    \path (A4_2) edge [->]node [auto,swap] {$\scriptstyle{f^m
      \circ\operatorname{s}^m\circ h^m}$} (A2_4);
    \path (A2_2) edge [->]node [auto,swap] {$\scriptstyle{\operatorname{u}^m\circ
     \operatorname{z}^m\circ\operatorname{a}^m\circ\operatorname{j}^m}$} (A0_2);
    \path (A2_2) edge [->]node [auto] {$\scriptstyle{\operatorname{b}^m
      \circ\operatorname{j}^m}$} (A4_2);
    \path (A0_2) edge [->]node [auto] {$\scriptstyle{f^m\circ q^m}$} (A2_4);
\end{tikzpicture}
\end{equation}

Now using \eqref{eq-180}, it is easy to prove that

\[g^m\circ(e^m\circ d^m)=\Big(\widetilde{B}^m,\operatorname{t}^m,f^m\circ\operatorname{s}^m\Big)
=(g^m\circ e^m)\circ d^m\]
and that $\Thetaa{g^m}{e^m}{d^m}$ is the $2$-identity of this morphism; hence for each $m=1,2$,
in diagram \eqref{eq-182} we have:

\begin{equation}\label{eq-189}
\Thetaa{g^m}{e^m}{d^m}\ast i_{r^m}=i_{((g^m\circ e^m)\circ d^m)\circ r^m}
\stackrel{\eqref{eq-185}}{=}i_{\left(
F^m,\operatorname{k}^m,f^m\circ\operatorname{s}^m\circ h^m\right)}.
\end{equation}

Now a direct check using \eqref{eq-109} proves that for each $m=1,2$ the $2$-morphism
$i_{g^m}\ast\Xi^m$ is represented by the following diagram:

\[
\begin{tikzpicture}[xscale=2.2,yscale=-0.8]
    \node (A0_2) at (2, 0) {$\widetilde{B}^m$};
    \node (A2_0) at (0, 2) {$B^m$};
    \node (A2_2) at (2, 2) {$B^{\prime m}$};
    \node (A2_4) at (4, 2) {$A$.};
    \node (A4_2) at (2, 4) {$B^m$};

    \node (A2_1) at (1.2, 2) {$\Downarrow\,\varepsilon^m$};
    \node (A2_3) at (2.8, 2) {$\Downarrow\,i_{f^m\circ\operatorname{s}^m\circ\operatorname{r}^m}$};
    
    \path (A0_2) edge [->]node [auto,swap] {$\scriptstyle{\operatorname{t}^m}$} (A2_0);
    \path (A0_2) edge [->]node [auto] {$\scriptstyle{f^m\circ\operatorname{s}^m}$} (A2_4);
    \path (A4_2) edge [->]node [auto] {$\scriptstyle{\id_{B^m}}$} (A2_0);
    \path (A4_2) edge [->]node [auto,swap] {$\scriptstyle{f^m}$} (A2_4);
    \path (A2_2) edge [->]node [auto,swap] {$\scriptstyle{\operatorname{r}^m}$} (A0_2);
    \path (A2_2) edge [->]node [auto] {$\scriptstyle{\operatorname{s}^m
      \circ\operatorname{r}^m}$} (A4_2);
\end{tikzpicture}
\]

Then we need to compute the $2$-morphisms $(i_{g^m}\ast\Xi^m)\ast i_{r^m}$ appearing in
\eqref{eq-182}. For each $m=1,2$ we use axiom
(\hyperref[BF3]{BF3}) in order to get data as in the upper part of the following diagram, with
$\operatorname{c}^m$ in $\SETW$ and $\phi^m$ invertible:

\[
\begin{tikzpicture}[xscale=2.5,yscale=-0.8]
    \node (A0_1) at (1, 0) {$L^m$};
    \node (A2_0) at (0, 2) {$T^m$};
    \node (A2_1) at (1, 2) {$\widetilde{B}^m$};
    \node (A2_2) at (2, 2) {$B^{\prime m}$};

    \node (A1_1) at (1, 1) {$\phi^m$};
    \node (B1_1) at (1, 1.4) {$\Rightarrow$};

    \path (A0_1) edge [->]node [auto,swap] {$\scriptstyle{\operatorname{c}^m}$} (A2_0);
    \path (A2_2) edge [->]node [auto] {$\scriptstyle{\operatorname{r}^m}$} (A2_1);
    \path (A2_0) edge [->]node [auto,swap] {$\scriptstyle{h^m
      \circ\operatorname{b}^m\circ\operatorname{j}^m}$} (A2_1);
    \path (A0_1) edge [->]node [auto] {$\scriptstyle{o^m}$} (A2_2);
\end{tikzpicture}
\]

Then we use~\cite[Proposition~0.3]{T3} in order to compute $(i_{g^m}\ast\Xi^m)\ast
i_{r^m}$. In 
the case under exam, the $2$-morphisms $\alpha$ and $\beta$ appearing in that Proposition
are given by $\varepsilon^m$ and
$i_{f^m\circ\operatorname{s}^m\circ\operatorname{r}^m}$ respectively; moreover, the $2$-morphisms
$\rho^1,\rho^2$ of that Proposition are given by
$\eta^m$ and $i_{p^m}$ respectively. Then we choose the $2$-morphisms $\sigma^1,
\sigma^2$ and $\alpha'$ appearing in that Proposition as follows: we set
$\sigma^1:=\phi^m$, $\alpha':=i_{\operatorname{k}^m\circ\operatorname{b}^m\circ\operatorname{j}^m
\circ\operatorname{c}^m}$ and we define $\sigma^2$ as the following composition:

\[
\begin{tikzpicture}[xscale=2.4,yscale=-1.0]
    \node (A0_2) at (1, 0) {$F^m$};
    \node (A1_0) at (0, 1) {$L^m$};
    \node (A1_2) at (2, 1) {$\widetilde{B}^m$};
    \node (A1_3) at (3, 1) {$B^m$.};
    \node (A2_1) at (1, 2) {$B^{\prime m}$};
    \node (A2_2) at (2, 2) {$\widetilde{B}^m$};
    \node (B2_2) at (2, 0) {$C$};
    
    \node (A2_0) at (2, 1.5) {$\Downarrow\,\varepsilon^m$};
    \node (A0_3) at (2, 0.5) {$\Downarrow\,\eta^m$};
    \node (A1_1) at (0.9, 1) {$\Downarrow\,\phi^m$};

    \path (A0_2) edge [->]node [auto] {$\scriptstyle{\operatorname{k}^m}$} (B2_2);
    \path (B2_2) edge [->]node [auto] {$\scriptstyle{p^m}$} (A1_3);
    \path (A1_0) edge [->]node [auto,swap] {$\scriptstyle{o^m}$} (A2_1);
    \path (A1_0) edge [->]node [auto] {$\scriptstyle{\operatorname{b}^m\circ\operatorname{j}^m
      \circ\operatorname{c}^m}$} (A0_2);
    \path (A0_2) edge [->]node [auto,swap] {$\scriptstyle{h^m}$} (A1_2);
    \path (A2_2) edge [->]node [auto,swap] {$\scriptstyle{\operatorname{s}^m}$} (A1_3);
    \path (A2_1) edge [->]node [auto] {$\scriptstyle{\operatorname{r}^m}$} (A1_2);
    \path (A1_2) edge [->]node [auto] {$\scriptstyle{\operatorname{t}^m}$} (A1_3);
    \path (A2_1) edge [->]node [auto,swap] {$\scriptstyle{\operatorname{r}^m}$} (A2_2);
\end{tikzpicture}
\]

Then replacing all these choices in~\cite[Proposition~0.3(0.12)]{T3} we get the $2$-identity
over $\operatorname{t}^m\circ\operatorname{r}^m\circ\,o^m$. So in that Proposition we can choose
$\delta:=i_{o^m}$. Therefore, replacing in the definition of $\beta'$ in that Proposition, we
conclude that for each $m=1,2$ the $2$-morphism $(i_{g^m}\ast\Xi^m)\ast i_{r^m}$ appearing in
\eqref{eq-182} is represented by the following diagram

\begin{equation}\label{eq-188}
\begin{tikzpicture}[xscale=2.6,yscale=-0.8]
    \node (A0_2) at (2, 0) {$F^m$};
    \node (A2_0) at (0, 2) {$C$};
    \node (A2_2) at (2, 2) {$L^m$};
    \node (A2_4) at (4, 2) {$A$,};
    \node (A4_2) at (2, 4) {$C$};

    \node (A2_1) at (1.2, 2) {$\Downarrow\,i_{\operatorname{k}^m\circ
      \operatorname{b}^m\circ\operatorname{j}^m\circ\operatorname{c}^m}$};
    \node (A2_3) at (2.8, 2) {$\Downarrow\,i_{f^m}\ast\zeta^m$};

    \path (A4_2) edge [->]node [auto,swap] {$\scriptstyle{f^m\circ p^m}$} (A2_4);
    \path (A0_2) edge [->]node [auto] {$\scriptstyle{f^m\circ\operatorname{s}^m\circ h^m}$} (A2_4);
    \path (A0_2) edge [->]node [auto,swap] {$\scriptstyle{\operatorname{k}^m}$} (A2_0);
    \path (A2_2) edge [->]node [auto,swap] {$\scriptstyle{\operatorname{b}^m\circ
      \operatorname{j}^m\circ\operatorname{c}^m}$} (A0_2);
    \path (A2_2) edge [->]node [auto] {$\scriptstyle{\operatorname{k}^m
      \circ\operatorname{b}^m\circ\operatorname{j}^m\circ\operatorname{c}^m}$} (A4_2);
    \path (A4_2) edge [->]node [auto] {$\scriptstyle{\id_C}$} (A2_0);
\end{tikzpicture}
\end{equation}
where $\zeta^m$ is the following composition:

\begin{equation}\label{eq-192}
\begin{tikzpicture}[xscale=3.2,yscale=-1.2]
    \node (A1_2) at (2, 1) {$\widetilde{B}^m$};
    \node (A2_0) at (0, 2) {$L^m$};
    \node (A2_1) at (1, 2) {$B^{\prime m}$};
    \node (A2_2) at (2, 2) {$\widetilde{B}^m$};
    \node (A2_3) at (3, 2) {$B^m$.};
    \node (A3_1) at (1, 3) {$F^m$};
    \node (A3_2) at (2, 3) {$C$};
    \node (A1_1) at (1, 1) {$F^m$};

    \node (A3_3) at (2, 2.5) {$\Downarrow\,(\eta^m)^{-1}$};
    \node (B1_1) at (1, 1.5) {$\Downarrow\,\phi^m$};
    \node (A1_3) at (2, 1.55) {$\Downarrow\,(\varepsilon^m)^{-1}$};
    \node (A3_0) at (0.9, 2.4) {$\Downarrow\,(\phi^m)^{-1}$};
    
    \path (A2_0) edge [->]node [auto] {$\scriptstyle{\operatorname{b}^m\circ
      \operatorname{j}^m\circ\operatorname{c}^m}$} (A1_1);
    \path (A1_1) edge [->]node [auto] {$\scriptstyle{h^m}$} (A1_2);
    \path (A1_2) edge [->]node [auto] {$\scriptstyle{\operatorname{s}^m}$} (A2_3);
    \path (A2_0) edge [->]node [auto] {$\scriptstyle{o^m}$} (A2_1);
    \path (A2_0) edge [->]node [auto,swap] {$\scriptstyle{\operatorname{b}^m
      \circ\operatorname{j}^m\circ\operatorname{c}^m}$} (A3_1);
    \path (A3_2) edge [->]node [auto,swap] {$\scriptstyle{p^m}$} (A2_3);
    \path (A2_2) edge [->]node [auto] {$\scriptstyle{\operatorname{t}^m}$} (A2_3);
    \path (A3_1) edge [->]node [auto] {$\scriptstyle{h^m}$} (A2_2);
    \path (A2_1) edge [->]node [auto] {$\scriptstyle{\operatorname{r}^m}$} (A1_2);
    \path (A2_1) edge [->]node [auto] {$\scriptstyle{\operatorname{r}^m}$} (A2_2);
    \path (A3_1) edge [->]node [auto,swap] {$\scriptstyle{\operatorname{k}^m}$} (A3_2);
\end{tikzpicture}
\end{equation}

Now using \eqref{eq-190}, \eqref{eq-189} and \eqref{eq-188} we get that for each $m=1,2$
the $2$-morphism

\begin{equation}\label{eq-194}
F^m:=\Big(\Big(i_{g^m}\ast\Xi^m\Big)\ast i_{r^m}\Big)\odot\Big(\Thetab{g^m}{e^m}{d^m}\ast
i_{r^m}\Big)\odot\Thetaa{g^m\circ e^m}{d^m}{r^m}
\end{equation}
is represented by the following diagram

\[
\begin{tikzpicture}[xscale=2.8,yscale=-0.8]
    \node (A0_2) at (2, 0) {$C^m$};
    \node (A2_0) at (0, 2) {$C$};
    \node (A2_2) at (2, 2) {$L^m$};
    \node (A2_4) at (4, 2) {$A$.};
    \node (A4_2) at (2, 4) {$C$};

    \node (A2_3) at (2.85, 2) {$\Downarrow\,i_{f^m}\ast(\zeta^m
      \odot(\rho^m\ast i_{\operatorname{c}^m}))$};
    \node (A2_1) at (1.2, 2) {$\Downarrow\,\gamma^m
      \ast i_{\operatorname{j}^m\circ\operatorname{c}^m}$};

    \node (E1_1) at (2.25, 0.9) {$\scriptstyle{\operatorname{u}^m\circ
      \operatorname{z}^m\circ}$};
    \node (E2_2) at (2.35, 1.3) {$\scriptstyle{\circ\operatorname{a}^m
      \circ\operatorname{j}^m\circ\operatorname{c}^m}$};
      
    \path (A4_2) edge [->]node [auto,swap] {$\scriptstyle{f^m\circ p^m}$} (A2_4);
    \path (A0_2) edge [->]node [auto] {$\scriptstyle{f^m\circ q^m}$} (A2_4);
    \path (A4_2) edge [->]node [auto] {$\scriptstyle{\id_C}$} (A2_0);
    \path (A0_2) edge [->]node [auto,swap] {$\scriptstyle{\operatorname{v}^m}$} (A2_0);
    \path (A2_2) edge [->]node [auto,swap] {} (A0_2);
    \path (A2_2) edge [->]node [auto] {$\scriptstyle{\operatorname{k}^m
      \circ \operatorname{b}^m\circ\operatorname{j}^m\circ\operatorname{c}^m}$} (A4_2);
\end{tikzpicture}
\]

Using the definition of $2$-morphism in a bicategory of fractions (see Appendix~\ref{sec-06}),
we get easily that $F^m$ is also represented by the following diagram

\begin{equation}\label{eq-80}
\begin{tikzpicture}[xscale=2.8,yscale=-0.8]
    \node (A0_2) at (2, 0) {$C^m$};
    \node (A2_0) at (0, 2) {$C$};
    \node (A2_2) at (2, 2) {$L^m$};
    \node (A2_4) at (4, 2) {$A$,};
    \node (A4_2) at (2, 4) {$C$};

    \node (A2_3) at (2.85, 2) {$\Downarrow\,i_{f^m}\ast\chi^m$};
    \node (A2_1) at (1.15, 2) {$\Downarrow\,i_{\operatorname{v}^m\circ\operatorname{u}^m
      \circ\operatorname{z}^m\circ\operatorname{a}^m\circ
      \operatorname{j}^m\circ\operatorname{c}^m}$};

    \node (E1_1) at (2.25, 0.9) {$\scriptstyle{\operatorname{u}^m\circ
      \operatorname{z}^m\circ}$};
    \node (E2_2) at (2.35, 1.3) {$\scriptstyle{\circ\operatorname{a}^m
      \circ\operatorname{j}^m\circ\operatorname{c}^m}$};
    \node (E3_3) at (2.38, 2.7) {$\scriptstyle{\operatorname{v}^m
      \circ\operatorname{u}^m\circ\operatorname{z}^m\circ}$};
    \node (E4_4) at (2.35, 3.1) {$\scriptstyle{\circ\operatorname{a}^m\circ
      \operatorname{j}^m\circ\operatorname{c}^m}$};
    
    \path (A4_2) edge [->]node [auto,swap] {$\scriptstyle{f^m\circ p^m}$} (A2_4);
    \path (A0_2) edge [->]node [auto] {$\scriptstyle{f^m\circ q^m}$} (A2_4);
    \path (A4_2) edge [->]node [auto] {$\scriptstyle{\id_C}$} (A2_0);
    \path (A0_2) edge [->]node [auto,swap] {$\scriptstyle{\operatorname{v}^m}$} (A2_0);
    \path (A2_2) edge [->]node [auto,swap] {} (A0_2);
    \path (A2_2) edge [->]node [auto] {} (A4_2);
\end{tikzpicture}
\end{equation}
where

\[\chi^m:=\Big(i_{p^m}\ast\left(\gamma^m\right)^{-1}\ast i_{\operatorname{j}^m\circ
\operatorname{c}^m}\Big)\odot\zeta^m\odot\Big(\rho^m\ast i_{\operatorname{c}^m}\Big).\]

Now we want to write $F^m$ in a shorter form. As a preliminary step, we want to compute
$i_{\operatorname{w}^m}\ast\chi^m$; for that, we replace $i_{\operatorname{w}^m}\ast\rho^m$
with \eqref{eq-191} and $\zeta^m$ with \eqref{eq-192}. So we get that $i_{\operatorname{w}^m}
\ast\chi^m$ coincides with the following composition:

\begin{equation}\label{eq-140}
\begin{tikzpicture}[xscale=2.3,yscale=-1.2]
    \node (A0_2) at (2, 0) {$C^m$};
    \node (A0_4) at (4, 0) {$B^m$};
    \node (A1_0) at (1, 2) {$L^m$};
    \node (A1_1) at (1, 1) {$G^m$};
    \node (B1_1) at (1, 4) {$G^m$};
    \node (B1_2) at (2, 2) {$B^{\prime m}$};
    \node (A1_3) at (3, 1) {$C$};
    \node (A1_4) at (4, 1) {$B^m$};
    \node (A1_5) at (5, 1) {$\overline{B}^m$.};
    \node (A2_2) at (1.8, 1) {$F^m$};
    \node (A2_3) at (3, 2) {$\widetilde{B}^m$};
    \node (B2_3) at (3, 3) {$\widetilde{B}^m$};
    \node (A2_4) at (4, 2) {$B^m$};
    \node (C1_1) at (2, 4) {$F^m$};
    \node (C2_1) at (4, 4) {$C$};
    \node (D1_1) at (2.5, 5) {$C^m$};
    
    \node (E1_1) at (1.8, 1.5) {$\Downarrow\,\phi^m$};
    \node (E2_2) at (1.7, 3) {$\Downarrow\,(\phi^m)^{-1}$};
    \node (E3_3) at (3, 2.45) {$\Downarrow\,(\varepsilon^m)^{-1}$};
    \node (E4_4) at (3.5, 3.35	) {$\Downarrow\,(\eta^m)^{-1}$};
    \node (E5_5) at (2.5, 4.4) {$\Downarrow\,(\gamma^m)^{-1}$};
    \node (A2_5) at (4, 1.5) {$\Downarrow\,\xi^m$};
    \node (A2_1) at (3, 1.4) {$\Downarrow\,\eta^m$};
    \node (A0_1) at (2, 0.55) {$\Downarrow\,\gamma^m$};
    \node (A0_3) at (3.5, 0.4) {$\Downarrow\,(\sigma^m)^{-1}$};

    \path (B1_1) edge [->]node [auto,swap] {$\scriptstyle{\operatorname{u}^m
      \circ\operatorname{z}^m\circ\operatorname{a}^m}$} (D1_1);
    \path (D1_1) edge [->]node [auto,swap] {$\scriptstyle{\operatorname{v}^m}$} (C2_1);
    \path (B1_1) edge [->]node [auto] {$\scriptstyle{\operatorname{b}^m}$} (C1_1);
    \path (C2_1) edge [->]node [auto,swap] {$\scriptstyle{\operatorname{p}^m}$} (A2_4);
    \path (C1_1) edge [->]node [auto] {$\scriptstyle{\operatorname{k}^m}$} (C2_1);
    \path (C1_1) edge [->]node [auto] {$\scriptstyle{h^m}$} (B2_3);
    \path (B2_3) edge [->]node [auto,swap] {$\scriptstyle{\operatorname{t}^m}$} (A2_4);
    \path (B1_2) edge [->]node [auto,swap] {$\scriptstyle{\operatorname{r}^m}$} (A2_3);
    \path (B1_2) edge [->]node [auto,swap] {$\scriptstyle{\operatorname{r}^m}$} (B2_3);
    \path (A1_0) edge [->]node [auto] {$\scriptstyle{\operatorname{o}^m}$} (B1_2);
    \path (A2_4) edge [->]node [auto,swap] {$\scriptstyle{\operatorname{w}^m}$} (A1_5);
    \path (A1_4) edge [->]node [auto] {$\scriptstyle{\operatorname{w}^m}$} (A1_5);
    \path (A1_1) edge [->]node [auto] {$\scriptstyle{\operatorname{u}^m\circ
      \operatorname{z}^m\circ\operatorname{a}^m}$} (A0_2);
    \path (A0_4) edge [->]node [auto] {$\scriptstyle{\operatorname{w}^m}$} (A1_5);
    \path (A1_3) edge [->]node [auto] {$\scriptstyle{p^m}$} (A1_4);
    \path (A1_0) edge [->]node [auto] {$\scriptstyle{\operatorname{j}^m
      \circ\operatorname{c}^m}$} (A1_1);
    \path (A1_0) edge [->]node [auto,swap] {$\scriptstyle{\operatorname{j}^m
      \circ\operatorname{c}^m}$} (B1_1);
    \path (A2_2) edge [->]node [auto] {$\scriptstyle{\operatorname{k}^m}$} (A1_3);
    \path (A1_1) edge [->]node [auto,swap] {$\scriptstyle{\operatorname{b}^m}$} (A2_2);
    \path (A2_3) edge [->]node [auto,swap] {$\scriptstyle{\operatorname{s}^m}$} (A2_4);
    \path (A0_2) edge [->]node [auto] {$\scriptstyle{q^m}$} (A0_4);
    \path (A0_2) edge [->]node [auto] {$\scriptstyle{\operatorname{v}^m}$} (A1_3);
    \path (A2_2) edge [->]node [auto] {$\scriptstyle{h^m}$} (A2_3);
    \path (A2_3) edge [->]node [auto] {$\scriptstyle{\operatorname{t}^m}$} (A1_4);
\end{tikzpicture}
\end{equation}

In such a diagram, using \eqref{eq-181} we can replace the composition of $\xi^m$ and
$(\varepsilon^m)^{-1}$ with a $2$-identity. Then we can simply the terms $\phi^m$, $\eta^m$ and
$\gamma^m$ (in this order) with their inverses. So we get:

\[i_{\operatorname{w}^m}\ast\chi^m=\left(\sigma^m\right)^{-1}\ast i_{\operatorname{u}^m\circ
\operatorname{z}^m\circ\operatorname{a}^m\circ\operatorname{j}^m\circ\operatorname{c}^m}
=i_{\operatorname{w}^m}\ast\left(\tau^m\right)^{-1}\ast i_{\operatorname{z}^m\circ\operatorname{a}^m
\circ\operatorname{j}^m\circ\operatorname{c}^m},\]
where the last identity is a consequence of hypothesis (i). So by~\cite[Lemma~1.1]{T3} there are
an object $H^m$ and a morphism $\operatorname{y}^m:H^m\rightarrow L^m$, such that 

\begin{equation}\label{eq-131}
\chi^m\ast i_{\operatorname{y}^m}=\left(\tau^m\right)^{-1}\ast i_{\operatorname{z}^m\circ\operatorname{a}^m\circ
\operatorname{j}^m\circ\operatorname{c}^m\circ\operatorname{y}^m}.
\end{equation}

So for each $m=1,2$ we have:

\begin{gather}
\nonumber F^m\stackrel{\eqref{eq-80},\eqref{eq-131}}{=}\Big[H^m,\operatorname{u}^m\circ
 \operatorname{z}^m
 \circ\operatorname{a}^m\circ\operatorname{j}^m\circ\operatorname{c}^m\circ\operatorname{y}^m,
 \operatorname{v}^m\circ\operatorname{u}^m\circ\operatorname{z}^m
 \circ\operatorname{a}^m\circ\operatorname{j}^m\circ\operatorname{c}^m\circ\operatorname{y}^m,\\
%%%
\nonumber i_{\operatorname{v}^m\circ\operatorname{u}^m\circ\operatorname{z}^m
 \circ\operatorname{a}^m\circ\operatorname{j}^m\circ\operatorname{c}^m\circ\operatorname{y}^m},
 i_{f^m}\ast\left(\tau^m\right)^{-1}\ast i_{\operatorname{z}^m\circ\operatorname{a}^m\circ
 \operatorname{j}^m\circ\operatorname{c}^m\circ\operatorname{y}^m}\Big]= \\
%%%
\label{eq-196} =\Big[C'',\operatorname{u}^m\circ\operatorname{z}^m,
 \operatorname{v}^m\circ\operatorname{u}^m\circ\operatorname{z}^m,
 i_{\operatorname{v}^m\circ\operatorname{u}^m\circ\operatorname{z}^m},
 i_{f^m}\ast\left(\tau^m\right)^{-1}\ast i_{\operatorname{z}^m}\Big].
\end{gather}

Now using \eqref{eq-194} together with \eqref{eq-182}, we get that $\overline{\Omega}=
(F^2)^{-1}\odot\Omega\odot F^1$. Using \eqref{eq-196} for $m=1$ and
\eqref{eq-193}, we get that $\Omega\odot F^1$ is represented by the following diagram:

\begin{equation}\label{eq-133}
\begin{tikzpicture}[xscale=2.6,yscale=-0.8]
    \node (A0_2) at (2, 0) {$C^1$};
    \node (A2_0) at (0, 2) {$C$};
    \node (A2_2) at (2, 2) {$C''$};
    \node (A2_4) at (4, 2) {$A$,};
    \node (A4_2) at (2, 4) {$C$};

    \node (A2_3) at (2.6, 2) {$\Downarrow\,\kappa^1$};
    \node (A2_1) at (1.2, 2) {$\Downarrow\,i_{\operatorname{v}^1\circ
      \operatorname{u}^1\circ\operatorname{z}^1}$};

    \path (A4_2) edge [->]node [auto,swap] {$\scriptstyle{f^2\circ p^2}$} (A2_4);
    \path (A0_2) edge [->]node [auto] {$\scriptstyle{f^1\circ q^1}$} (A2_4);
    \path (A4_2) edge [->]node [auto] {$\scriptstyle{\id_C}$} (A2_0);
    \path (A0_2) edge [->]node [auto,swap] {$\scriptstyle{\operatorname{v}^1}$} (A2_0);
    \path (A2_2) edge [->]node [auto,swap] {$\scriptstyle{\operatorname{u}^1
      \circ\operatorname{z}^1}$} (A0_2);
    \path (A2_2) edge [->]node [auto] {$\scriptstyle{\operatorname{v}^1
      \circ\operatorname{u}^1\circ\operatorname{z}^1}$} (A4_2);
\end{tikzpicture}
\end{equation}
where

\[\kappa^1:=\Big(\omega\ast i_{\operatorname{v}^1\circ\operatorname{u}^1\circ\operatorname{z}^1}\Big)
\odot\Big(i_{f^1}\ast(\tau^1)^{-1}\ast i_{\operatorname{z}^1}\Big).\]

Using the inverse of \eqref{eq-196} for $m=2$ and the choices in (ii)
in the claim, we get that $F^2$ is represented by the following diagram:

\begin{equation}\label{eq-142}
\begin{tikzpicture}[xscale=2.6,yscale=-0.8]
    \node (A0_2) at (2, 0) {$C$};
    \node (A2_0) at (0, 2) {$C$};
    \node (A2_2) at (2, 2) {$C''$};
    \node (A2_4) at (4, 2) {$A$,};
    \node (A4_2) at (2, 4) {$C^2$};

    \node (A2_3) at (2.6, 2) {$\Downarrow\,\kappa^2$};
    \node (A2_1) at (1.2, 2) {$\Downarrow\,\mu$};

    \path (A4_2) edge [->]node [auto,swap] {$\scriptstyle{f^2\circ q^2}$} (A2_4);
    \path (A0_2) edge [->]node [auto] {$\scriptstyle{f^1\circ p^2}$} (A2_4);
    \path (A4_2) edge [->]node [auto] {$\scriptstyle{\operatorname{v}^2}$} (A2_0);
    \path (A0_2) edge [->]node [auto,swap] {$\scriptstyle{\id_C}$} (A2_0);
    \path (A2_2) edge [->]node [auto,swap] {$\scriptstyle{\operatorname{v}^1\circ\operatorname{u}^1
      \circ\operatorname{z}^1}$} (A0_2);
    \path (A2_2) edge [->]node [auto] {$\scriptstyle{\operatorname{u}^2\circ
      \operatorname{z}^2}$} (A4_2);
\end{tikzpicture}
\end{equation}
where

\[\kappa^2:=\Big(i_{f^2}\ast\tau^2\ast i_{\operatorname{z}^2}\Big)\odot\Big(i_{f^2\circ p^2}
\ast\mu\Big).\]

Lastly, using \eqref{eq-133}, \eqref{eq-142} and~\cite[Proposition~0.2]{T3}, we get
that the $2$-morphisms
$\overline{\Omega}=(F^2)^{-1}\odot\Omega\odot F^1$ coincides with \eqref{eq-139}, so we conclude.
\end{proof}

Now we want to prove Theorem~\ref{theo-04}, so the problem that we have to solve is the following:
given any set of data in $\CATC$ as follows

\begin{equation}\label{eq-92}
\begin{tikzpicture}[xscale=1.5,yscale=-0.8]
    \node (A0_0) at (0, 0) {$C$};
    \node (A0_2) at (2, 0) {$B^1$};
    \node (A2_0) at (0, 2) {$B^2$};
    \node (A2_2) at (2, 2) {$A$};

    \node (A1_1) [rotate=225] at (0.9, 1) {$\Longrightarrow$};
    \node (B1_1) at (1.2, 1) {$\omega$};

    \path (A2_0) edge [->]node [auto,swap] {$\scriptstyle{f^2}$} (A2_2);
    \path (A0_0) edge [->]node [auto,swap] {$\scriptstyle{p^2}$} (A2_0);
    \path (A0_2) edge [->]node [auto] {$\scriptstyle{f^1}$} (A2_2);
    \path (A0_0) edge [->]node [auto] {$\scriptstyle{p^1}$} (A0_2);
\end{tikzpicture}
\end{equation}
with $\omega$ invertible, when is the associated diagram \eqref{eq-91} a weak fiber product in
$\CATC\left[\SETWinv\right]$? In the next $2$ sections we will consider separately conditions
\textbf{\hyperref[A1]{A1}} and \textbf{\hyperref[A2]{A2}}
for \eqref{eq-91} and we will
manage to give equivalent but simple conditions for both properties.

\section{Condition A1 in a bicategory of fractions}
\begin{lem}\label{lem-11}
Let us fix any pair $(\CATC,\SETW)$ satisfying conditions \emphatic{(\hyperref[BF]{BF})}
and any bicategory of fractions $\CATC\left[\SETWinv\right]$ associated to it \emphatic{(}i.e.\ any
set of choices \emphatic{\hyperref[C]{C}}$(\SETW)$\emphatic{)}. Moreover,
let us also fix any set of data in $\CATC$ as in diagram \eqref{eq-92} with $\omega$ invertible. 
Then the following facts are equivalent:

\begin{enumerate}[\emphatic{(}{i}1\emphatic{)}]
 \item for any object $D$, condition \emphatic{\textbf{\hyperref[A1]{A1}}}$(D)$
  holds for diagram \eqref{eq-91} in $\CATC\left[\SETWinv\right]$;
 \item for any object $D$ and for any pair of morphisms $q^1:D\rightarrow B^1$, $q^2:D\rightarrow
  B^2$ in $\CATC$, condition \emphatic{\textbf{\hyperref[B1]{B1}}}$(D,(D,\id_D,q^1),(D,\id_D,q^2))$
  holds for diagram \eqref{eq-91}.
\end{enumerate}
\end{lem}

\begin{proof}
For simplicity of exposition, let us suppose that $\CATC$ is a $2$-category.\\

We recall from Remark~\ref{rem-02} that (i1) is equivalent to

\begin{enumerate}[({i}1)$'$]
 \item for any object $D$, and for any pair of morphisms $s^1:D\rightarrow B^1$ and $s^2:D\rightarrow
  B^2$ in $\CATC\left[\SETWinv\right]$, property \textbf{\hyperref[B1]{B1}}$(D,s^1,s^2)$
  holds for diagram \eqref{eq-91}.
\end{enumerate}

Clearly (i1)$'$ implies (i2): it is simply the case when $s^m:=(D,\id_D,q^m)$
for $m=1,2$. Let us assume that (i2) holds and let us prove (i1)$'$. So let us fix any
object $D$ in $\CATC$ and any pair of morphisms
$s^1:D\rightarrow B^1$ and $s^2:D\rightarrow B^2$ in
$\CATC\left[\SETWinv\right]$. By definition of morphisms in $\CATC\left[\SETWinv\right]$,
for each $m=1,2$ there are an object $D^m$, a morphism
$\operatorname{w}^m:D^m\rightarrow D$ in $\SETW$ and a morphism
$t^m:D^m\rightarrow B^m$ in $\CATC$, such that $s^m=(D^m,
\operatorname{w}^m,t^m)$. Now we use (\hyperref[BF3]{BF3}) in order to get data as in the upper
part of the following diagram, with $\operatorname{v}^2$ in $\SETW$ and $\alpha$ invertible:

\[
\begin{tikzpicture}[xscale=-2.2,yscale=-0.8]
    \node (A0_1) at (1, 0) {$D^3$};
    \node (A1_0) at (0.3, 2) {$D^1$.};
    \node (A1_2) at (1.7, 2) {$D^2$};
    \node (A2_1) at (1, 2) {$D$};
    
    \node (A1_1) at (1, 1) {$\alpha$};
    \node (B1_1) at (1, 1.4) {$\Rightarrow$};
    
    \path (A1_2) edge [->]node [auto,swap] {$\scriptstyle{\operatorname{w}^2}$} (A2_1);
    \path (A0_1) edge [->]node [auto,swap] {$\scriptstyle{\operatorname{v}^2}$} (A1_2);
    \path (A1_0) edge [->]node [auto] {$\scriptstyle{\operatorname{w}^1}$} (A2_1);
    \path (A0_1) edge [->]node [auto] {$\scriptstyle{\operatorname{v}^1}$} (A1_0);
\end{tikzpicture}
\]

Moreover, let us suppose that the fixed choices \hyperref[C]{C}$(\SETW)$ give data as in the
upper part of the following diagram, with $\operatorname{r}$ in $\SETW$ and $\varepsilon$ invertible:

\begin{equation}\label{eq-158}
\begin{tikzpicture}[xscale=2.4,yscale=-0.8]
    \node (A0_1) at (1, 0) {$D^4$};
    \node (A1_0) at (0.3, 2) {$D^3$};
    \node (A1_2) at (1.7, 2) {$D^3$};
    \node (A2_1) at (1, 2) {$D$};
    
    \node (A1_1) at (1, 1) {$\varepsilon$};
    \node (B1_1) at (1, 1.4) {$\Rightarrow$};
    
    \path (A1_2) edge [->]node [auto] {$\scriptstyle{\operatorname{w}^2
      \circ\operatorname{v}^2}$} (A2_1);
    \path (A0_1) edge [->]node [auto] {$\scriptstyle{\operatorname{q}}$} (A1_2);
    \path (A1_0) edge [->]node [auto,swap] {$\scriptstyle{\operatorname{w}^2\circ
      \operatorname{v}^2}$} (A2_1);
    \path (A0_1) edge [->]node [auto,swap] {$\scriptstyle{\operatorname{r}}$} (A1_0);
\end{tikzpicture}
\end{equation}
(since the choices \hyperref[C]{C}$(\SETW)$ are arbitrary, then we cannot assume that condition
(\hyperref[C3]{C3}) holds, so we cannot say anything more about the data above).
By construction and (\hyperref[BF2]{BF2}), the morphism $\operatorname{w}^2\circ\operatorname{v}^2$
belongs to $\SETW$. So
using (\hyperref[BF4a]{BF4a}) and (\hyperref[BF4b]{BF4b}), there are an object $D^5$, a morphism
$\operatorname{h}:D^5\rightarrow D^4$ in $\SETW$ and an invertible $2$-morphism $\eta:
\operatorname{r}\circ\operatorname{h}\Rightarrow\operatorname{q}\circ\operatorname{h}$, such that
$\varepsilon\ast i_{\operatorname{h}}=i_{\operatorname{w}^2\circ\operatorname{v}^2}\ast\eta$.
Since we are assuming (i2), then condition \textbf{\hyperref[B1]{B1}}$(D^3,(D^3,\id_{D^3},t^1
\circ\operatorname{v}^1),(D^3,\id_{D^3},t^2
\circ\operatorname{v}^2))$ holds for \eqref{eq-91}. Then for each $m=1,2$
we consider the invertible $2$-morphism

\begin{equation}\label{eq-121}
\Omega^m:\Big(D^3,\id_{D^3},t^m\circ\operatorname{v}^m\Big)\Longrightarrow
\Big(D^4,\operatorname{r},t^m\circ\operatorname{v}^m\circ\operatorname{q}\Big)
\end{equation}
represented by the following diagram:

\[
\begin{tikzpicture}[xscale=2.2,yscale=-0.8]
    \node (A0_2) at (2, 0) {$D^3$};
    \node (A2_0) at (0, 2) {$D^3$};
    \node (A2_2) at (2, 2) {$D^5$};
    \node (A2_4) at (4, 2) {$B^m$.};
    \node (A4_2) at (2, 4) {$D^4$};

    \node (A2_1) at (1.2, 2) {$\Downarrow\,\eta^{-1}$};
    \node (A2_3) at (2.8, 2) {$\Downarrow\,i_{t^m\circ\operatorname{v}^m\circ\operatorname{q}
      \circ\operatorname{h}}$};
    
    \path (A0_2) edge [->]node [auto,swap] {$\scriptstyle{\id_{D^3}}$} (A2_0);
    \path (A4_2) edge [->]node [auto] {$\scriptstyle{\operatorname{r}}$} (A2_0);
    \path (A4_2) edge [->]node [auto,swap] {$\scriptstyle{t^m\circ\operatorname{v}^m
      \circ\operatorname{q}}$} (A2_4);
    \path (A2_2) edge [->]node [auto,swap] {$\scriptstyle{\operatorname{q}
      \circ\operatorname{h}}$} (A0_2);
    \path (A2_2) edge [->]node [auto] {$\scriptstyle{\operatorname{h}}$} (A4_2);
    \path (A0_2) edge [->]node [auto] {$\scriptstyle{t^m\circ\operatorname{v}^m}$} (A2_4);
\end{tikzpicture}
\]

Then using the equivalence of (a) and (b) in Proposition~\ref{prop-03} and \eqref{eq-121}, we get that
condition \textbf{\hyperref[B1]{B1}}$(D^3,(D^4,\operatorname{r},t^1
\circ\operatorname{v}^1\circ\operatorname{q}),(D^4,\operatorname{r},t^2
\circ\operatorname{v}^2\circ\operatorname{q}))$ holds for \eqref{eq-91}. Now using 
\eqref{eq-158}, for each $m=1,2$ we have

\begin{equation}\label{eq-122}
\Big(D^3,\operatorname{w}^2\circ\operatorname{v}^2,
t^m\circ\operatorname{v}^m\Big)\circ\Big(D^3,\id_{D^3},\operatorname{w}^2\circ\operatorname{v}^2
\Big)=\Big(D^4,\operatorname{r},t^m
\circ\operatorname{v}^m\circ\operatorname{q}\Big).
\end{equation}

Since $\operatorname{w}^2\circ\operatorname{v}^2$ belongs to $\SETW$, then the morphism
$e:=(D^3,\id_{D^3},\operatorname{w}^2\circ\operatorname{v}^2)$ is an internal equivalence
in $\CATC\left[\SETWinv\right]$ (see~\cite[Proposition~20]{Pr}).
Therefore, using the equivalence of (a) and (c) in Proposition~\ref{prop-03}, and
\eqref{eq-122}, we get that
\textbf{\hyperref[B1]{B1}}$(D,(D^3,\operatorname{w}^2\circ\operatorname{v}^2,t^1
\circ\operatorname{v}^1),(D^3,\operatorname{w}^2\circ\operatorname{v}^2,t^2\circ\operatorname{v}^2))$
holds for \eqref{eq-91}. Now we consider the pair of invertible $2$-morphisms

\begin{gather*}
\widetilde{\Omega}^1:=\Big[D^3,\id_{D^3},\operatorname{v}^1,\alpha,i_{t^1\circ\operatorname{v}^1}
 \Big]:\Big(D^3,\operatorname{w}^2\circ\operatorname{v}^2,t^1
 \circ\operatorname{v}^1\Big)\Longrightarrow\Big(D^1,\operatorname{w}^1,t^1\Big)=s^1, \\
%%%
\widetilde{\Omega}^2:=\Big[D^3,\id_{D^3},\operatorname{v}^2,i_{\operatorname{w}^2
 \circ\operatorname{v}^2},i_{t^2\circ\operatorname{v}^2}
 \Big]:\Big(D^3,\operatorname{w}^2\circ\operatorname{v}^2,t^2
 \circ\operatorname{v}^2\Big)\Longrightarrow\Big(D^2,\operatorname{w}^2,t^2\Big)=s^2.
\end{gather*}

Using the equivalence of (a) and (b) in Proposition~\ref{prop-03}, we get that
\textbf{\hyperref[B1]{B1}}$(D,s^1,s^2)$ holds for \eqref{eq-91}, i.e.\ (i1)$'$ holds.
\end{proof}

\begin{lem}\label{lem-15}
Let us fix the same notations of \emphatic{Lemma~\ref{lem-11}}. Then
the following facts are equivalent:

\begin{enumerate}[\emphatic{(}{i}1\emphatic{)}]
\setcounter{enumi}{1}
 \item for any object $D$ and for any pair of morphisms $q^1:D\rightarrow B^1$, $q^2:D\rightarrow B^2$
  in $\CATC$, condition \emphatic{\textbf{\hyperref[B1]{B1}}}$(D,(D,\id_D,q^1),(D,\id_D,q^2))$ holds
  for diagram \eqref{eq-91};
 \item for any object $D$ of $\CATC$ the following condition holds:
 
  \begin{enumerate}[\emphatic{(}a\emphatic{)}]
   \item given any pair of morphisms $q^m:D\rightarrow B^m$ for $m=1,2$
     and any invertible $2$-morphism $\lambda:f^1\circ q^1\Rightarrow f^2\circ q^2$ in
     $\CATC$, there are an object $E$, a morphism $\operatorname{v}:E\rightarrow D$ in $\SETW$, a 
     morphism $q:E\rightarrow C$ and a pair of invertible $2$-morphisms $\lambda^m:q^m\circ
     \operatorname{v}\Rightarrow p^m\circ q$ for $m=1,2$ in $\CATC$, such that:

     \begin{gather*}
     \thetab{f^2}{p^2}{q}\odot\Big(\omega\ast i_{q}\Big)\odot\thetaa{f^1}{p^1}{q}\odot
      \Big(i_{f^1}\ast\lambda^1\Big)= \\
     =\Big(i_{f^2}\ast\lambda^2\Big)\odot\thetab{f^2}{q^2}{\operatorname{v}}\odot\Big(\lambda\ast
      i_{\operatorname{v}}\Big)\odot\thetaa{f^1}{q^1}{\operatorname{v}}.  
     \end{gather*}
  \end{enumerate}
\end{enumerate}
\end{lem}

\begin{proof}
As usual, we assume for simplicity that $\CATC$ is a $2$-category.
Let us suppose that (i2) holds and let us fix any quadruple $(D,q^1,q^2,\lambda)$ as in (i3). Then
we can consider a diagram as follows in $\CATC\left[\SETWinv\right]$:

\begin{equation}\label{eq-159}
\begin{tikzpicture}[xscale=2.8,yscale=-0.8]
    \node (A0_0) at (0, 0) {$D$};
    \node (A0_2) at (2, 0) {$B^1$};
    \node (A2_0) at (0, 2) {$B^2$};
    \node (A2_2) at (2, 2) {$A$.};
    
    \node (A1_1) [rotate=225] at (0.3, 1) {$\Longrightarrow$};
    \node (B0_0) at (1.1, 1) {$\Lambda:=[D,\id_D,\id_D,i_{\id_D},\lambda]$};
    
    \path (A0_0) edge [->]node [auto,swap] {$\scriptstyle{(D,\id_D,q^2)}$} (A2_0);
    \path (A0_0) edge [->]node [auto] {$\scriptstyle{(D,\id_D,q^1)}$} (A0_2);
    \path (A0_2) edge [->]node [auto] {$\scriptstyle{(B^1,\id_{B^1},f^1)}$} (A2_2);
    \path (A2_0) edge [->]node [auto,swap] {$\scriptstyle{(B^2,\id_{B^2},f^2)}$} (A2_2);
\end{tikzpicture}
\end{equation}

Since $\lambda$ is invertible in $\CATC$, then we get easily that $\Lambda$ is
invertible in $\CATC\left[\SETWinv\right]$, so by (i2) there are a morphism

\[
\begin{tikzpicture}[xscale=1.5,yscale=-1.2]
    \node (A0_0) at (-0.2, 0) {$s:=\Big(D$};
    \node (A0_1) at (1, 0) {$\overline{D}$};
    \node (A0_2) at (2.4, 0) {$C\Big):D\longrightarrow C$};
    
    \path (A0_1) edge [->]node [auto,swap] {$\scriptstyle{\operatorname{w}}$} (A0_0);
    \path (A0_1) edge [->]node [auto] {$\scriptstyle{r}$} (A0_2);
\end{tikzpicture}
\]
in $\CATC\left[\SETWinv\right]$ and a pair of invertible $2$-morphisms

\[\Lambda^m:\Big(D,\id_D,q^m\Big)\Longrightarrow\Big(C,\id_C,p^m\Big)\circ\Big(\overline{D},
\operatorname{w},r\Big)\]
for $m=1,2$ in $\CATC\left[\SETWinv\right]$, such that

\begin{gather}
\nonumber \Big(\Omega\ast i_{(\overline{D},\operatorname{w},r)}\Big)\odot\Thetaa{(B^1,\id_{B^1},f^1)}
 {(C,\id_C,p^1)}{(\overline{D},\operatorname{w},r)}\odot\Big(i_{(B^1,\id_{B^1},f^1)}\ast\Lambda^1
 \Big)= \\
%%%
\label{eq-49} =\Thetaa{(B^2,\id_{B^2},f^2)}{(C,\id_C,p^2)}{(\overline{D},\operatorname{w},r)}
 \odot\Big(i_{(B^2,\id_{B^2},f^2)}\ast\Lambda^2\Big)\odot\Lambda.  
\end{gather}

For each $m=1,2$, $\Lambda^m$ is defined from $(D,\id_D,
q^m)$ to $(\overline{D},\operatorname{w},p^m\circ r)$. Therefore by~\cite[Lemma~6.1]{T3} applied
to $\alpha:=i_{\operatorname{w}}$ and to $\Lambda^1$, there are an object $D^1$, a morphism
$\operatorname{u}^1:D^1\rightarrow\overline{D}$ such
that $\operatorname{w}\circ\operatorname{u}^1$ belongs to $\SETW$, and a $2$-morphism

\[\alpha^1:\,q^1\circ\operatorname{w}\circ\operatorname{u}^1\Longrightarrow p^1\circ r\circ
\operatorname{u}^1\]
in $\CATC$, such that

\[\Lambda^1=\Big[D^1,\operatorname{w}\circ\operatorname{u}^1,\operatorname{u}^1,i_{\operatorname{w}
\circ\operatorname{u}^1},\alpha^1\Big].\]

By~\cite[Proposition~0.8]{T3}, we can assume that $\alpha^1$ is
invertible in $\CATC$ since $\Lambda^1$ is invertible in $\CATC\left[\SETWinv\right]$.
By~\cite[Lemma~6.1]{T3} applied to $\alpha:=i_{\operatorname{w}\circ\operatorname{u}^1}$ and to
$\Lambda^2$, there are an object $D^2$, a morphism $\operatorname{u}^2:D^2\rightarrow D^1$ such that
$\operatorname{w}\circ\operatorname{u}^1\circ\operatorname{u}^2$ belongs to $\SETW$, and a
$2$-morphism

\[\alpha^2:\,q^2\circ\operatorname{w}\circ\operatorname{u}^1\circ\operatorname{u}^2\Longrightarrow
p^2\circ r\circ\operatorname{u}^1\circ\operatorname{u}^2\]
in $\CATC$, such that

\[\Lambda^2=\Big[D^2,\operatorname{w}\circ\operatorname{u}^1\circ\operatorname{u}^2,
\operatorname{u}^1\circ\operatorname{u}^2,i_{\operatorname{w}
\circ\operatorname{u}^1\circ\operatorname{u}^2},\alpha^2\Big].\]

As above, we can assume that $\alpha^2$ is invertible in $\CATC$ since $\Lambda^2$ is invertible
in $\CATC\left[\SETWinv\right]$. Now by Lemma~\ref{lem-12} (in the special case when $\CATC$ is a
$2$-category), we have:

\begin{equation}\label{eq-48}
i_{(B^1,\id_{B^1},f^1)}\ast\Lambda^1=
\Big[D^2,\operatorname{w}\circ\operatorname{u}^1\circ\operatorname{u}^2,\operatorname{u}^1
\circ\operatorname{u}^2,i_{\operatorname{w}
\circ\operatorname{u}^1\circ\operatorname{u}^2},i_{f^1}\ast\alpha^1\ast i_{\operatorname{u}^2}\Big]
\end{equation}
and

\begin{equation}\label{eq-110}
i_{(B^2,\id_{B^2},f^2)}\ast\Lambda^2=
\Big[D^2,\operatorname{w}\circ\operatorname{u}^1\circ\operatorname{u}^2,\operatorname{u}^1
\circ\operatorname{u}^2,i_{\operatorname{w}
\circ\operatorname{u}^1\circ\operatorname{u}^2},i_{f^2}\ast\alpha^2\Big].  
\end{equation}

Moreover, using \eqref{eq-91} and Lemma~\ref{lem-09} (in the special case when $\CATC$ is a
$2$-category), we have:

\begin{equation}\label{eq-50} 
\Omega\ast i_{(\overline{D},\operatorname{w},r)}=\Big[C,\id_C,\id_C,i_{\id_C},\omega\Big]
\ast i_{(\overline{D},\operatorname{w},r)}=\Big[\overline{D},\id_{\overline{D}},\id_{\overline{D}},
i_{\operatorname{w}},\omega\ast i_r\Big].
\end{equation}

In addition, by Lemma~\ref{lem-10} each $2$-morphism of the form $\Theta_{\bullet}$ in \eqref{eq-49}
is trivial. Therefore, by replacing \eqref{eq-159},
\eqref{eq-48}, \eqref{eq-110} and \eqref{eq-50} in \eqref{eq-49}, we get:

\begin{gather}
\nonumber \Big[\overline{D},\id_{\overline{D}},\id_{\overline{D}},i_{\operatorname{w}},
 \omega\ast i_r\Big]\odot\Big[D^2,\operatorname{w}\circ\operatorname{u}^1\circ\operatorname{u}^2,
 \operatorname{u}^1\circ\operatorname{u}^2,i_{\operatorname{w}\circ\operatorname{u}^1\circ
 \operatorname{u}^2},i_{f^1}\ast\alpha^1\ast i_{\operatorname{u}^2}\Big]= \\
%%%
\label{eq-33} =\Big[D^2,\operatorname{w}\circ\operatorname{u}^1\circ\operatorname{u}^2,
 \operatorname{u}^1\circ\operatorname{u}^2,i_{\operatorname{w}\circ\operatorname{u}^1\circ
 \operatorname{u}^2},i_{f^2}\ast\alpha^2\Big]\odot\Big[D,\id_D,\id_D,i_{\id_D},\lambda\Big].  
\end{gather}

This is equivalent to saying that

\begin{gather*}
\Big[D^2,\operatorname{w}\circ\operatorname{u}^1\circ\operatorname{u}^2,\operatorname{u}^1\circ
 \operatorname{u}^2,i_{\operatorname{w}\circ\operatorname{u}^1\circ\operatorname{u}^2},
 \Big(\omega\ast i_{r\circ\operatorname{u}^1\circ\operatorname{u}^2}\Big)\odot\Big(i_{f^1}\ast
 \alpha^1\ast i_{\operatorname{u}^2}\Big)\Big]= \\
%%%
=\Big[D^2,\operatorname{w}\circ\operatorname{u}^1\circ\operatorname{u}^2,\operatorname{u}^1\circ
 \operatorname{u}^2,i_{\operatorname{w}\circ\operatorname{u}^1\circ\operatorname{u}^2},
 \Big(i_{f^2}\ast\alpha^2\Big)\odot\Big(\lambda\ast i_{\operatorname{w}\circ\operatorname{u}^1\circ
 \operatorname{u}^2}\Big)\Big].
\end{gather*}

So by~\cite[Proposition~0.7]{T3} there are an object $E$ and a morphisms $\operatorname{u}^3:E
\rightarrow D^2$, such that $\operatorname{w}\circ\operatorname{u}^1\circ\operatorname{u}^2
\circ\operatorname{u}^3$ belongs to $\SETW$ and such that

\begin{equation}\label{eq-57}
\Big(\omega\ast i_{r\circ\operatorname{u}^1\circ\operatorname{u}^2\circ\operatorname{u}^3}\Big)
\odot\Big(i_{f^1}\ast\alpha^1\ast i_{\operatorname{u}^2\circ\operatorname{u}^3}\Big)=
\Big(i_{f^2}\ast\alpha^2\ast i_{\operatorname{u}^3}\Big)\odot\Big(\lambda\ast i_{\operatorname{w}
\circ\operatorname{u}^1\circ\operatorname{u}^2\circ\operatorname{u}^3}\Big).
\end{equation}

Now we define

\begin{gather*}
\operatorname{v}:=\operatorname{w}\circ\operatorname{u}^1\circ\operatorname{u}^2\circ
 \operatorname{u}^3:\,E\longrightarrow
 D,\quad\quad q:=r\circ\operatorname{u}^1\circ\operatorname{u}^2\circ\operatorname{u}^3:\,E
 \longrightarrow C, \\
%%%
\lambda^1:=\alpha^1\ast i_{\operatorname{u}^2\circ\operatorname{u}^3}:\,q^1
 \circ\operatorname{v}\Longrightarrow p^1\circ q,\quad\quad\lambda^2:=\alpha^2\ast
 i_{\operatorname{u}^3}:\,\,q^2\circ\operatorname{v}\Longrightarrow p^2\circ q.  
\end{gather*}

So \eqref{eq-57} reads as follows:

\[\Big(\omega\ast i_{q}\Big)\odot\Big(i_{f^1}\ast\lambda^1\Big)= 
\Big(i_{f^2}\ast\lambda^2\Big)\odot\Big(\lambda\ast i_{\operatorname{v}}\Big),\]

hence we have proved that (i2) implies (i3).\\

Conversely, let us assume that (i3) holds. Let us fix any object $D$, any pair of morphisms $q^1:
D\rightarrow B^1$, $q^2:D\rightarrow B^2$ in $\CATC$; then we have to prove that condition
\textbf{\hyperref[B1]{B1}}$(D,(D,\id_D,q^1),(D,\id_D,q^2))$ holds for diagram \eqref{eq-91}.
So let us fix any invertible $2$-morphism 

\begin{equation}\label{eq-111}
\Lambda:\Big(B^1,\id_{B^1},f^1\Big)\circ\Big(D,\id_D,q^1\Big)\Longrightarrow\Big(B^2,\id_{B^2},f^2
\Big)\circ\Big(D,\id_D,q^2\Big)
\end{equation}
in $\CATC\left[\SETWinv\right]$. By~\cite[Lemma~6.1]{T3} applied to $\alpha:=i_{\id_D}$
and $\Lambda$,
there are an object $\overline{D}$,
a morphism $\operatorname{w}:\overline{D}\rightarrow D$ in $\SETW$ and an invertible $2$-morphism 
$\lambda:f^1\circ q^1\circ\operatorname{w}\Rightarrow f^2\circ q^2\circ
\operatorname{w}$ in $\CATC$, such that

\begin{equation}\label{eq-113}
\Lambda=\Big[\overline{D},\operatorname{w},\operatorname{w},
i_{\operatorname{w}},\lambda\Big]:\,\Big(D,\id_D,f^1\circ q^1\Big)\Longrightarrow
\Big(D,\id_D,f^2\circ q^2\Big).
\end{equation}

Now we apply condition (i3) for the set of data $(\overline{D},q^1\circ\operatorname{w},q^2\circ
\operatorname{w},\lambda)$. Then there are an object $E$, a morphism
$\operatorname{v}:E\rightarrow\overline{D}$ in $\SETW$, a morphism
$q:E\rightarrow C$ and a pair of invertible $2$-morphisms

\[\lambda^m:\,q^m\circ\operatorname{w}\circ\operatorname{v}\Longrightarrow p^m\circ q\quad
\textrm{for }m=1,2,\]
in $\CATC$, such that 

\begin{equation}\label{eq-51}
\Big(\omega\ast i_{q}\Big)\odot\Big(i_{f^1}\ast\lambda^1\Big)=
\Big(i_{f^2}\ast\lambda^2\Big)\odot\Big(\lambda\ast i_{\operatorname{v}}\Big).  
\end{equation}

Now we consider the morphism $s:=(E,\operatorname{w}\circ\operatorname{v},q):D\rightarrow C$
in $\CATC\left[\SETWinv\right]$; moreover, for each $m=1,2$ we consider the invertible $2$-morphism

\[\Lambda^m:\Big(D,\id_D,q^m\Big)\Longrightarrow\Big(C,\id_C,p^m\Big)\circ s=
\Big(E,\operatorname{w}\circ\operatorname{v},p^m\circ q\Big),\]
represented by the data in the internal part of the following diagram

\begin{equation}\label{eq-114}
\begin{tikzpicture}[xscale=2.2,yscale=-0.8]
    \node (A0_2) at (2, 0) {$D$};
    \node (A2_2) at (2, 2) {$E$};
    \node (A2_0) at (0, 2) {$D$};
    \node (A2_4) at (4, 2) {$B^m$.};
    \node (A4_2) at (2, 4) {$E$};
    
    \node (A2_3) at (2.8, 2) {$\Downarrow\,\lambda^m$};
    \node (A2_1) at (1.2, 2) {$\Downarrow\,i_{\operatorname{w}\circ\operatorname{v}}$};

    \path (A4_2) edge [->]node [auto,swap] {$\scriptstyle{p^m\circ q}$} (A2_4);
    \path (A0_2) edge [->]node [auto] {$\scriptstyle{q^m}$} (A2_4);
    \path (A2_2) edge [->]node [auto,swap] {$\scriptstyle{\operatorname{w}
      \circ\operatorname{v}}$} (A0_2);
    \path (A2_2) edge [->]node [auto] {$\scriptstyle{\id_E}$} (A4_2);
    \path (A4_2) edge [->]node [auto] {$\scriptstyle{\operatorname{w}\circ\operatorname{v}}$} (A2_0);
    \path (A0_2) edge [->]node [auto,swap] {$\scriptstyle{\id_D}$} (A2_0);
\end{tikzpicture}
\end{equation}

By Lemma~\ref{lem-09} we have:

\begin{equation}\label{eq-76}
\Omega\ast i_s\stackrel{\eqref{eq-91}}{=}
\Big[C,\id_C,\id_C,i_{\id_C},\omega\Big]\ast i_{(E,\operatorname{w}\circ\operatorname{v},q)}=
\Big[E,\id_E,\id_E,i_{\operatorname{w}\circ\operatorname{v}},\omega\ast i_q\Big];  
\end{equation}
moreover, by Lemma~\ref{lem-12} we have the following formula for each $m=1,2$:

\begin{gather}
\nonumber i_{(B^m,\id_{B^m},f^m)}\ast\Lambda^m\stackrel{\eqref{eq-114}}{=}i_{(B^m,\id_{B^m},f^m)}
 \ast\Big[E,\operatorname{w}\circ\operatorname{v},\id_E,
 i_{\operatorname{w}\circ\operatorname{v}},\lambda^m\Big]= \\
%%%
\label{eq-64} =\Big[E,\operatorname{w}\circ\operatorname{v},
 \id_E,i_{\operatorname{w}\circ\operatorname{v}},i_{f^m}\ast
 \lambda^m\Big]:\,\Big(D,\id_D,f^m\circ q^m\Big)\Longrightarrow\Big(E,
 \operatorname{w}\circ\operatorname{v},f^m\circ p^m\circ q\Big).  
\end{gather}

Therefore, using \eqref{eq-76}, \eqref{eq-64} for $m=1$ and Lemma~\ref{lem-10}, we get:

\begin{gather}
\nonumber \Big(\Omega\ast i_s\Big)
 \odot\Thetaa{(B^1,\id_{B^1},f^1)}{(C,\id_C,p^1)}{s}\odot\Big(i_{(B^1,
 \id_{B^1},f^1)}\ast\Lambda^1\Big)= \\
%%%
\nonumber =\Big[E,
 \id_E,\id_E,i_{\operatorname{w}\circ
 \operatorname{v}},\omega\ast i_q\Big]\odot i_{\left(E,\operatorname{w}\circ\operatorname{v},
 f^1\circ p^1\circ q\right)}\odot \Big[E,\operatorname{w}\circ\operatorname{v},
 \id_E,i_{\operatorname{w}\circ\operatorname{v}},i_{f^1}\ast
 \lambda^1\Big]= \\
%%%
\label{eq-61}=\Big[E,\operatorname{w}\circ\operatorname{v},
 \id_E,i_{\operatorname{w}\circ\operatorname{v}},\Big(\omega\ast i_q\Big)\odot
 \Big(i_{f^1}\ast\lambda^1\Big)\Big].  
\end{gather}

Using \eqref{eq-64} for $m=2$, \eqref{eq-113} and Lemma~\ref{lem-10}, we have

\begin{gather}
\nonumber \Thetaa{(B^2,\id_{B^2},f^2)}{(C,\id_C,p^2)}{s}\odot\Big(i_{(B^2,\id_{B^2},f^2)}\ast
  \Lambda^2\Big)\odot\Lambda= \\
%%%
\nonumber =i_{\left(E,\operatorname{w}\circ\operatorname{v},f^2\circ p^2\circ q\right)}\odot
 \Big[E,\operatorname{w}\circ\operatorname{v},\id_E,i_{\operatorname{w}\circ\operatorname{v}},
 i_{f^2}\ast\lambda^2\Big]\odot\Big[E,
 \operatorname{w}\circ\operatorname{v},\operatorname{w}\circ\operatorname{v},i_{\operatorname{w}
 \circ\operatorname{v}},\lambda\ast i_{\operatorname{v}}\Big]= \\
%%%
\label{eq-08} =\Big[E,\operatorname{w}\circ\operatorname{v},\id_E,i_{\operatorname{w}\circ
 \operatorname{v}},\Big(i_{f^2}\ast\lambda^2\Big)\odot\Big(\lambda\ast i_{\operatorname{v}}\Big)
 \Big].  
\end{gather}

Then using \eqref{eq-51} we get that \eqref{eq-61} and \eqref{eq-08} coincide. So we conclude
that condition \textbf{\hyperref[B1]{B1}}$(D,(D,\id_D,q^1),(D,\id_D,q^2))$ holds for diagram
\eqref{eq-91} in $\CATC\left[\SETWinv\right]$, i.e.\ property (i2) is verified.
\end{proof}

\section{Condition A2 in a bicategory of fractions}
\begin{lem}\label{lem-02}
Let us fix the same notations of \emphatic{Lemma~\ref{lem-11}}. Then the
following facts are equivalent:

\begin{enumerate}[\emphatic{(}{ii}1\emphatic{)}]
 \item for any object $D$, condition \emphatic{\textbf{\hyperref[A2]{A2}}}$(D)$
  holds for diagram \eqref{eq-91} in $\CATC\left[\SETWinv\right]$;
 \item for any object $D$ and for any pair of morphisms $t,t':D\rightarrow C$ in $\CATC$, condition
  \emphatic{\textbf{\hyperref[B2]{B2}}}$(D,(D,\id_D,t),(D,\id_D,t'))$ holds for diagram \eqref{eq-91}.
\end{enumerate}
\end{lem}

The proof follows the same lines of the proof of Lemma~\ref{lem-11}, using
Proposition~\ref{prop-04} instead of Proposition~\ref{prop-03}, so we omit the details.

\begin{lem}\label{lem-01}
Let us fix the same notations of \emphatic{Lemma~\ref{lem-11}}. Then the following facts
are equivalent:

\begin{enumerate}[\emphatic{(}{ii}1\emphatic{)}]
\setcounter{enumi}{1}
 \item for any object $D$ and for any pair of morphisms $t,t':D\rightarrow C$ in $\CATC$, condition
  \emphatic{\textbf{\hyperref[B2]{B2}}}$(D,(D,\id_D,t),(D,\id_D,t'))$ holds for diagram \eqref{eq-91};
 \item for any object $D$, the following $2$ conditions hold:
  \begin{enumerate}[\emphatic{(}a\emphatic{)}]
   \setcounter{enumii}{1}
   \item given any pair of morphisms $t,t':D\rightarrow C$ and any pair of
    invertible $2$-morphisms $\gamma^m:p^m\circ t\Rightarrow p^m\circ t'$ for $m=1,2$ in $\CATC$,
    such that
    
    \begin{gather}
    \nonumber \thetab{f^2}{p^2}{t'}\odot\Big(\omega\ast i_{t'}\Big)\odot\thetaa{f^1}{p^1}{t'}\odot
     \Big(i_{f^1}\ast\gamma^1\Big)= \\
    \label{eq-65} =\Big(i_{f^2}\ast\gamma^2\Big)\odot\thetab{f^2}{p^2}{t}\odot\Big(\omega\ast
     i_t\Big)\odot\thetaa{f^1}{p^1}{t},  
    \end{gather}
    there are an object $F$, a morphism $\operatorname{u}:F\rightarrow D$ in $\SETW$ and an
    invertible $2$-morphism $\gamma:t\circ\operatorname{u}\Rightarrow t'\circ\operatorname{u}$ in
    $\CATC$, such that
    
    \begin{equation}\label{eq-141}
    \thetaa{p^m}{t'}{\operatorname{u}}\odot\Big(i_{p^m}\ast\gamma\Big)=\Big(\gamma^m\ast
    i_{\operatorname{u}}\Big)\odot\thetaa{p^m}{t}{\operatorname{u}}\quad\textrm{for}\,\,\,m=1,2;
    \end{equation}

   \item given any set of data $(t,t',\gamma^1,\gamma^2,F,\operatorname{u},\gamma)$
    as in \emphatic{(}b\emphatic{)}, if there is another choice
    of data $\widetilde{F}$, $\widetilde{\operatorname{u}}:\widetilde{F}\rightarrow D$ in $\SETW$
    and $\widetilde{\gamma}:t\circ\widetilde{\operatorname{u}}\Rightarrow t'\circ
    \widetilde{\operatorname{u}}$ invertible, such that
    
    \begin{equation}\label{eq-09}
    \thetaa{p^m}{t'}{\widetilde{\operatorname{u}}}\odot\Big(i_{p^m}\ast\widetilde{\gamma}\Big)=\Big(
    \gamma^m\ast i_{\widetilde{\operatorname{u}}}\Big)\odot\thetaa{p^m}{t}
    {\widetilde{\operatorname{u}}}\quad\textrm{for}\,\,\,m=1,2,
    \end{equation}
    then there are an object $G$, a morphism $\operatorname{z}:G\rightarrow F$ in $\SETW$, a morphism
    $\widetilde{\operatorname{z}}:G\rightarrow\widetilde{F}$ and an invertible $2$-morphism
    $\mu:\operatorname{u}\circ\operatorname{z}\Rightarrow\widetilde{\operatorname{u}}\circ
    \widetilde{\operatorname{z}}$, such that

    \begin{gather}
    \nonumber \thetaa{t'}{\widetilde{\operatorname{u}}}{\widetilde{\operatorname{z}}}\odot
     \Big(i_{t'}\ast\mu\Big)\odot\thetab{t'}{\operatorname{u}}{\operatorname{z}}\odot
     \Big(\gamma\ast i_{\operatorname{z}}\Big)= \\
%%%
    \label{eq-86} =\Big(\widetilde{\gamma}\ast i_{\widetilde{\operatorname{z}}}\Big)\odot\thetaa{t}
     {\widetilde{\operatorname{u}}}{\widetilde{\operatorname{z}}}\odot\Big(i_t\ast\mu
     \Big)\odot\thetab{t}{\operatorname{u}}{\operatorname{z}}.  
    \end{gather}
  \end{enumerate}
\end{enumerate}
\end{lem}

\begin{proof}
Again, we give a complete proof in the case when $\CATC$ is a $2$-category.
Let us suppose that (ii2) holds, let us fix any object $D$ and
let us prove that (b) holds. So let us fix any tuple
$(t,t',\gamma^1,\gamma^2)$ as in (b), such that \eqref{eq-65} is satisfied. Then for each $m=1,2$
we define an invertible $2$-morphism $\Gamma^m$ from

\[\Big(C,\id_C,p^m\Big)\circ\Big(D,\id_D,t\Big)=\Big(D,\id_D,p^m\circ t\Big)\]
to

\[\Big(C,\id_C,p^m\Big)\circ\Big(D,\id_D,t'\Big)=\Big(D,\id_D,p^m\circ t'\Big)\]
in $\CATC\left[\SETWinv\right]$ as the $2$-morphism represented by the following diagram:

\begin{equation}\label{eq-03}
\begin{tikzpicture}[xscale=2.2,yscale=-0.8]
    \node (A0_2) at (2, 0) {$D$};
    \node (A2_2) at (2, 2) {$D$};
    \node (A2_0) at (0, 2) {$D$};
    \node (A2_4) at (4, 2) {$B^m$.};
    \node (A4_2) at (2, 4) {$D$};
    
    \node (A2_3) at (2.8, 2) {$\Downarrow\,\gamma^m$};
    \node (A2_1) at (1.2, 2) {$\Downarrow\,i_{\id_D}$};
    
    \path (A4_2) edge [->]node [auto,swap] {$\scriptstyle{p^m\circ t'}$} (A2_4);
    \path (A0_2) edge [->]node [auto] {$\scriptstyle{p^m\circ t}$} (A2_4);
    \path (A2_2) edge [->]node [auto,swap] {$\scriptstyle{\id_D}$} (A0_2);
    \path (A2_2) edge [->]node [auto] {$\scriptstyle{\id_D}$} (A4_2);
    \path (A4_2) edge [->]node [auto] {$\scriptstyle{\id_D}$} (A2_0);
    \path (A0_2) edge [->]node [auto,swap] {$\scriptstyle{\id_D}$} (A2_0);
\end{tikzpicture}
\end{equation}

Using \eqref{eq-91} and \eqref{eq-03} together with Lemmas~\ref{lem-10}, \ref{lem-09}
and~\ref{lem-12}, we have

\begin{gather}
\nonumber \Thetab{(B^2,\id_{B^2},f^2)}{(C,\id_C,p^2)}{(D,\id_D,t')}\odot\Big(\Omega
 \ast i_{(D,\id_D,t')}\Big)\odot \\
\nonumber \odot\,\Thetaa{(B^1,\id_{B^1},f^1)}{(C,\id_C,p^1)}{(D,\id_D,t')}\odot\Big(
 i_{(B^1,\id_{B^1},f^1)}\ast\Gamma^1\Big)= \\
%%%
\nonumber =i_{\left(D,\id_D,f^2\circ p^2\circ t'\right)}\odot\Big[D,\id_D,\id_D,
 i_{\id_D},\omega\ast i_{t'}\Big]\odot \\
\nonumber \odot\, i_{\left(D,\id_D,f^1\circ p^1\circ t'\right)}\odot\Big[D,\id_D,\id_D,
 i_{\id_D},i_{f^1}\ast\gamma^1\Big]= \\
%%%
\nonumber =\Big[D,\id_D,\id_D,i_{\id_D},\Big(\omega\ast i_{t'}\Big)
 \odot\Big(i_{f^1}\ast\gamma^1\Big)\Big]\stackrel{\eqref{eq-65}}{=} \\
%%%
\nonumber \stackrel{\eqref{eq-65}}{=}\Big[D,\id_D,\id_D,i_{\id_D},\Big(i_{f^2}\ast\gamma^2\Big)\odot
 \Big(\omega\ast i_t\Big)\Big]=  \\
%%%
\nonumber =\Big[D,\id_D,\id_D,i_{\id_D},i_{f^2}\ast\gamma^2\Big]\odot i_{\left(D,\id_D,f^2\circ p^2
 \circ t\right)}\odot \\
\nonumber \odot\Big[D,\id_D,\id_D,i_{\id_D},\omega\ast i_t\Big]\odot i_{\left(D,\id_D,f^1\circ p^1
 \circ t\right)}= \\ 
%%%
\nonumber =\Big(i_{(B^2,\id_{B^2},f^2)}\ast\Gamma^2\Big)\odot\Thetab{(B^2,\id_{B^2},
 f^2)}{(C,\id_C,p^2)}{(D,\id_D,t)}\odot \\
\label{eq-68} \odot\Big(\Omega\ast i_{(D,\id_D,t)}\Big)\odot
 \Thetaa{(B^1,\id_{B^1},f^1)}{(C,\id_C,p^1)}{(D,\id_D,t)}.  
\end{gather}

Since we are assuming (ii2), then \eqref{eq-68} implies that there is a unique invertible
$2$-morphism $\Gamma:(D,\id_D,t)\Rightarrow(D,\id_D,t')$ in $\CATC\left[\SETWinv\right]$, such that

\begin{equation}\label{eq-66}
\Gamma^m=i_{(C,\id_C,p^m)}\ast\Gamma\quad\textrm{for }m=1,2.
\end{equation}

By~\cite[Lemma~6.1]{T3} for $\alpha:=i_{\id_D}$ and $\Gamma$, there are an object $T$, a morphism
$\operatorname{q}:T\rightarrow D$ in $\SETW$ and a $2$-morphism $\eta:t\circ
\operatorname{q}\Rightarrow t'\circ\operatorname{q}$, such that $\Gamma=[
T,\operatorname{q},\operatorname{q},i_{\operatorname{q}},
\eta]$. Since $\Gamma$ is invertible in $\CATC\left[\SETWinv\right]$, then
by~\cite[Proposition~0.8]{T3} we can assume that $\eta$ is invertible.
Then by Lemma~\ref{lem-12} we have:

\begin{gather*}
\Big[T,\operatorname{q},\operatorname{q},i_{\operatorname{q}},\gamma^1\ast i_{\operatorname{q}}
 \Big]=\Big[D,\id_D,\id_D,i_{\id_D},\gamma^1\Big]\stackrel{\eqref{eq-03}}{=} \\
%%%
\stackrel{\eqref{eq-03}}{=}\Gamma^1
 \stackrel{\eqref{eq-66}}{=}i_{(C,\id_C,p^1)}\ast\Gamma=\Big[T,\operatorname{q},
 \operatorname{q},i_{\operatorname{q}},i_{p^1}\ast\eta\Big].
\end{gather*}

By~\cite[Proposition~0.7]{T3}, the previous identity implies that there are an object $R$ and
a morphism $\operatorname{x}^1:R\rightarrow T$, such that $\operatorname{q}\circ\operatorname{x}^1$
belongs to $\SETW$ and such that

\begin{equation}\label{eq-103}
\gamma^1\ast i_{\operatorname{q}\circ\operatorname{x}^1}=i_{p^1}\ast\eta
\ast i_{\operatorname{x}^1}.
\end{equation}

By Lemma~\ref{lem-12} we have:

\begin{gather*}
\Big[R,\operatorname{q}\circ\operatorname{x}^1,\operatorname{q}\circ\operatorname{x}^1,
 i_{\operatorname{q}\circ\operatorname{x}^1},\gamma^2\ast i_{\operatorname{q}\circ\operatorname{x}^1}
 \Big]=\Big[D,\id_D,\id_D,i_{\id_D},\gamma^2\Big]\stackrel{\eqref{eq-03}}{=}
 \Gamma^2\stackrel{\eqref{eq-66}}{=} \\
%%%
\stackrel{\eqref{eq-66}}{=}i_{(C,\id_C,p^2)}\ast\Gamma=\Big[T,\operatorname{q},\operatorname{q},
 i_{\operatorname{q}},i_{p^2}\ast\eta\Big]=\Big[R,
 \operatorname{q}\circ\operatorname{x}^1,\operatorname{q}\circ\operatorname{x}^1,
 i_{\operatorname{q}\circ\operatorname{x}^1},i_{p^2}\ast\eta\ast i_{\operatorname{x}^1}\Big].
\end{gather*}

Again by~\cite[Proposition~0.7]{T3}, the previous identity implies that there are an object
$F$ and a morphism $\operatorname{x}^2:F\rightarrow R$, such that $\operatorname{q}\circ
\operatorname{x}^1\circ\operatorname{x}^2$ belongs to $\SETW$ and such that

\begin{equation}\label{eq-104}
\gamma^2\ast i_{\operatorname{q}\circ\operatorname{x}^1\circ\operatorname{x}^2}=i_{p^2}\ast
\eta\ast i_{\operatorname{x}^1\circ\operatorname{x}^2}.
\end{equation}

We set $\operatorname{u}:=\operatorname{q}\circ\operatorname{x}^1\circ\operatorname{x}^2:F
\rightarrow D$ and 

\begin{equation}\label{eq-115}
\gamma:=\eta\ast i_{\operatorname{x}^1\circ\operatorname{x}^2}:\,\,t\circ\operatorname{u}
\Longrightarrow t'\circ\operatorname{u}.
\end{equation}

Then from \eqref{eq-103} and \eqref{eq-104} we get that $i_{p^m}\ast\gamma=\gamma^m\ast
i_{\operatorname{u}}$ for each $m=1,2$; moreover $\gamma$ is invertible because $\eta$ is so
by construction. So we have proved that (ii2) implies condition (b) for each object $D$ of $\CATC$.\\

Let us also prove that
(ii2) implies (c). So let us fix any set of data $(t,t',\gamma^1,\gamma^2,F,\operatorname{u},
\gamma)$ as in (b) and any set of data $(\widetilde{F},
\widetilde{\operatorname{u}},\widetilde{\gamma})$ as in (c). In particular, we assume that
\eqref{eq-141} and \eqref{eq-09} hold. Then we define a pair of invertible
$2$-morphisms in $\CATC\left[\SETWinv\right]$ as follows:

\[\Gamma:=\Big[F,\operatorname{u},\operatorname{u},i_{\operatorname{u}},\gamma\Big],\,
\widetilde{\Gamma}:=\Big[\widetilde{F},\widetilde{\operatorname{u}},\widetilde{\operatorname{u}},
i_{\widetilde{\operatorname{u}}},\widetilde{\gamma}\Big]:\Big(D,\id_D,t\Big)\Longrightarrow
\Big(D,\id_D,t'\Big).\]

Then by Lemma~\ref{lem-12}, for each $m=1,2$ we have

\begin{gather*}
i_{(C,\id_C,p^m)}\ast\Gamma=\Big[F,
 \operatorname{u},\operatorname{u},i_{\operatorname{u}},
 i_{p^m}\ast\gamma\Big]\stackrel{\eqref{eq-141}}{=} \\
%%%
\stackrel{\eqref{eq-141}}{=}\Big[F,\operatorname{u},
 \operatorname{u},i_{\operatorname{u}},\gamma^m\ast
 i_{\operatorname{u}}\Big]=\Big[D,\id_D,\id_D,i_{\id_D},\gamma^m
 \Big]=\Big[\widetilde{F},\widetilde{\operatorname{u}},\widetilde{\operatorname{u}},
 i_{\widetilde{\operatorname{u}}},\gamma^m\ast i_{\widetilde{\operatorname{u}}}\Big]
 \stackrel{\eqref{eq-09}}{=} \\
%%%
\stackrel{\eqref{eq-09}}{=}\Big[\widetilde{F},\widetilde{\operatorname{u}},
 \widetilde{\operatorname{u}},i_{\widetilde{\operatorname{u}}},i_{p^m}\ast\widetilde{\gamma}
 \Big]=i_{(C,\id_C,p^m)}\ast\widetilde{\Gamma}.
\end{gather*}

Then by the uniqueness part of condition \textbf{\hyperref[B2]{B2}}$(D,(D,\id_D,t),(D,\id_D,t'))$ we
conclude that $\Gamma=\widetilde{\Gamma}$.
Then by Lemma~\ref{lem-19} there are an object $G$, a morphism
$\operatorname{z}:G\rightarrow F$ in $\SETW$, a morphisms
$\widetilde{\operatorname{z}}:G\rightarrow\widetilde{F}$ and an invertible $2$-morphism
$\mu:\operatorname{u}\circ\operatorname{z}\Rightarrow\widetilde{\operatorname{u}}\circ
\widetilde{\operatorname{z}}$, such that

\[\Big(i_{t'}\ast\mu\Big)\odot\Big(\gamma\ast
i_{\operatorname{z}}\Big)\odot\Big(i_t\ast\mu^{-1}
\Big)=\widetilde{\gamma}\ast i_{\widetilde{\operatorname{z}}}.\]

Such an identity is equivalent to \eqref{eq-86} (in the case when $\CATC$ is a $2$-category),
so we have proved that (ii2) implies condition (c) for each object $D$, hence (ii3) holds.\\

Conversely, let us suppose that (ii3) holds and let us prove that (ii2) holds. So let us
fix any object $D$ and any pair of morphisms $t,t':D\rightarrow C$ in $\CATC$; we have to prove that
condition \textbf{\hyperref[B2]{B2}}$(D,(D,\id_D,t),(D,\id_D,t'))$ holds for diagram \eqref{eq-91}.
In order to do that, let us fix any pair of invertible $2$-morphisms

\[\Gamma^m:\Big(C,\id_C,p^m\Big)\circ\Big(D,\id_D,t\Big)\Longrightarrow\Big(C,\id_C,p^m\Big)\circ
\Big(D,\id_D,t'\Big)\quad\textrm{for }m=1,2\]
in $\CATC\left[\SETWinv\right]$, such that

\begin{gather}
\nonumber \Thetab{(B^2,\id_{B^2},f^2)}{(C,\id_C,p^2)}{(D,\id_D,t')}\odot\Big(\Omega\ast i_{(D,
 \id_D,t')}\Big)\odot \\
\nonumber \odot\,\Thetaa{(B^1,\id_{B^1},f^1)}{(C,\id_C,p^1)}{(D,\id_D,t')}\odot\Big(i_{(B^1,
 \id_{B^1},f^1)}\ast\Gamma^1\Big)= \\
%%%
\nonumber =\Big(i_{(B^2,\id_{B^2},f^2)}\ast\Gamma^2\Big)\odot\Thetab{(B^2,\id_{B^2},
 f^2)}{(C,\id_C,p^2)}{(D,\id_D,t)}\odot \\
\label{eq-39} \odot\Big(\Omega\ast i_{(D,\id_D,t)}\Big)\odot\Thetaa{(B^1,\id_{B^1},f^1)}{(C,\id_C,
 p^1)}{(D,\id_D,t)}.  
\end{gather}

By~\cite[Lemma~6.1]{T3} applied to $\alpha:=i_{\id_D}$ and to $\Gamma^1$, there are an object
$K$, a morphism
$\operatorname{x}^1:K\rightarrow D$ in $\SETW$ and a $2$-morphism
$\alpha^1:p^1\circ t\circ\operatorname{x}^1\Rightarrow p^1\circ t'\circ\operatorname{x}^1$ in
$\CATC$, such that

\begin{equation}\label{eq-79}
\Gamma^1=\Big[K,\operatorname{x}^1,\operatorname{x}^1,i_{\operatorname{x}^1},
\alpha^1\Big]:\Big(
D,\id_D,p^1\circ t\Big)\Longrightarrow\Big(D,\id_D,p^1\circ t'\Big).
\end{equation}

Since $\Gamma^1$ is invertible in $\CATC\left[\SETWinv\right]$, then by~\cite[Proposition~0.8]{T3}
we can assume that $\alpha^1$ is invertible in $\CATC$.
Now we apply~\cite[Lemma~6.1]{T3} to $\alpha:=i_{\operatorname{x}^1}$ and to $\Gamma^2$. Then 
there are an object $M$, a morphism $\operatorname{x}^2:M\rightarrow K$ such that $\operatorname{x}^1
\circ\operatorname{x}^2$ belongs to $\SETW$,
and a $2$-morphism 

\[\widetilde{\alpha}^2:\,p^2\circ t\circ\operatorname{x}^1\circ\operatorname{x}^2
\Longrightarrow p^2\circ t'\circ\operatorname{x}^1\circ\operatorname{x}^2,\]
such that

\begin{equation}\label{eq-59}
\Gamma^2=\Big[M,\operatorname{x}^1\circ\operatorname{x}^2,\operatorname{x}^1\circ\operatorname{x}^2,
i_{\operatorname{x}^1\circ\operatorname{x}^2},\widetilde{\alpha}^2\Big].
\end{equation}

As above, we can assume that $\widetilde{\alpha}^2$ is invertible in $\CATC$ since $\Gamma^2$
is invertible in $\CATC\left[\SETWinv\right]$.
If we set $\widetilde{\alpha}^1:=\alpha^1\ast i_{\operatorname{x}^2}$, then from \eqref{eq-79}
we get

\begin{equation}\label{eq-01}
\Gamma^1=\Big[M,\operatorname{x}^1\circ\operatorname{x}^2,\operatorname{x}^1\circ\operatorname{x}^2,
i_{\operatorname{x}^1\circ\operatorname{x}^2},\widetilde{\alpha}^1\Big].
\end{equation}

So using Lemma~\ref{lem-12}, for each $m=1,2$ we have

\begin{equation}\label{eq-17}
i_{(B^m,\id_{B^m},f^m)}\ast\Gamma^m=\Big[M,\operatorname{x}^1\circ\operatorname{x}^2,
\operatorname{x}^1\circ\operatorname{x}^2,i_{\operatorname{x}^1\circ\operatorname{x}^2},
i_{f^m}\ast\widetilde{\alpha}^m\Big].
\end{equation}

Moreover, using \eqref{eq-91} and
Lemma~\ref{lem-09}, we have $\Omega\ast i_{(D,\id_D,t)}=[D,\id_D,\id_D,i_{\id_D},
\omega\ast i_t]$ and analogously $\Omega\ast i_{(D,\id_D,t')}=[D,\id_D,\id_D,i_{\id_D},
\omega\ast i_{t'}]$. Using such identities together with Lemma~\ref{lem-10},
we get that

\begin{gather}
\nonumber \Big[M,\operatorname{x}^1\circ\operatorname{x}^2,\operatorname{x}^1\circ\operatorname{x}^2,
 i_{\operatorname{x}^1\circ\operatorname{x}^2},\Big(\omega\ast i_{t'\circ
 \operatorname{x}^1\circ\operatorname{x}^2}\Big)
 \odot \Big(i_{f^1}\ast\widetilde{\alpha}^1\Big)\Big]= \\
%%%
\nonumber =\Big[D,\id_D,\id_D,i_{\id_D},\omega\ast i_{t'}\Big]\odot \\
%%%
\nonumber \odot\Big[M,\operatorname{x}^1\circ\operatorname{x}^2,\operatorname{x}^1
 \circ\operatorname{x}^2,i_{\operatorname{x}^1\circ\operatorname{x}^2},
 i_{f^1}\ast\widetilde{\alpha}^1\Big]\stackrel{\eqref{eq-17}}{=} \\
%%%
\nonumber \stackrel{\eqref{eq-17}}{=}i_{\left(D,\id_D,f^2\circ p^2\circ t'\right)}\odot\Big(\Omega\ast
 i_{\left(D,\id_D,t'\right)}\Big)\odot i_{\left(D,\id_D,f^1\circ p^1\circ t'\right)}\odot
 \Big(i_{\left(B^1,\id_{B^1},f^1\right)}\ast\Gamma^1\Big)\stackrel{\eqref{eq-39}}{=} \\
%%%
\nonumber \stackrel{\eqref{eq-39}}{=}\Big(i_{\left(B^2,\id_{B^2},f^2\right)}\ast\Gamma^2\Big)
 \odot i_{\left(D,\id_D,f^2\circ p^2\circ t\right)}\odot\Big(\Omega\ast i_{\left(D,\id_D,t\right)}
 \Big)\odot i_{\left(D,\id_D,f^1\circ p^1\circ t\right)}\stackrel{\eqref{eq-17}}{=} \\
%%%
\nonumber \stackrel{\eqref{eq-17}}{=}\Big[M,\operatorname{x}^1\circ\operatorname{x}^2,
 \operatorname{x}^1\circ\operatorname{x}^2,
 i_{\operatorname{x}^1\circ\operatorname{x}^2},i_{f^2}\ast\widetilde{\alpha}^2\Big]\odot \\
\nonumber \odot\Big[D,\id_D,\id_D,i_{\id_D},\omega\ast i_t\Big]= \\
%%%
\label{eq-69} =\Big[M,\operatorname{x}^1\circ\operatorname{x}^2,\operatorname{x}^1\circ
 \operatorname{x}^2,i_{\operatorname{x}^1\circ\operatorname{x}^2},\Big(
 i_{f^2}\ast\widetilde{\alpha}^2\Big)\odot\Big(\omega\ast i_{t\circ\operatorname{x}^1
 \circ\operatorname{x}^2}\Big)\Big].  
\end{gather}

Using \eqref{eq-69} and~\cite[Proposition~0.7]{T3}, there are an object $N$ and a morphism
$\operatorname{y}:N\rightarrow M$, such that $\operatorname{x}^1\circ\operatorname{x}^2\circ
\operatorname{y}$ belongs to $\SETW$ and

\begin{gather}
\nonumber \Big(\Big(\omega\ast i_{t'\circ\operatorname{x}^1\circ\operatorname{x}^2}\Big)\odot
 \Big(i_{f^1}\ast\widetilde{\alpha}^1\Big)\Big)\ast i_{\operatorname{y}}= \\
%%%
\label{eq-75} =\Big(\Big(i_{f^2}\ast
 \widetilde{\alpha}^2\Big)\odot\Big(\omega\ast i_{t\circ\operatorname{x}^1\circ\operatorname{x}^2}
 \Big)\Big)\ast i_{\operatorname{y}}.  
\end{gather}

Then we set

\begin{equation}\label{eq-23}
n:=t\circ\operatorname{x}^1\circ\operatorname{x}^2\circ\operatorname{y}:\,N\longrightarrow
C,\quad\quad n':=t'\circ\operatorname{x}^1\circ\operatorname{x}^2\circ\operatorname{y}:\,N
\longrightarrow C
\end{equation}
and

\begin{equation}\label{eq-105}
\gamma^m:=\widetilde{\alpha}^m\ast i_{\operatorname{y}}:\,
p^m\circ n\Longrightarrow p^m\circ n'\quad\textrm{for }m=1,2.
\end{equation}

Then \eqref{eq-75} implies that:

\[\Big(\omega\ast i_{n'}\Big)\odot\Big(i_{f^1}\ast\gamma^1\Big)=
\Big(i_{f^2}\ast\gamma^2\Big)\odot\Big(\omega\ast i_n\Big).\]

We recall that we assumed that (ii3) holds. This implies that (b) holds for $D$ replaced by
$N$ and $(t,t')$ replaced by $(n,n')$. So there are an object $F$, a morphism $\operatorname{u}:
F\rightarrow N$ in $\SETW$ and an invertible $2$-morphism $\gamma:
n\circ\operatorname{u}\Rightarrow n'\circ\operatorname{u}$,
such that

\begin{equation}\label{eq-77}
i_{p^m}\ast\gamma=\gamma^m\ast i_{\operatorname{u}}\quad\textrm{for }m=1,2.
\end{equation}

Then we set $\operatorname{a}:=\operatorname{x}^1\circ\operatorname{x}^2\circ\operatorname{y}
\circ\operatorname{u}:F\rightarrow D$ (so that $\gamma$ is defined from $t\circ\operatorname{a}$
to $t'\circ\operatorname{a}$) and 

\begin{equation}\label{eq-29}
\Gamma:=\Big[F,\operatorname{a},\operatorname{a},i_{\operatorname{a}},\gamma\Big]:\Big(D,\id_D,t
\Big)\Longrightarrow\Big(D,\id_D,t'\Big).
\end{equation}

Then for each $m=1,2$ we have:

\begin{equation}\label{eq-112}
i_{p^m}\ast\gamma\stackrel{\eqref{eq-77}}{=}\gamma^m\ast i_{\operatorname{u}}
\stackrel{\eqref{eq-105}}{=}
\widetilde{\alpha}^m\ast i_{\operatorname{y}\circ\operatorname{u}}.
\end{equation}

Then using Lemma~\ref{lem-12}, for each $m=1,2$ we have

\begin{gather*}
i_{(C,\id_C,p^m)}\ast\Gamma=i_{(C,\id_C,p^m)}\ast\Big[F,\operatorname{a},\operatorname{a},
 i_{\operatorname{a}},\gamma\Big]=\Big[F,\operatorname{a},\operatorname{a},
 i_{\operatorname{a}},i_{p^m}\ast\gamma\Big]\stackrel{\eqref{eq-112}}{=} \\
%%%
\stackrel{\eqref{eq-112}}{=}\Big[F,\operatorname{x}^1\circ\operatorname{x}^2\circ
 \operatorname{y}\circ\operatorname{u},\operatorname{x}^1\circ\operatorname{x}^2\circ
 \operatorname{y}\circ\operatorname{u},i_{\operatorname{x}^1\circ\operatorname{x}^2\circ
 \operatorname{y}\circ\operatorname{u}},\widetilde{\alpha}^m\ast i_{\operatorname{y}\circ
 \operatorname{u}}\Big]= \\
%%%
=\Big[M,\operatorname{x}^1\circ\operatorname{x}^2,\operatorname{x}^1\circ\operatorname{x}^2,
 i_{\operatorname{x}^1\circ\operatorname{x}^2},\widetilde{\alpha}^m\Big]
 \stackrel{\eqref{eq-59},\eqref{eq-01}}{=}\Gamma^m.
\end{gather*}

This proves that the existence part of condition
\textbf{\hyperref[B2]{B2}}$(D,(D,\id_D,t),(D,\id_D,t'))$ is satisfied.
Then we need only to prove that the $2$-morphism $\Gamma$ defined above is the unique invertible
$2$-morphism in $\CATC\left[\SETWinv\right]$ such that $i_{(C,\id_C,p^m)}\ast\Gamma=\Gamma^m$ for
each $m=1,2$. So let us suppose that there is another invertible $2$-morphism $\widetilde{\Gamma}:
(D,\id_D,t)\Rightarrow(D,\id_D,t')$ in $\CATC\left[\SETWinv\right]$, such that $i_{(C,\id_C,p^m)}
\ast\widetilde{\Gamma}=\Gamma^m$ for each $m=1,2$. Then we apply~\cite[Lemma~6.1]{T3} to $\alpha:=
i_{\operatorname{x}^1\circ\operatorname{x}^2\circ\operatorname{y}\circ\operatorname{u}}$ and 
to $\widetilde{\Gamma}$. Then there are an object $L$, a morphism
$\operatorname{b}:L\rightarrow F$ such that $\operatorname{x}^1\circ\operatorname{x}^2\circ
\operatorname{y}\circ\operatorname{u}\circ\operatorname{b}$ belongs to $\SETW$, and a $2$-morphism

\[\beta:\,t\circ\operatorname{x}^1\circ\operatorname{x}^2\circ\operatorname{y}\circ
\operatorname{u}\circ\operatorname{b}\Longrightarrow t'\circ\operatorname{x}^1
\circ\operatorname{x}^2\circ\operatorname{y}\circ\operatorname{u}\circ\operatorname{b},\]
such that

\begin{equation}\label{eq-123}
\widetilde{\Gamma}=\Big[L,\operatorname{x}^1\circ\operatorname{x}^2\circ\operatorname{y}
\circ\operatorname{u}\circ\operatorname{b},\operatorname{x}^1
\circ\operatorname{x}^2\circ\operatorname{y}\circ\operatorname{u}\circ\operatorname{b},
i_{\operatorname{x}^1\circ\operatorname{x}^2\circ\operatorname{y}\circ\operatorname{u}
\circ\operatorname{b}},\beta\Big].
\end{equation}

Then by Lemma~\ref{lem-12}, for each $m=1,2$ we have

\begin{gather}
\nonumber \Big[L,\operatorname{x}^1\circ\operatorname{x}^2\circ\operatorname{y}\circ
 \operatorname{u}\circ\operatorname{b},\operatorname{x}^1\circ\operatorname{x}^2\circ
 \operatorname{y}\circ\operatorname{u}\circ\operatorname{b},i_{\operatorname{x}^1\circ
 \operatorname{x}^2\circ\operatorname{y}\circ\operatorname{u}\circ\operatorname{b}},
 i_{p^m}\ast\beta\Big]= \\
%%%
\nonumber =i_{\left(C,\id_C,p^m\right)}\ast\widetilde{\Gamma}=\Gamma^m
 \stackrel{\eqref{eq-59},\eqref{eq-01}}{=}\Big[M,\operatorname{x}^1\circ
 \operatorname{x}^2,\operatorname{x}^1\circ\operatorname{x}^2,i_{\operatorname{x}^1\circ
 \operatorname{x}^2},\widetilde{\alpha}^m\Big]\stackrel{\eqref{eq-105}}{=} \\
%%%
\label{eq-78} \stackrel{\eqref{eq-105}}{=}\Big[L,\operatorname{x}^1\circ\operatorname{x}^2
 \circ\operatorname{y}\circ\operatorname{u}\circ
 \operatorname{b},\operatorname{x}^1\circ\operatorname{x}^2\circ\operatorname{y}\circ\operatorname{u}
 \circ\operatorname{b},i_{\operatorname{x}^1\circ\operatorname{x}^2\circ\operatorname{y}\circ
 \operatorname{u}\circ\operatorname{b}},\gamma^m\ast i_{\operatorname{u}\circ\operatorname{b}}\Big].
\end{gather}

Now we apply~\cite[Proposition~0.7]{T3} to \eqref{eq-78} for $m=1$. So there are an object $H$ and a
morphism
$\operatorname{s}:H\rightarrow L$, such that $\operatorname{x}^1\circ\operatorname{x}^2\circ
\operatorname{y}\circ\operatorname{u}\circ\operatorname{b}\circ\operatorname{s}$ belongs to
$\SETW$ and such that

\begin{equation}\label{eq-70}
i_{p^1}\ast\beta\ast i_{\operatorname{s}}=\gamma^1\ast i_{\operatorname{u}\circ\operatorname{b}
\circ\operatorname{s}}.
\end{equation}

Moreover, from \eqref{eq-78} for $m=2$, we get:

\begin{gather*}
\Big[H,\operatorname{x}^1\circ\operatorname{x}^2\circ\operatorname{y}\circ\operatorname{u}\circ
 \operatorname{b}\circ\operatorname{s},\operatorname{x}^1\circ\operatorname{x}^2\circ
 \operatorname{y}\circ\operatorname{u}\circ\operatorname{b}\circ\operatorname{s},i_{\operatorname{x}^1
 \circ\operatorname{x}^2\circ\operatorname{y}\circ\operatorname{u}\circ\operatorname{b}\circ
 \operatorname{s}},i_{p^2}\ast\beta\ast i_{\operatorname{s}}\Big]= \\
%%%
=\Big[H,\operatorname{x}^1\circ\operatorname{x}^2\circ\operatorname{y}\circ\operatorname{u}\circ
 \operatorname{b}\circ\operatorname{s},\operatorname{x}^1\circ\operatorname{x}^2\circ
 \operatorname{y}\circ\operatorname{u}\circ\operatorname{b}\circ\operatorname{s},i_{\operatorname{x}^1
 \circ\operatorname{x}^2\circ\operatorname{y}\circ\operatorname{u}\circ\operatorname{b}\circ
 \operatorname{s}},\gamma^2\ast i_{\operatorname{u}\circ\operatorname{b}\circ\operatorname{s}}
 \Big].
\end{gather*}

So again by~\cite[Proposition~0.7]{T3},
there are an object $I$ and a morphism $\operatorname{r}:
I\rightarrow H$, such that $\operatorname{x}^1\circ\operatorname{x}^2\circ
\operatorname{y}\circ\operatorname{u}\circ\operatorname{b}\circ\operatorname{s}\circ
\operatorname{r}$ belongs to $\SETW$ and

\begin{equation}\label{eq-72}
i_{p^2}\ast\beta\ast i_{\operatorname{s}\circ
\operatorname{r}}=\gamma^2\ast i_{\operatorname{u}\circ\operatorname{b}\circ\operatorname{s}
\circ\operatorname{r}}.
\end{equation}

Since also $\operatorname{x}^1\circ\operatorname{x}^2\circ\operatorname{y}$ belongs to $\SETW$
by construction, then by Lemma~\ref{lem-06} there are an object $\widetilde{F}$ and a morphism
$\operatorname{c}:\widetilde{F}\rightarrow I$, such that the morphism $\widetilde{\operatorname{u}}:=
\operatorname{u}\circ\operatorname{b}\circ\operatorname{s}\circ\operatorname{r}\circ\operatorname{c}:
\widetilde{F}\rightarrow N$ belongs to $\SETW$. Then we define $\widetilde{\gamma}:=\beta\ast 
i_{\operatorname{s}\circ\operatorname{r}\circ\operatorname{c}}$, so from \eqref{eq-70} and
\eqref{eq-72} we get that

\begin{equation}\label{eq-128}
i_{p^m}\ast\widetilde{\gamma}=\gamma^m\ast i_{\widetilde{\operatorname{u}}}\quad\textrm{for }
\,\,m=1,2.
\end{equation}

We recall that we already used (b) (for the data $(N,n,n',\gamma^1,\gamma^2)$)
in order to get a set of data $(F,\operatorname{u},\gamma)$
such that \eqref{eq-77} holds. Since \eqref{eq-128} holds, then we can apply
(c) for the data $(\widetilde{F},\widetilde{\operatorname{u}},\widetilde{\gamma})$,
so there are an object $G$, a morphism $\operatorname{z}:G\rightarrow F$ in $\SETW$, a morphism
$\widetilde{\operatorname{z}}:G\rightarrow\widetilde{F}$ and an invertible $2$-morphism

\[\mu:\,\,\operatorname{u}\circ
\operatorname{z}\Longrightarrow\widetilde{\operatorname{u}}\circ\widetilde{\operatorname{z}}=
\operatorname{u}\circ\operatorname{b}\circ\operatorname{s}
\circ\operatorname{r}\circ\operatorname{c}\circ\,\widetilde{\operatorname{z}},\]
such that

\[\Big(i_{n'}\ast\mu\Big)\odot\Big(\gamma\ast i_{\operatorname{z}}\Big)=
\Big(\widetilde{\gamma}\ast i_{\widetilde{\operatorname{z}}}\Big)\odot\Big(i_n\ast\mu\Big).\]

If we replace $\widetilde{\gamma}$ with $\beta\ast i_{\operatorname{s}\circ\operatorname{r}\circ
\operatorname{c}}$,
and $n,n'$ with their definition in \eqref{eq-23}, then the previous identity implies that:

\begin{equation}\label{eq-20}
\gamma\ast i_{\operatorname{z}}=\Big(i_{t'\circ
\operatorname{x}^1\circ\operatorname{x}^2\circ\operatorname{y}}\ast\mu^{-1}\Big)
\odot\Big(\beta\ast i_{\operatorname{s}\circ\operatorname{r}\circ\operatorname{c}\circ
\widetilde{\operatorname{z}}}\Big)\odot\Big(i_{t\circ\operatorname{x}^1\circ\operatorname{x}^2
\circ\operatorname{y}}\ast\mu\Big).  
\end{equation}

So using~\cite[\S~2.3]{Pr} we have

\begin{gather*}
\Gamma\stackrel{\eqref{eq-29}}{=}\Big[F,\operatorname{a},
 \operatorname{a},i_{\operatorname{a}},\gamma\Big]\stackrel{\eqref{eq-20}}{=} \\
\stackrel{\eqref{eq-20}}{=} \Big[L,\operatorname{x}^1\circ\operatorname{x}^2\circ\operatorname{y}
 \circ\operatorname{u}\circ\operatorname{b},\operatorname{x}^1\circ\operatorname{x}^2\circ
 \operatorname{y}\circ\operatorname{u}\circ\operatorname{b},i_{\operatorname{x}^1\circ
 \operatorname{x}^2\circ\operatorname{y}\circ\operatorname{u}
 \circ\operatorname{b}},\beta\Big]\stackrel{\eqref{eq-123}}{=}\widetilde{\Gamma},
\end{gather*}
so we have proved also the uniqueness part of \textbf{\hyperref[B2]{B2}}$(D,(D,\id_D,t),
(D,\id_D,t'))$, i.e.\ we have proved that (ii3) implies (ii2).
\end{proof}

Therefore, we have:

\begin{proof}[Proof of Theorem~\ref{theo-04}.]
Given an object $C$, a pair of morphisms $p^m:C\rightarrow B^m$ for $m=1,2$ and an invertible
$2$-morphism $\omega:f^1\circ p^1\Rightarrow f^2\circ p^2$ in $\CATC$, the induced
diagram \eqref{eq-91} is a weak fiber product if and only if it satisfies conditions
\textbf{\hyperref[A1]{A1}}$(D)$ and \textbf{\hyperref[A2]{A2}}$(D)$
for each object $D$ of $\CATC\left[
\SETWinv\right]$, i.e.\ for each object $D$ of $\CATC$.
Using Lemmas~\ref{lem-11} and~\ref{lem-15}, condition \textbf{\hyperref[A1]{A1}}$(D)$
holds for each object $D$
if and only if property (\hyperref[a]{a}) of Theorem~\ref{theo-04} is satisfied for each $D$.
Using Lemmas~\ref{lem-02} and~\ref{lem-01}, condition
\textbf{\hyperref[A2]{A2}}$(D)$  holds for each object $D$ if and only
if properties (\hyperref[b]{b})
and (\hyperref[c]{c}) are satisfied for each $D$. This suffices to conclude.
\end{proof}

Moreover, we are ready to give also the following proof.

\begin{proof}[Proof of Corollary~\ref{cor-01}.]
As usual, for simplicity of exposition we give the proof assuming that $\CATC$ is a $2$-category.
Let us fix any object $D$ in $\CATC$ and let us start by proving that condition (a) of
Theorem~\ref{theo-04} is satisfied.
So let us suppose that we have fixed any pair of morphisms $q^m:D\rightarrow B^m$
for $m=1,2$ and any invertible $2$-morphism $\lambda:f^1\circ q^1\Rightarrow f^2\circ q^2$ 
in $\CATC$. By hypothesis, \eqref{eq-106} is a weak fiber product in the bicategory $\CATC$; so by
\textbf{\hyperref[A1]{A1}}$(D)$ there are a morphism $q:D\rightarrow C$ and a pair of invertible
$2$-morphisms $\lambda^m:q^m\Rightarrow p^m\circ q$ for $m=1,2$, such that

\[\Big(\omega\ast i_q\Big)\odot\Big(i_{f^1}\ast\lambda^1\Big)=
\Big(i_{f^2}\ast\lambda^2\Big)\odot\lambda.\]

Then condition (\hyperref[a]{a}) holds if we set $E:=D$ and $\operatorname{v}:=\id_D$. Now let us
prove (b), so let us fix any pair of morphisms $t,t':D\rightarrow C$ and any pair of invertible
$2$-morphisms $\gamma^m:p^m\circ t\Rightarrow p^m\circ t'$ for $m=1,2$ in $\CATC$, such that

\begin{equation}\label{eq-84}
\Big(\omega\ast i_{t'}\Big)\odot\Big(i_{f^1}\ast\gamma^1\Big)=
\Big(i_{f^2}\ast\gamma^2\Big)\odot\Big(\omega\ast i_t\Big).  
\end{equation}

Since \eqref{eq-106} is a weak fiber product in $\CATC$, then by \textbf{\hyperref[A2]{A2}}$(D)$ 
there is a unique invertible $2$-morphism $\gamma:t\Rightarrow t'$, such that 

\begin{equation}\label{eq-87}
i_{p^m}\ast\gamma=\gamma^m\quad\textrm{for }m=1,2.
\end{equation}

So condition (\hyperref[b]{b}) is satisfied if we set $F:=D$ and $\operatorname{u}:=\id_D$.\\

Hence, we only need to prove condition (\hyperref[c]{c}). So let us fix any pair of morphisms $t,t':D
\rightarrow C$, any pair of invertible $2$-morphisms $\gamma^m:p^m\circ t\Rightarrow p^m\circ t'$
for $m=1,2$ such that \eqref{eq-84} holds, any pair of objects $F,\widetilde{F}$,
any pair of morphisms $\operatorname{u}:
F\rightarrow D$ and $\widetilde{\operatorname{u}}:\widetilde{F}\rightarrow D$, both in $\SETW$,
and any pair of invertible $2$-morphisms $\gamma:t\circ\operatorname{u}\Rightarrow t'\circ
\operatorname{u}$ and $\widetilde{\gamma}:t\circ\widetilde{\operatorname{u}}\Rightarrow
t'\circ\widetilde{\operatorname{u}}$, such that

\begin{equation}\label{eq-124}
i_{p^m}\ast\gamma=\gamma^m\ast i_{\operatorname{u}}\quad\textrm{and}\quad i_{p^m}\ast
\widetilde{\gamma}=\gamma^m\ast i_{\widetilde{\operatorname{u}}}\quad\textrm{for }m=1,2.
\end{equation}

Using axiom (\hyperref[BF3]{BF3}) there is a set of data as in the upper part of the following
diagram, with $\operatorname{z}$ in $\SETW$ and $\mu$ invertible.

\[
\begin{tikzpicture}[xscale=2.2,yscale=-0.8]
    \node (A0_1) at (1, 0) {$G$};
    \node (A1_0) at (0.3, 2) {$F$};
    \node (A1_2) at (1.7, 2) {$\widetilde{F}$.};
    \node (A2_1) at (1, 2) {$D$};
    
    \node (A1_1) at (1, 1) {$\mu$};
    \node (B1_1) at (1, 1.4) {$\Rightarrow$};
    
    \path (A1_2) edge [->]node [auto] {$\scriptstyle{\widetilde{\operatorname{u}}}$} (A2_1);
    \path (A0_1) edge [->]node [auto] {$\scriptstyle{\widetilde{\operatorname{z}}}$} (A1_2);
    \path (A1_0) edge [->]node [auto,swap] {$\scriptstyle{\operatorname{u}}$} (A2_1);
    \path (A0_1) edge [->]node [auto,swap] {$\scriptstyle{\operatorname{z}}$} (A1_0);
\end{tikzpicture}
\]

For each $m=1,2$, we consider the invertible $2$-morphism

\[\phi^m:=\gamma^m\ast i_{\operatorname{u}\circ\operatorname{z}}:\,\,p^m\circ t\circ\operatorname{u}
\circ\operatorname{z}\Longrightarrow p^m\circ t'\circ\operatorname{u}\circ\operatorname{z}.\]

Then using \eqref{eq-84} we get

\begin{equation}\label{eq-126}
\Big(\omega\ast i_{t'\circ\operatorname{u}\circ\operatorname{z}}\Big)\odot\Big(i_{f^1}\ast
\phi^1\Big)=\Big(i_{f^2}\ast\phi^2\Big)\odot\Big(\omega\ast i_{t\circ\operatorname{u}\circ
\operatorname{z}}\Big).
\end{equation}

From the first part of \eqref{eq-124}, for each $m=1,2$ we have

\begin{equation}\label{eq-125}
i_{p^m}\ast\Big(\gamma\ast i_{\operatorname{z}}\Big)=\phi^m.
\end{equation}

Moreover, from the second part of \eqref{eq-124} and interchange law, for each $m=1,2$ we have:

\begin{gather}
\nonumber i_{p^m}\ast\Big(\Big(i_{t'}\ast\mu^{-1}\Big)\odot\Big(\widetilde{\gamma}\ast
 i_{\widetilde{\operatorname{z}}}\Big)\odot\Big(i_t\ast\mu\Big)\Big)= \\
%%%
\label{eq-127} =\Big(i_{p^m\circ t'}\ast\mu^{-1}\Big)\odot\Big(\gamma^m\ast
 i_{\widetilde{\operatorname{u}}\circ\widetilde{\operatorname{z}}}\Big)\odot
 \Big(i_{p^m\circ t}\ast\mu\Big)=\gamma^m\ast i_{\operatorname{u}\circ\operatorname{z}}=\phi^m.
\end{gather}

Since \eqref{eq-106} is a weak fiber product in $\CATC$, then using condition
\textbf{\hyperref[A2]{A2}}$(G)$ 
together with \eqref{eq-126}, \eqref{eq-125} and \eqref{eq-127}, we get that 

\[\gamma\ast i_{\operatorname{z}}=\Big(i_{t'}\ast\mu^{-1}\Big)\odot\Big(\widetilde{\gamma}\ast
i_{\widetilde{\operatorname{z}}}\Big)\odot\Big(i_t\ast\mu\Big).\]

This equation is equivalent to \eqref{eq-32} when $\CATC$ is a $2$-category, so condition (c) holds
for each object $D$.\\

So Theorem~\ref{theo-04} implies that diagram \eqref{eq-91} is a weak fiber product in $\CATC\left[
\SETWinv\right]$. Then by Theorem~\ref{theo-02} for every pair of morphisms in $\SETW$ of the form
$\operatorname{w}^1:B^1\rightarrow\overline{B}^1$ and 
$\operatorname{w}^2:B^2\rightarrow\overline{B}^2$, the pair of morphisms $(B^1,\operatorname{w}^1,
f^1)$ and $(B^2,\operatorname{w}^2,f^2)$ has a weak fiber product in $\CATC\left[
\SETWinv\right]$.
\end{proof}

\section{(Strong) pullbacks in categories of fractions}\label{sec-02}
As we mentioned in the Introduction, the right bicalculus of fractions developed by Dorette Pronk
generalizes the usual right calculus of fractions described by
Pierre Gabriel and Michel Zisman
(see~\cite{GZ}). We refer to Appendix~\ref{sec-04} for more details on axioms
(\hyperref[CF]{CF}) for a right calculus of fractions and on the construction of a right category of
fractions. Then we can give a proof of the last result mentioned in the Introduction.

\begin{proof}[Proof of Proposition~\ref{prop-11}.]
We recall
(see Proposition~\ref{prop-08}) that given any category $\CATC$ and any class $\SETW$ of morphisms in
it, the pair $(\CATC,\SETW)$ satisfies the axioms for a right calculus of fractions if and only if
the pair $(\CATC^2,\SETW)$ satisfies the axioms for a right bicalculus of fractions (here
given any category $\CATC$, we denote by $\CATC^2$ the associated trivial bicategory).
If any of such conditions is satisfied, then there is an equivalence of bicategories

\[\functor{E}:\,\CATC^2\left[\SETWinv\right]\longrightarrow\left(\CATC\left[\SETWinv\right]\right)^2\]
given on objects as the identity and on any morphism $(A',\operatorname{w},f):A\rightarrow B$ as
$\functor{E}(A',\operatorname{w},f):=[A',\operatorname{w},f]$.\\

If we fix any pair of morphisms $g^1:B^1\rightarrow A$ and $g^2:B^2\rightarrow A$ in $\CATC^2
\left[\SETWinv\right]$; then the following facts are equivalent:

\begin{itemize}
 \item the pair $(g^1,g^2)$ has a weak fiber product in the bicategory $\CATC^2\left[\SETWinv
  \right]$;
 \item the pair $(\functor{E}(g^1),\functor{E}(g^2))$ has a weak fiber product in the bicategory
  $(\CATC\left[\SETWinv\right])^2$;
 \item the pair $(\functor{E}(g^1),\functor{E}(g^2))$ has a (strong) fiber product in the category
  $\CATC\left[\SETWinv\right]$.
\end{itemize}

The equivalence of the first $2$ conditions follows from the existence of $\functor{E}$ and
Proposition~\ref{prop-07}; the equivalence of the last $2$ conditions is simply Remark~\ref{rem-01}.\\

Since $\CATC$ is a category, considered as a trivial bicategory, then the $2$-morphism $\omega$ 
appearing in Theorems~\ref{theo-02} is a $2$-identity, i.e.\ $f^1\circ p^1=f^2\circ p^2$.
Using the previous set of equivalent conditions
and the equivalence of (\hyperref[i]{i}) and (\hyperref[ii]{ii}) in Theorem~\ref{theo-02}, this
implies at once the equivalence of (\hyperref[iii]{iii}) and (\hyperref[iv]{iv}) in
Proposition~\ref{prop-11}.\\

Now also all the $2$-morphisms appearing in Theorem~\ref{theo-04} are $2$-identities, hence saying
that there is
a $2$-morphism joining a pair of morphisms is equivalent to saying that such a pair of morphisms 
coincide. Moreover, all the identities from \eqref{eq-53} to \eqref{eq-32} are simply of the form
$i_a=i_a$ for some morphism $a$ in $\CATC$, hence they are automatically satisfied, so they will be 
ignored in the following lines. So let us fix any set of data $(C,p^1,p^2)$ such that $f^1\circ p^1=
f^2\circ p^2$. Then the following facts are equivalent:

\begin{enumerate}[(1)]
 \item for any object $D$, condition (\hyperref[a]{a}) of Theorem~\ref{theo-04} holds;
 \item condition (\hyperref[d]{d}) of Proposition~\ref{prop-11} holds.
\end{enumerate}

Moreover, also the following facts are equivalent:

\begin{enumerate}[(1)]
\setcounter{enumi}{2}
 \item for any object $D$, condition (\hyperref[b]{b}) of Theorem~\ref{theo-04} holds;
 \item given any object $R$ and any pair of morphisms $r,r':R\rightarrow C$ such that $p^m\circ r
  =p^m\circ r'$ for each $m=1,2$, there are an object $S$ and a morphism $\operatorname{h}:S
  \rightarrow R$ in $\SETW$, such that $r\circ\operatorname{h}=r'\circ\operatorname{h}$.
\end{enumerate}

In addition, the following facts are equivalent:

\begin{enumerate}[(1)]
\setcounter{enumi}{4}
  \item for any object $D$, condition (\hyperref[c]{c}) of Theorem~\ref{theo-04} holds;
  \item given any set of data $(R,r,r',S,\operatorname{h})$ as in (4), any object $\widetilde{S}$ and
   any morphism $\widetilde{\operatorname{h}}:\widetilde{S}\rightarrow
   R$ in $\SETW$ such that $r\circ\widetilde{\operatorname{h}}=r'\circ\widetilde{\operatorname{h}}$,
   there are an object $M$ and a pair of morphisms $\operatorname{d}:M\rightarrow S$ in $\SETW$ and
   $\widetilde{\operatorname{d}}:M\rightarrow\widetilde{S}$, such that $\operatorname{h}\circ
   \operatorname{d}=\widetilde{\operatorname{h}}\circ\widetilde{\operatorname{d}}$.
\end{enumerate}

Using condition (\hyperref[CF3]{CF3}) (with $f:=\operatorname{h}$ and $\operatorname{w}:=
\widetilde{\operatorname{h}}$), there are an object $M$ and a pair of morphisms $\operatorname{d}:M
\rightarrow S$ in $\SETW$ and $\widetilde{\operatorname{d}}:M\rightarrow\widetilde{S}$, such that
$\operatorname{h}\circ\operatorname{d}=\widetilde{\operatorname{h}}\circ
\widetilde{\operatorname{d}}$. So (6) is automatically satisfied, hence also (5) is true.
So using Theorem~\ref{theo-04} we have that \eqref{eq-83} is a (strong) fiber product if and only if
conditions (2) and (4) holds.\\

Now we claim that if we assume (2), then (4) is equivalent to:

\begin{enumerate}[(1)]
\setcounter{enumi}{6}
  \item condition (\hyperref[e]{e}) of Proposition~\ref{prop-11} holds.
\end{enumerate}

So first of all, let us assume (2) and (4) and let us prove that (7) holds. So let us fix any set
of data $(D,q^1,q^2,E,\operatorname{v},q,\widetilde{E},\widetilde{\operatorname{v}},\widetilde{q})$
as in Proposition~\ref{prop-11} (\hyperref[d]{d}) and (\hyperref[e]{e}) (in particular, such that
$q^m\circ\operatorname{v}=p^m\circ q$ and $q^m\circ\widetilde{\operatorname{v}}=p^m\circ
\widetilde{q}$ for each $m=1,2$). Let us apply axiom
(\hyperref[CF3]{CF3}) to the pair of morphisms $(\operatorname{v},\widetilde{\operatorname{v}})$.
Then there are an object $R$, a pair of morphisms $\operatorname{w}:R\rightarrow E$ in $\SETW$ and
$\widetilde{\operatorname{w}}:R\rightarrow\widetilde{E}$, such that $\operatorname{v}\circ
\operatorname{w}=\widetilde{\operatorname{v}}\circ\widetilde{\operatorname{w}}$. Then we set

\[r:=q\circ\operatorname{w}:R\longrightarrow C\quad\textrm{and}\quad r':=\widetilde{q}\circ
\widetilde{\operatorname{w}}:R\longrightarrow C.\]

Then for each $m=1,2$ we have:

\[p^m\circ r=p^m\circ q\circ\operatorname{w}=q^m\circ\operatorname{v}\circ\operatorname{w}=
q^m\circ\widetilde{\operatorname{v}}\circ\widetilde{\operatorname{w}}=p^m\circ\widetilde{q}\circ
\widetilde{\operatorname{w}}=p^m\circ r'.\]

So by (4) there are an object $F$ and a morphism $\operatorname{z}:F\rightarrow R$ in $\SETW$,
such that $r\circ\operatorname{z}=r'\circ\operatorname{z}$. 
We set $\operatorname{u}:=\operatorname{w}\circ\operatorname{z}$ and
$\widetilde{\operatorname{u}}:=\widetilde{\operatorname{w}}\circ\operatorname{z}$. So we have

\[\operatorname{v}\circ\operatorname{u}=\operatorname{v}\circ\operatorname{w}\circ\operatorname{z}=
\widetilde{\operatorname{v}}
\circ\widetilde{\operatorname{w}}\circ\operatorname{z}=\widetilde{\operatorname{v}}\circ
\widetilde{\operatorname{u}}\]
and

\[q\circ\operatorname{u}=
q\circ\operatorname{w}\circ\operatorname{z}=r\circ\operatorname{z}=r'\circ\operatorname{z}=
\widetilde{q}\circ\widetilde{\operatorname{w}}\circ\operatorname{z}=\widetilde{q}\circ
\widetilde{\operatorname{u}},\]
so (7) is satisfied.\\

Conversely, let us suppose that (2) and (7) hold and let us prove (4). So let us fix any
object $R$ and any pair of morphisms $r,r':R\rightarrow C$, such that $p^m\circ r=p^m\circ r'$ for
each $m=1,2$. Then let us set $q^m:=p^m\circ r:R\rightarrow B^m$ for $m=1,2$. Then condition
(\hyperref[d]{d}) of Proposition~\ref{prop-11} is satisfied is we choose
$E:=R$, $\operatorname{v}:=\id_R$ and $q:=r$. Moreover, (\hyperref[d]{d}) is also satisfied by
choosing $\widetilde{E}:=R$, $\widetilde{\operatorname{v}}:=\id_R$ and $\widetilde{q}:=r'$. Hence, by
(7) there are an object $S$ and a pair of morphisms $\operatorname{h},\widetilde{\operatorname{h}}:
S\rightarrow R$ in $\SETW$, such that $\id_R\circ\operatorname{h}=\id_R\circ
\widetilde{\operatorname{h}}$ and $r\circ\operatorname{h}=r'\circ\widetilde{\operatorname{h}}$.
This implies that
$r\circ\operatorname{h}=r'\circ\operatorname{h}$, so (4) holds. So \eqref{eq-83} is a (strong) fiber
product if and only if (2) and (4) hold, if and only if (2) and (7) hold. This is sufficient to
conclude.
\end{proof}

Proposition~\ref{prop-11} can also be obtained directly working in the category of fractions,
i.e.\ not relying on Theorems~\ref{theo-02} and~\ref{theo-04}. This gives a check of correctness for
the mentioned $2$ Theorems.

\appendix
\section{Bicategories of fractions}\label{sec-03}
In this and in the next appendix we will recall some basic notions about categories and bicategories
of fractions and we will list a series of lemmas used often in this paper.\\

Let us fix any bicategory $\CATC$ (with the notations already mentioned in \S~\ref{sec-07})
and any class $\SETW$ of morphisms in it. We recall
that $\SETW$ is said \emph{to admit a right bicalculus of fractions} if and only if
the following conditions are satisfied (see~\cite[\S~2.1]{Pr}):

\begin{enumerate}[({BF}1)]\label{BF}
 \item\label{BF1} for every object $A$ of $\CATC$, the $1$-identity $\id_A$ belongs to $\SETW$;
 \item\label{BF2} $\SETW$ is closed under compositions;
 \item\label{BF3} for every morphism $\operatorname{w}:A\rightarrow B$ in $\SETW$ and for every
  morphism $f:C\rightarrow B$, there are an object $D$, a morphism $\operatorname{w}':D\rightarrow
  C$ in $\SETW$, a morphism $f':D\rightarrow A$ and an invertible $2$-morphism $\alpha:
  f\circ\operatorname{w}'\Rightarrow\operatorname{w}\circ f'$;
 \item\label{BF4}
 \begin{enumerate}[(a)]
  \item\label{BF4a} given any morphism $\operatorname{w}:B\rightarrow A$ in $\SETW$, any pair of
   morphisms $f^1,f^2:C\rightarrow B$ and any $2$-morphism
   $\alpha:\operatorname{w}\circ f^1\Rightarrow
   \operatorname{w}\circ f^2$, there are an object $D$, a morphism $\operatorname{v}:D\rightarrow
   C$ in $\SETW$ and a $2$-morphism $\beta:f^1\circ\operatorname{v}\Rightarrow f^2\circ
   \operatorname{v}$, such that
   
   \[\alpha\ast i_{\operatorname{v}}=\thetaa{\operatorname{w}}{f^2}{\operatorname{v}}\odot\Big(
   i_{\operatorname{w}}\ast\beta\Big)\odot\thetab{\operatorname{w}}{f^1}{\operatorname{v}};\]
  
  \item\label{BF4b} if $\alpha$ in (a) is invertible, then so is $\beta$;
  \item\label{BF4c} if $(D',\operatorname{v}':D'\rightarrow C,\beta':f^1\circ\operatorname{v}'
   \Rightarrow f^2\circ\operatorname{v}')$ is another triple with the same properties of $(D,
   \operatorname{v},\beta)$ in (a), then there are an object $E$, a pair of morphisms
   $\operatorname{u}:E\rightarrow D$, $\operatorname{u}':E\rightarrow D'$ and an invertible
   $2$-morphism $\zeta:\operatorname{v}\circ\operatorname{u}\Rightarrow\operatorname{v}'\circ
   \operatorname{u}'$, such that $\operatorname{v}\circ\operatorname{u}$ belongs to $\SETW$ and
   
   \begin{gather*}
   \thetab{f^2}{\operatorname{v}'}{\operatorname{u}'}\odot\Big(\beta'\ast i_{\operatorname{u}'}
    \Big)\odot\thetaa{f^1}{\operatorname{v}'}{\operatorname{u}'}\odot\Big(i_{f^1}\ast\zeta\Big)= \\
%%%
    =\Big(i_{f^2}\ast\zeta\Big)\odot\thetab{f^2}{\operatorname{v}}{\operatorname{u}}\odot\Big(
   \beta\ast i_{\operatorname{u}}\Big)\odot\thetaa{f^1}{\operatorname{v}}{\operatorname{u}}.  
   \end{gather*}
  \end{enumerate}
 \item\label{BF5} if $\operatorname{w}:A\rightarrow B$ is a morphism in $\SETW$, $\operatorname{v}:A
  \rightarrow B$ is any morphism and if there exists an invertible $2$-morphism $\alpha:
  \operatorname{v}\Rightarrow\operatorname{w}$, then also $\operatorname{v}$ belongs to $\SETW$.
\end{enumerate}

We recall the following fundamental result:

\begin{theo}
\cite[Theorem~21]{Pr} Given any pair $(\CATC,\SETW)$ satisfying conditions
\emphatic{(\hyperref[BF]{BF})}, there are a bicategory $\CATC\left[\SETWinv\right]$ \emph{(}called
\emph{(right) bicategory of fractions}\emph{)} and a pseudofunctor $\functor{U}_{\SETW}:\CATC
\rightarrow\CATC\left[\SETWinv\right]$ that sends each element of $\SETW$ to an internal
equivalence and that is universal with respect to such property.
\end{theo}

In the notations of~\cite{Pr}, $\functor{U}_{\SETW}$ is called \emph{bifunctor}, but this notation
is no more in use; for the precise meaning of ``universal'' above, we refer directly to~\cite{Pr}.

\begin{rem}
In~\cite{Pr} the theorem above is stated with (\hyperref[BF1]{BF1}) replaced by the slightly stronger
hypothesis 

\begin{enumerate}[({BF}1)$'$]
\item\label{BF1prime} all the internal equivalences of $\CATC$ are in $\SETW$.
\end{enumerate}

By looking carefully at the proofs in~\cite{Pr}, it is easy to see that the only part of axiom
(\hyperref[BF1prime]{BF1})$'$ that is really used in all the computations is (\hyperref[BF1]{BF1}),
so we are allowed to state the theorem of~\cite{Pr} under such less restrictive hypothesis.
\end{rem}

In order to describe explicitly $\CATC\left[\SETWinv\right]$, one has to make some choices as below.
By~\cite[Theorem~21]{Pr}, \emph{different choices will give equivalent bicategories of
fractions} where objects, $1$-morphisms and $2$-morphisms are the same, but compositions of
$1$-morphisms and $2$-morphisms are (possibly) different.

\subsection{Choices in a bicategory of fractions}\label{sec-01}
Following~\cite[\S~2.2 and 2.3]{Pr} in order to construct a bicategory of fractions, we
have to fix a set of choices as follows: 

\begin{enumerate}[A$(\SETW)$:]
 \setcounter{enumi}{2}
 \item\label{C} for every set of data in $\CATC$ as follows

  \begin{equation}\label{eq-30}
  \begin{tikzpicture}[xscale=1.5,yscale=-1.2]
    \node (A0_0) at (0, 0) {$A'$};
    \node (A0_1) at (1, 0) {$B$};
    \node (A0_2) at (2, 0) {$B'$};
    
    \path (A0_0) edge [->]node [auto] {$\scriptstyle{f}$} (A0_1);
    \path (A0_2) edge [->]node [auto,swap] {$\scriptstyle{\operatorname{v}}$} (A0_1);
  \end{tikzpicture}
  \end{equation}
  with $\operatorname{v}$ in $\SETW$, using (\hyperref[BF3]{BF3}) we \emph{choose} an object $A''$,
  a pair of morphisms $\operatorname{v}':A''\rightarrow A'$ in $\SETW$ and $f':A''\rightarrow B'$ and
  an invertible $2$-morphism $\rho:f
  \circ\operatorname{v}'\Rightarrow\operatorname{v}\circ f'$ in $\CATC$.
\end{enumerate} 
  
The choices using (\hyperref[BF3]{BF3}) in general are not unique; following~\cite[\S~2.2]{Pr} we
have only to impose the following conditions:

\begin{enumerate}[({C}1)]
 \item\label{C1} whenever \eqref{eq-30} is such that $B=A'$ and $f=\id_B$, then we choose the data
   of \hyperref[C]{C}$(\SETW)$ to be given by $A'':=B'$, $f':=\id_B$, $\operatorname{v}':=
  \operatorname{v}$ and $\rho:=\pi^{-1}_{\operatorname{v}}\odot\upsilon_{\operatorname{v}}$;
 \item\label{C2} whenever \eqref{eq-30} is such that $B=B'$ and $\operatorname{v}=\id_B$, then we
   choose the data of \hyperref[C]{C}$(\SETW)$ to be given by $A'':=A'$, $f':=f$,
   $\operatorname{v}':=\id_{A'}$ and $\rho:=\upsilon^{-1}_f\odot\pi_f$.
\end{enumerate}

For simplicity of computations, in some of the proofs of this paper we will consider a set of choices
\hyperref[C]{C}$(\SETW)$ satisfying also the following additional condition:

\begin{enumerate}[({C}1)]
 \setcounter{enumi}{2}
 \item\label{C3} whenever \eqref{eq-30} is such that $A'=B'$ and $f=\operatorname{v}$ (with
  $\operatorname{v}$ in $\SETW$), then we choose the data of \hyperref[C]{C}$(\SETW)$ to be given by
  $A'':=A'$, $f':=\id_{A'}$, $\operatorname{v}':=\id_{A'}$ and $\rho:=i_{f\circ\id_{A'}}$.
\end{enumerate}

Condition (\hyperref[C3]{C3}) is not strictly necessary in order to do a right bicalculus of
fractions, but it simplifies lots of the computations in the present paper. We have only to check
that it is 
compatible with conditions (\hyperref[C1]{C1}) and (\hyperref[C2]{C2}) required by~\cite{Pr},
but this is obvious using the axioms of a bicategory.
% if A'=B'$, $f=\operatorname{v}$ and in addition $B=B'$ and $f=\operatorname{v}=\id_{A'}=\id_{B'}
% =\id_B$, then $\pi_{\operatorname{v}}\odot\upsilon_{\operatorname{v}}=i_{\id_{A'}\circ\id_{A'}}$,
% so (C3) is compatible with (C1). Analogously, (C3) is also compatible with (C2).
In other terms, for each pair $(\CATC,\SETW)$ satisfying condition
(\hyperref[BF3]{BF3}), there is always
a set of choices \hyperref[C]{C}$(\SETW)$ satisfying (\hyperref[C1]{C1}), (\hyperref[C2]{C2})
and (\hyperref[C3]{C3}).
According to~\cite[\S~2.3]{Pr} one should also fix an additional set of choices
depending on axiom (\hyperref[BF4]{BF4}), but actually such additional set of choices is not
necessary (see~\cite[Theorem~0.5]{T3}).

\subsection{Morphisms and 2-morphisms in $\CATC[\SETWinv]$}\label{sec-06}
We recall (see~\cite{Pr}) that the objects of $\CATC\left[\SETWinv\right]$ are the same as those
of $\CATC$.
A morphism from $A$ to $B$ in $\CATC\left[\SETWinv\right]$ is any triple $(A',\operatorname{w},f)$,
where $A'$ is an object of $\CATC$, $\operatorname{w}:A'\rightarrow A$ is an element of $\SETW$ and
$f:A'\rightarrow B$ is a morphism of $\CATC$. Given any pair of morphisms from $A$ to $B$ and from
$B$ to $C$ in $\CATC\left[\SETWinv\right]$ as follows

\[
\begin{tikzpicture}[xscale=1.5,yscale=-1.2]
    \node (A0_0) at (0, 0) {$A$};
    \node (A0_1) at (1, 0) {$A'$};
    \node (A0_2) at (2, 0) {$B$};
    \node (A0_3) at (3, 0) {$\textrm{and}$};
    \path (A0_1) edge [->]node [auto,swap] {$\scriptstyle{\operatorname{w}}$} (A0_0);
    \path (A0_1) edge [->]node [auto] {$\scriptstyle{f}$} (A0_2);
    
    \node (A0_4) at (4, 0) {$B$};
    \node (A0_5) at (5, 0) {$B'$};
    \node (A0_6) at (6, 0) {$C$};
    \path (A0_5) edge [->]node [auto,swap] {$\scriptstyle{\operatorname{v}}$} (A0_4);
    \path (A0_5) edge [->]node [auto] {$\scriptstyle{g}$} (A0_6);
\end{tikzpicture}
\]
(with both $\operatorname{w}$ and $\operatorname{v}$ in $\SETW$), one has to use choices
\hyperref[C]{C}$(\SETW)$ for the pair $(f,\operatorname{v})$ in order to get data
$(A'',\operatorname{v}',f')$ as above and then define the composition of the previous morphisms of
$\CATC\left[\SETWinv\right]$ as $(A'',\operatorname{w}\circ\operatorname{v}',g\circ f')$.\\

Given any pair of objects $A,B$ and any pair of morphisms $(A^m,\operatorname{w}^m,f^m):A\rightarrow
B$ for $m=1,2$, a $2$-morphism from $(A^1,\operatorname{w}^1,f^1)$ to $(A^2,\operatorname{w}^2,f^2)$
is an equivalence class of data $(A^3,\operatorname{v}^1,\operatorname{v}^2,\alpha,\beta)$ in
$\CATC$ as follows

\begin{equation}\label{eq-16}
\begin{tikzpicture}[xscale=2.2,yscale=-0.8]
    \node (A0_2) at (2, 0) {$A^1$};
    \node (A2_2) at (2, 2) {$A^3$};
    \node (A2_0) at (0, 2) {$A$};
    \node (A2_4) at (4, 2) {$B$,};
    \node (A4_2) at (2, 4) {$A^2$};
    
    \node (A2_3) at (2.8, 2) {$\Downarrow\,\beta$};
    \node (A2_1) at (1.2, 2) {$\Downarrow\,\alpha$};

    \path (A4_2) edge [->]node [auto,swap] {$\scriptstyle{f^2}$} (A2_4);
    \path (A0_2) edge [->]node [auto] {$\scriptstyle{f^1}$} (A2_4);
    \path (A2_2) edge [->]node [auto,swap] {$\scriptstyle{\operatorname{v}^1}$} (A0_2);
    \path (A2_2) edge [->]node [auto] {$\scriptstyle{\operatorname{v}^2}$} (A4_2);
    \path (A4_2) edge [->]node [auto] {$\scriptstyle{\operatorname{w}^2}$} (A2_0);
    \path (A0_2) edge [->]node [auto,swap] {$\scriptstyle{\operatorname{w}^1}$} (A2_0);
\end{tikzpicture}
\end{equation}
such that $\operatorname{w}^1\circ\operatorname{v}^1$ belongs to $\SETW$ and such that $\alpha$ is
invertible in $\CATC$ (in~\cite[\S~2.3]{Pr} it is also required that $\operatorname{w}^2\circ
\operatorname{v}^2$ belongs to $\SETW$, but this follows from (\hyperref[BF5]{BF5})). Any other set
of data

\[
\begin{tikzpicture}[xscale=2.2,yscale=-0.8]
    \node (A0_2) at (2, 0) {$A^1$};
    \node (A2_2) at (2, 2) {$A^{\prime 3}$};
    \node (A2_0) at (0, 2) {$A$};
    \node (A2_4) at (4, 2) {$B$};
    \node (A4_2) at (2, 4) {$A^2$};
    
    \node (A2_3) at (2.8, 2) {$\Downarrow\,\beta'$};
    \node (A2_1) at (1.2, 2) {$\Downarrow\,\alpha'$};

    \path (A4_2) edge [->]node [auto,swap] {$\scriptstyle{f^2}$} (A2_4);
    \path (A0_2) edge [->]node [auto] {$\scriptstyle{f^1}$} (A2_4);
    \path (A2_2) edge [->]node [auto,swap] {$\scriptstyle{\operatorname{v}^{\prime 1}}$} (A0_2);
    \path (A2_2) edge [->]node [auto] {$\scriptstyle{\operatorname{v}^{\prime 2}}$} (A4_2);
    \path (A4_2) edge [->]node [auto] {$\scriptstyle{\operatorname{w}^2}$} (A2_0);
    \path (A0_2) edge [->]node [auto,swap] {$\scriptstyle{\operatorname{w}^1}$} (A2_0);
\end{tikzpicture}
\]
(such that $\operatorname{w}^1\circ\operatorname{v}^{\prime 1}$ belongs to $\SETW$ and
$\alpha'$ is invertible) represents the same $2$-morphism in $\CATC\left[\SETWinv\right]$
if and only if there is a set of data $(A^4,\operatorname{z},\operatorname{z}',\sigma^1,
\sigma^2)$ in $\CATC$ as in the following diagram

\[
\begin{tikzpicture}[xscale=2.0,yscale=-0.8]
    \node (A0_4) at (4, 0) {$A^1$};
    \node (A2_2) at (3, 2) {$A^{\prime 3}$};
    \node (A2_4) at (4, 2) {$A^4$};
    \node (A2_6) at (5, 2) {$A^3$,};
    \node (A4_4) at (4, 4) {$A^2$};
    
    \node (A1_4) at (4, 1.4) {$\Rightarrow$};
    \node (B1_4) at (4, 1) {$\sigma^1$};
    \node (A3_4) at (4, 3) {$\Leftarrow$};
    \node (B3_4) at (4, 2.6) {$\sigma^2$};

    \path (A2_2) edge [->]node [auto,swap] {$\scriptstyle{\operatorname{v}^{\prime 2}}$} (A4_4);
    \path (A2_4) edge [->]node [auto] {$\scriptstyle{\operatorname{z}}$} (A2_6);
    \path (A2_6) edge [->]node [auto] {$\scriptstyle{\operatorname{v}^2}$} (A4_4);
    \path (A2_2) edge [->]node [auto] {$\scriptstyle{\operatorname{v}^{\prime 1}}$} (A0_4);
    \path (A2_4) edge [->]node [auto,swap] {$\scriptstyle{\operatorname{z}'}$} (A2_2);
    \path (A2_6) edge [->]node [auto,swap] {$\scriptstyle{\operatorname{v}^1}$} (A0_4);
\end{tikzpicture}
\]
such that $(\operatorname{w}^1\circ\operatorname{v}^1)\circ\operatorname{z}$ belongs to $\SETW$,
$\sigma^1$ and $\sigma^2$ are both invertible,

\begin{gather}
\nonumber \Big(i_{\operatorname{w}^2}\ast\sigma^2\Big)\odot\thetab{\operatorname{w}^2}
 {\operatorname{v}^2}{\operatorname{z}}\odot\Big(\alpha\ast i_{\operatorname{z}}\Big)\odot
 \thetaa{\operatorname{w}^1}{\operatorname{v}^1}{\operatorname{z}}\odot\Big(i_{\operatorname{w}^1}
 \ast\sigma^1\Big)= \\
%%%
\label{eq-19} =\thetab{\operatorname{w}^2}{\operatorname{v}^{\prime 2}}{\operatorname{z}'}\odot
 \Big(\alpha'\ast i_{\operatorname{z}'}\Big)\odot\thetaa{\operatorname{w}^1}
 {\operatorname{v}^{\prime 1}}{\operatorname{z}'}  
\end{gather}
and

\begin{gather}
\nonumber \Big(i_{f^2}\ast\sigma^2\Big)\odot\thetab{f^2}{\operatorname{v}^2}{\operatorname{z}}
 \odot\Big(\beta\ast i_{\operatorname{z}}\Big)\odot\thetaa{f^1}{\operatorname{v}^1}
 {\operatorname{z}}\odot\Big(i_{f^1}\ast\sigma^1\Big)= \\
%%%
\label{eq-26} =\thetab{f^2}{\operatorname{v}^{\prime 2}}{\operatorname{z}'}\odot\Big(
 \beta'\ast i_{\operatorname{z}'}\Big)\odot\thetaa{f^1}{\operatorname{v}^{\prime 1}}
 {\operatorname{z}'}  
\end{gather}
(in~\cite[\S~2.3]{Pr} it is also required that $(\operatorname{w}^1\circ\operatorname{v}^{\prime 1})
\circ\operatorname{z}'$ belongs to $\SETW$, but this follows from (\hyperref[BF5]{BF5})). We denote by

\[\Big[A^3,\operatorname{v}^1,\operatorname{v}^2,\alpha,\beta\Big]:\Big(A^1,\operatorname{w}^1,
f^1\Big)\Longrightarrow\Big(A^2,\operatorname{w}^2,f^2\Big)\]
the class of any data as in \eqref{eq-16}. We refer to~\cite{Pr} for the description of associators
and compositions of $2$-morphisms in $\CATC\left[\SETWinv\right]$. A simplified
description can be found in~\cite[Propositions~0.1, 0.2, 0.3 and~0.4]{T3}.\\

\subsection{Useful lemmas in a bicategory of fractions}
We denote by $\Theta_{\bullet}$ the associators of a bicategory of fractions
$\CATC\left[\SETWinv\right]$ (constructed as in~\cite[Appendix A.2]{Pr}). Then we have:

\begin{lem}\label{lem-10}
\cite[Corollary~2.2 and Remark~2.3]{T3}
Let us fix any triple of morphisms $h:D\rightarrow C,g:C\rightarrow B,f:B\rightarrow A$ in $\CATC$
and any morphism $\operatorname{w}:D\rightarrow D'$ in $\SETW$. If $\CATC$ is a $2$-category,
then the associator $\Thetaa{(B,\id_B,f)}{(C,\id_C,g)}{(D,\operatorname{w},h)}$
coincides with the $2$-identity of the morphism $(D,\operatorname{w},f\circ g\circ h):D'
\rightarrow A$ in $\CATC\left[\SETWinv\right]$.
\end{lem}

The following is a special case of~\cite[Proposition~0.3]{T3}.

\begin{lem}\label{lem-09}
Let us fix any morphism and any representative of a $2$-morphism in $\CATC\left[\SETWinv\right]$
as follows.

\[
\begin{tikzpicture}[xscale=1.6,yscale=-0.8]
    \node (B0_4) at (-3, 2) {$A$};
    \node (B0_3) at (-2, 2) {$A'$};
    \node (B0_2) at (-1, 2) {$B$,};
    \path (B0_3) edge [->]node [auto,swap] {$\scriptstyle{\operatorname{w}}$} (B0_4);
    \path (B0_3) edge [->]node [auto] {$\scriptstyle{g}$} (B0_2);

    \node (A0_2) at (2, 0) {$B$};
    \node (A2_2) at (2, 2) {$B$};
    \node (A2_0) at (0, 2) {$B$};
    \node (A2_4) at (4, 2) {$C$.};
    \node (A4_2) at (2, 4) {$B$};
    
    \node (A2_3) at (2.8, 2) {$\Downarrow\,\gamma$};
    \node (A2_1) at (1.2, 2) {$\Downarrow\,i_{\id_B\circ\id_B}$};

    \path (A4_2) edge [->]node [auto,swap] {$\scriptstyle{h^2}$} (A2_4);
    \path (A0_2) edge [->]node [auto] {$\scriptstyle{h^1}$} (A2_4);
    \path (A2_2) edge [->]node [auto,swap] {$\scriptstyle{\id_B}$} (A0_2);
    \path (A2_2) edge [->]node [auto] {$\scriptstyle{\id_B}$} (A4_2);
    \path (A4_2) edge [->]node [auto] {$\scriptstyle{\id_B}$} (A2_0);
    \path (A0_2) edge [->]node [auto,swap] {$\scriptstyle{\id_B}$} (A2_0);
\end{tikzpicture}
\]

Then the $2$-morphism

\[\Big[B,\id_B,\id_B,i_{\id_B},\gamma\Big]\ast i_{(A',\operatorname{w},g)}:\,\Big(A',
\operatorname{w}\circ\id_{A'},h^1\circ g\Big)\Longrightarrow\Big(A',\operatorname{w}\circ
\id_{A'},h^2\circ g\Big)\]
is equal to $[A',\id_{A'},\id_{A'},i_{(\operatorname{w}\circ\id_{A'})\circ\id_{A'}},
\pi_{h^2\circ g}^{-1}\odot((\pi_{h^2}\odot\gamma\odot\pi_{h^1}^{-1})\ast i_g)\odot\pi_{h^1\circ g}]$.
\end{lem}

The following is a special case of~\cite[Proposition~0.4]{T3}.

\begin{lem}\label{lem-12}
Let us fix any morphism and any representative of a $2$-morphism in $\CATC\left[\SETWinv\right]$
in $\CATC\left[\SETWinv\right]$ as follows:

\[
\begin{tikzpicture}[xscale=1.6,yscale=-0.8]
    \node (A0_2) at (2, 0) {$A^1$};
    \node (A2_2) at (2, 2) {$A^3$};
    \node (A2_0) at (0, 2) {$A$};
    \node (A2_4) at (4, 2) {$B$,};
    \node (A4_2) at (2, 4) {$A^2$};
    
    \node (A2_3) at (2.8, 2) {$\Downarrow\,\beta$};
    \node (A2_1) at (1.2, 2) {$\Downarrow\,\alpha$};

    \path (A4_2) edge [->]node [auto,swap] {$\scriptstyle{f^2}$} (A2_4);
    \path (A0_2) edge [->]node [auto] {$\scriptstyle{f^1}$} (A2_4);
    \path (A2_2) edge [->]node [auto,swap] {$\scriptstyle{\operatorname{u}^1}$} (A0_2);
    \path (A2_2) edge [->]node [auto] {$\scriptstyle{\operatorname{u}^2}$} (A4_2);
    \path (A4_2) edge [->]node [auto] {$\scriptstyle{\operatorname{w}^2}$} (A2_0);
    \path (A0_2) edge [->]node [auto,swap] {$\scriptstyle{\operatorname{w}^1}$} (A2_0);
    \node (B0_4) at (5, 2) {$B$};
    \node (B0_3) at (6, 2) {$B$};
    \node (B0_2) at (7, 2) {$C$.};
    \path (B0_3) edge [->]node [auto,swap] {$\scriptstyle{\id_B}$} (B0_4);
    \path (B0_3) edge [->]node [auto] {$\scriptstyle{g}$} (B0_2);
\end{tikzpicture}
\]

Then the $2$-morphism 

\[i_{(B,\id_B,g)}\ast\Big[A^3,\operatorname{u}^1,\operatorname{u}^2,\alpha,\beta\Big]:\,\Big(A^1,
\operatorname{w}^1\circ\id_{A^1},g\circ f^1\Big)\Longrightarrow\Big(A^2,\operatorname{w}^2
\circ\id_{A^2},g\circ f^2\Big)\]
is equal to $[A^3,\operatorname{u}^1,\operatorname{u}^2,(\pi_{\operatorname{w}^2}^{-1}\ast
i_{\operatorname{u}^2})\odot\alpha\odot(\pi_{\operatorname{w}^1}\ast
i_{\operatorname{u}^1}),\thetaa{g}{f^2}{\operatorname{u}^2}
\odot(i_g\ast\beta)\odot\thetab{g}{f^1}{\operatorname{u}^1}]$.
\end{lem}

The following is a simple application of the definition of right saturation
(see~\cite[Definition~2.11]{T4}) together with~\cite[Proposition~2.11]{T4}.

\begin{lem}\label{lem-06}
Let us fix any triple of objects $A,B,C$, and any pair of morphisms $\operatorname{w}:B\rightarrow A$
and $\operatorname{v}:C\rightarrow B$, such that both $\operatorname{w}$ and $\operatorname{w}
\circ\operatorname{v}$ belong to $\SETW$. Then there are an object $D$ and a morphism
$\operatorname{z}:D\rightarrow C$, such that $\operatorname{v}\circ\operatorname{z}$ belongs
to $\SETW$.
\end{lem}

\begin{lem}\label{lem-19}
Let us fix any set of objects $A,A',B$, any morphism $\operatorname{w}:A'\rightarrow A$ in
$\SETW$ and any pair of morphisms $f^1,f^2:A'\rightarrow B$. Let us also fix any pair of
$2$-morphisms in $\CATC\left[\SETWinv\right]$

\[\Gamma,\,\Gamma':\Big(A',\operatorname{w},f^1\Big)\Longrightarrow\Big(A',\operatorname{w},f^2\Big)\]
and let us suppose that $\Gamma=[C,\operatorname{v},\operatorname{v},i_{\operatorname{w}\circ
\operatorname{v}},\gamma]$ and $\Gamma'=[C',\operatorname{v}',\operatorname{v}',i_{\operatorname{w}
\circ\operatorname{v}'},\gamma']$ \emphatic{(}for some choice of $C,C',\operatorname{v},
\operatorname{v}',\gamma$ and $\gamma'$\emphatic{)}. Then
$\Gamma=\Gamma'$ if and only if there are an object $D$, a morphism $\operatorname{z}:
D\rightarrow C$ in $\SETW$, a morphism $\operatorname{z}':D\rightarrow C'$
and an invertible $2$-morphism
$\mu:\operatorname{v}\circ\operatorname{z}\Rightarrow\operatorname{v}'\circ\operatorname{z}'$,
such that
  
\begin{gather*}
\Big(i_{f^2}\ast\mu\Big)\odot\thetab{f^2}{\operatorname{v}}{\operatorname{z}}\odot
\Big(\gamma\ast i_{\operatorname{z}}\Big)\odot\thetaa{f^1}{\operatorname{v}}{\operatorname{z}}\odot
\Big(i_{f^1}\ast\mu^{-1}\Big)= \\
%%%
=\thetab{f^2}{\operatorname{v}'}{\operatorname{z}'}\odot\Big(\gamma'\ast
i_{\operatorname{z}'}\Big)\odot\thetaa{f^1}{\operatorname{v}'}{\operatorname{z}'}.
\end{gather*}
\end{lem}

\begin{proof}
The ``if'' part is a direct consequence of the definition of $2$-morphism in $\CATC\left[
\SETWinv\right]$. Conversely, let us suppose that $\Gamma=\Gamma'$. Then 
there is a set of data $(\overline{C},\operatorname{t},\operatorname{t}',\sigma^1,
\sigma^2)$ in $\CATC$ as in the following diagram

\[
\begin{tikzpicture}[xscale=2.0,yscale=-0.8]
    \node (A0_4) at (4, 0) {$A'$};
    \node (A2_2) at (3, 2) {$C'$};
    \node (A2_4) at (4, 2) {$\overline{C}$};
    \node (A2_6) at (5, 2) {$C$,};
    \node (A4_4) at (4, 4) {$A'$};
    
    \node (A1_4) at (4, 1.4) {$\Rightarrow$};
    \node (B1_4) at (4, 1) {$\sigma^1$};
    \node (A3_4) at (4, 3) {$\Leftarrow$};
    \node (B3_4) at (4, 2.6) {$\sigma^2$};

    \path (A2_2) edge [->]node [auto,swap] {$\scriptstyle{\operatorname{v}'}$} (A4_4);
    \path (A2_4) edge [->]node [auto] {$\scriptstyle{\operatorname{t}}$} (A2_6);
    \path (A2_6) edge [->]node [auto] {$\scriptstyle{\operatorname{v}}$} (A4_4);
    \path (A2_2) edge [->]node [auto] {$\scriptstyle{\operatorname{v}'}$} (A0_4);
    \path (A2_4) edge [->]node [auto,swap] {$\scriptstyle{\operatorname{t}'}$} (A2_2);
    \path (A2_6) edge [->]node [auto,swap] {$\scriptstyle{\operatorname{v}}$} (A0_4);
\end{tikzpicture}
\]
such that $(\operatorname{w}\circ\operatorname{v})\circ\operatorname{t}$ belongs to $\SETW$,
$\sigma^1$ and $\sigma^2$ are both invertible,

\begin{gather}
\nonumber \Big(i_{\operatorname{w}}\ast\sigma^2\Big)\odot\thetab{\operatorname{w}}
 {\operatorname{v}}{\operatorname{t}}\odot\Big(i_{\operatorname{w}\circ\operatorname{v}}
 \ast i_{\operatorname{t}}\Big)\odot
 \thetaa{\operatorname{w}}{\operatorname{v}}{\operatorname{t}}\odot\Big(i_{\operatorname{w}}
 \ast\sigma^1\Big)= \\
%%%
\label{eq-117} =\thetab{\operatorname{w}}{\operatorname{v}'}{\operatorname{t}'}\odot
 \Big(i_{\operatorname{w}\circ\operatorname{v}'}\ast i_{\operatorname{t}'}\Big)\odot
 \thetaa{\operatorname{w}}{\operatorname{v}'}{\operatorname{t}'}  
\end{gather}
and

\begin{gather}
\nonumber \Big(i_{f^2}\ast\sigma^2\Big)\odot\thetab{f^2}{\operatorname{v}}{\operatorname{t}}
 \odot\Big(\gamma\ast i_{\operatorname{t}}\Big)\odot\thetaa{f^1}{\operatorname{v}}
 {\operatorname{t}}\odot\Big(i_{f^1}\ast\sigma^1\Big)= \\
%%%
\label{eq-118} =\thetab{f^2}{\operatorname{v}'}{\operatorname{t}'}\odot\Big(
 \gamma'\ast i_{\operatorname{t}'}\Big)\odot\thetaa{f^1}{\operatorname{v}'}
 {\operatorname{t}'}  
\end{gather}

Since $\Gamma$ is a $2$-morphism in a bicategory of fractions, then $\operatorname{w}\circ
\operatorname{v}$ belongs to $\SETW$. Since also $(\operatorname{w}\circ\operatorname{v})\circ
\operatorname{t}$ belongs to $\SETW$, then by Lemma~\ref{lem-06} there are an object $\widetilde{C}$
and a morphism $\operatorname{r}:\widetilde{C}\rightarrow\overline{C}$, such that $\operatorname{t}
\circ\operatorname{r}$ belongs to $\SETW$.\\

From \eqref{eq-117} we get that $i_{\operatorname{w}}\ast(\sigma^2\ast i_{\operatorname{r}})=
i_{\operatorname{w}}\ast(\sigma^1\ast i_{\operatorname{r}})^{-1}$. So using~\cite[Lemma~1.1]{T3}
there are an object $D$ and a morphism
$\operatorname{s}:D\rightarrow\widetilde{C}$ in $\SETW$, such that

\begin{equation}\label{eq-04}
\Big(\sigma^2\ast i_{\operatorname{r}}\Big)\ast i_{\operatorname{s}}=\Big(\sigma^1\ast
i_{\operatorname{r}}\Big)^{-1}\ast i_{\operatorname{s}}.
\end{equation}

By construction and (\hyperref[BF2]{BF2}) the morphism $(\operatorname{t}\circ\operatorname{r})\circ
\operatorname{s}$ belongs to $\SETW$, so by (\hyperref[BF5]{BF5}) we conclude that also the morphism
$\operatorname{z}:=\operatorname{t}\circ(\operatorname{r}\circ\operatorname{s}):D\rightarrow C$
belongs to $\SETW$. We define also $\operatorname{z}':=\operatorname{t}'\circ(\operatorname{r}
\circ\operatorname{s}):D\rightarrow C'$ and

\[\mu:=\thetab{\operatorname{v}'}{\operatorname{t}'}{\operatorname{r}\circ\operatorname{s}}\odot
\Big(\sigma^2\ast i_{\operatorname{r}\circ\operatorname{s}}\Big)\odot
\thetaa{\operatorname{v}}{\operatorname{t}}{\operatorname{r}\circ\operatorname{s}}:\,
\operatorname{v}\circ\operatorname{z}\Longrightarrow\operatorname{v}'\circ\operatorname{z}'.\]

Then we conclude using \eqref{eq-04} and \eqref{eq-118}.
\end{proof}

\section{Categories of fractions}\label{sec-04}

We recall (see~\cite{GZ}) that given a category $\CATC$ and a class $\SETW$ of morphisms in it, the
pair $(\CATC,\SETW)$ is said to \emph{admit a right calculus of fractions} if and only if
the following properties hold:

\begin{enumerate}[({CF}1)]\label{CF}
 \item\label{CF1} $\SETW$ contains all the identities of $\CATC$;
 \item\label{CF2} $\SETW$ is closed under compositions;
 \item\label{CF3} (``right Ore condition'') for every morphism $\operatorname{w}:A\rightarrow B$ in
  $\SETW$ and any morphism $f:C\rightarrow B$, there are an object $D$, a morphism
  $\operatorname{w}':D\rightarrow C$ in $\SETW$ and a morphism $f':D\rightarrow A$, such that
  $f\circ\operatorname{w}'=\operatorname{w}\circ f'$;
 \item\label{CF4} (``right cancellability'') given any morphism $\operatorname{w}:B\rightarrow A$ in
  $\SETW$ and any pair of morphisms $f^1,f^2:C\rightarrow B$ such that $\operatorname{w}\circ f^1=
  \operatorname{w}\circ f^2$, there are an object $D$ and a morphism $\operatorname{v}:D\rightarrow C$
  in $\SETW$, such that $f^1\circ\operatorname{v}=f^2\circ\operatorname{v}$.
\end{enumerate}

Given any pair $(\CATC,\SETW)$ satisfying this set of axioms, the right
category of fractions $\CATC\left[
\SETWinv\right]$ associated to it is described as follows. Its objects are the same as those of
$\CATC$; a morphism from $A$ to $B$ is any equivalence class $[A',\operatorname{w},f]$ of a
triple $(A',\operatorname{w},f)$ as follows:

\[
\begin{tikzpicture}[xscale=1.5,yscale=-1.2]
    \node (A0_0) at (0, 0) {$A$};
    \node (A0_1) at (1, 0) {$A'$};
    \node (A0_2) at (2, 0) {$B$};
    
    \path (A0_1) edge [->]node [auto,swap] {$\scriptstyle{\operatorname{w}}$} (A0_0);
    \path (A0_1) edge [->]node [auto] {$\scriptstyle{f}$} (A0_2);
\end{tikzpicture}
\]
with $\operatorname{w}$ in $\SETW$. Any $2$ triples $(A^1,\operatorname{w}^1,f^1)$ and 
$(A^2,\operatorname{w}^2,f^2)$ (both defined from $A$ to $B$) are declared equivalent if and only
if there are an object $A^3$, a pair of morphisms $\operatorname{v}^1:A^3\rightarrow A^1$ and
$\operatorname{v}^2:A^3\rightarrow A^2$, such that:

\begin{itemize}
 \item $\operatorname{w}^1\circ\operatorname{v}^1$ belongs to $\SETW$;
 \item $\operatorname{w}^1\circ\operatorname{v}^1=\operatorname{w}^2\circ\operatorname{v}^2$;
 \item $f^1\circ\operatorname{v}^1=f^2\circ\operatorname{v}^2$;
\end{itemize}
(the fact that this is an equivalence relation is obvious using the axioms). The composition of
morphisms in $\CATC\left[\SETWinv\right]$ is obtained by choosing representatives, then using 
(\hyperref[CF3]{CF3}) and then taking the class of the resulting composition. As such, composition
is well-defined and associative.\\

Now given any category $\CATC$, we denote by $\CATC^2$ the trivial bicategory obtained from
$\CATC$, i.e.\ the bicategory whose objects and morphisms are the same as those of $\CATC$ and whose
$2$-morphisms are only the $2$-identities. Then a direct check proves the following fact.

\begin{prop}\label{prop-08}
Let us fix any category $\CATC$ and any class $\SETW$ of morphisms in it. Then the pair $(\CATC,
\SETW)$ satisfies the axioms for a right calculus of fractions if and only if the pair $(\CATC^2,
\SETW)$ satisfies the axioms for a right bicalculus of fractions. If any of such conditions is
satisfied, then:

\begin{enumerate}[\emphatic{(}a\emphatic{)}]
 \item given any pair of objects $A,B$ in $\CATC$ and any pair of morphisms $(A^m,
  \operatorname{w}^m,f^m):A\rightarrow B$ for $m=1,2$ in $\CATC^2\left[\SETWinv\right]$, if $[A^1,
  \operatorname{w}^1,f^1]=[A^2,\operatorname{w}^2,f^2]$ in $\CATC\left[\SETWinv\right]$ then
  there is exactly one $2$-morphism $\Gamma$ from $(A^1,\operatorname{w}^1,f^1)$ to $(A^2,
  \operatorname{w}^2,f^2)$ in $\CATC^2\left[\SETWinv\right]$; moreover such a $\Gamma$ is invertible;
 \item given any set of data $(A,B,A^m,\operatorname{w}^m,f^m)$ as before, if $[A^1,
  \operatorname{w}^1,f^1]\neq[A^2,\operatorname{w}^2,f^2]$ in $\CATC\left[\SETWinv\right]$, then
  there are no $2$-morphisms from $(A^1,\operatorname{w}^1,f^1)$ to $(A^2,\operatorname{w}^2,f^2)$ in
  $\CATC^2\left[\SETWinv\right]$; 
 \item there is an equivalence of bicategories

  \[\functor{E}:\,\CATC^2\left[\SETWinv\right]\longrightarrow\left(\CATC\left[\SETWinv\right]
  \right)^2\]
%%%
  given on objects as the identity, on any morphism $(A',\operatorname{w},f):A\rightarrow B$ as
  $\functor{E}(A',\operatorname{w},f):=[A',\operatorname{w},f]$ and induced on $2$-morphisms by
  \emphatic{(}a\emphatic{)} and \emphatic{(}b\emphatic{)}.
\end{enumerate}
\end{prop}

This makes precise the informal concept (stated in the Introduction) that the right bicalculus of
fractions generalizes the right calculus of fractions.

\section{Proofs of some technical lemmas}\label{sec-05}

\begin{proof}[Proof of Lemma~\ref{lem-16}.]
As usual, we give a complete proof assuming for simplicity that $\CATD$ is a $2$-category;
in this case $\overline{\Omega}=i_e\ast\Omega$.
Since $e$ is an internal equivalence, then by~\cite[Proposition~1.5.7]{L} $e$ is the first
component of an adjoint equivalence. So there are an internal equivalence $d:\overline{A}
\rightarrow A$ and invertible $2$-morphisms $\Delta:\id_A\Longrightarrow d\circ e$ and
$\Xi:e\circ d\Longrightarrow\id_{\overline{A}}$ such that

\begin{equation}\label{eq-137}
\Big(\Xi\ast i_e\Big)\odot\Big(i_e\ast\Delta\Big)=i_e\quad\textrm{and}\quad\Big(i_d\ast\Xi\Big)\odot
\Big(\Delta\ast i_d\Big)=i_d.
\end{equation}

Let us fix any object $D$ in $\CATD$ and let us prove conditions \textbf{\hyperref[A1]{A1}}$(D)$ and
\textbf{\hyperref[A2]{A2}}$(D)$ for diagram \eqref{eq-134}. First of all, we prove 
\textbf{\hyperref[A1]{A1}}$(D)$, so we fix any set of data
$(s^1,s^2,\overline{\Lambda})$ with $\overline{\Lambda}$ invertible as follows

\[
\begin{tikzpicture}[xscale=1.5,yscale=-0.8]
    \node (A0_0) at (0, 0) {$D$};
    \node (A0_2) at (2, 0) {$B^1$};
    \node (A2_0) at (0, 2) {$B^2$};
    \node (A2_2) at (2, 2) {$\overline{A}$.};
    
    \node (A1_1) [rotate=225] at (0.9, 1) {$\Longrightarrow$};
    \node (B1_1) at (1.2, 1) {$\overline{\Lambda}$};
    
    \path (A0_0) edge [->]node [auto,swap] {$\scriptstyle{s^2}$} (A2_0);
    \path (A0_0) edge [->]node [auto] {$\scriptstyle{s^1}$} (A0_2);
    \path (A0_2) edge [->]node [auto] {$\scriptstyle{e\circ g^1}$} (A2_2);
    \path (A2_0) edge [->]node [auto,swap] {$\scriptstyle{e\circ g^2}$} (A2_2);
\end{tikzpicture}
\]

Then we consider the invertible $2$-morphism

\begin{equation}\label{eq-43}
\Lambda:=\Big(\Delta^{-1}\ast i_{g^2\circ s^2}\Big)\odot\Big(i_d\ast\overline{\Lambda}\Big)\odot
\Big(\Delta\ast i_{g^1\circ s^1}\Big):\,\,g^1\circ s^1\Longrightarrow g^2\circ s^2.
\end{equation}

Since \eqref{eq-60} satisfies condition \textbf{\hyperref[A1]{A1}}$(D)$, then
there are a morphism $s:D\rightarrow C$ and a pair of
invertible $2$-morphisms $\Lambda^m:s^m\Rightarrow r^m\circ s$ for $m=1,2$, such that

\begin{equation}\label{eq-135}
\Big(\Omega\ast i_s\Big)\odot\Big(i_{g^1}\ast\Lambda^1\Big)=\Big(i_{g^2}\ast\Lambda^2\Big)\odot
\Lambda.
\end{equation}

Now by interchange law we have

\begin{gather}
\nonumber i_e\ast\Lambda\stackrel{\eqref{eq-43}}{=}\Big(i_e\ast\Delta^{-1}\ast i_{g^2\circ s^2}
 \Big)\odot\Big(i_{e\circ d}\ast\overline{\Lambda}\Big)\odot\Big(i_e\ast\Delta\ast
 i_{g^1\circ s^1}\Big)\stackrel{\eqref{eq-137}}{=} \\
%%%
\label{eq-138} \stackrel{\eqref{eq-137}}{=}\Big(\Xi\ast i_{e\circ g^2\circ s^2}\Big)\odot\Big(
 i_{e\circ d}\ast\overline{\Lambda}\Big)\odot\Big(\Xi^{-1}\ast i_{e\circ g^1\circ s^1}\Big)=
 \overline{\Lambda},
\end{gather}
so

\begin{gather*}
\Big(\overline{\Omega}\ast i_s\Big)\odot\Big(i_{e\circ g^1}\ast\Lambda^1\Big)=i_e\ast\Big(\Big(
 \Omega\ast i_s\Big)\odot\Big(i_{g^1}\ast\Lambda^1\Big)\Big)
 \stackrel{\eqref{eq-135}}{=} \\
%%%
\stackrel{\eqref{eq-135}}{=}\Big(i_{e\circ g^2}\ast\Lambda^2\Big)\odot\Big(i_e\ast\Lambda\Big)
 \stackrel{\eqref{eq-138}}{=}\Big(i_{e\circ g^2}\ast\Lambda^2\Big)\odot\overline{\Lambda}.
\end{gather*}

This proves that condition \textbf{\hyperref[A1]{A1}}$(D)$
holds for diagram \eqref{eq-134}. Let us also prove condition
\textbf{\hyperref[A2]{A2}}$(D)$, so let us fix any pair of morphisms $t,t':D\rightarrow C$ and
any pair of invertible $2$-morphisms $\Gamma^m:r^m\circ t\Rightarrow r^m\circ t'$ for $m=1,2$, such
that

\begin{equation}\label{eq-132}
\Big(\overline{\Omega}\ast i_{t'}\Big)\odot\Big(i_{e\circ g^1}\ast\Gamma^1\Big)=\Big(i_{e\circ
g^2}\ast\Gamma^2\Big)\odot\Big(\overline{\Omega}\ast i_t\Big).
\end{equation}

Then by interchange law we have

\begin{gather*}
\Big(\Omega\ast i_{t'}\Big)\odot\Big(i_{g^1}\ast\Gamma^1\Big)= \\
%%%
=\Big(\Delta^{-1}\ast i_{g^2\circ r^2\circ t'}\Big)\odot\Big\{i_{d\circ e}\ast\Big[\Big(\Omega\ast
 i_{t'}\Big)\odot\Big(i_{g^1}\ast\Gamma^1\Big)\Big]\Big\}\odot\Big(\Delta\ast i_{g^1\circ r^1\circ t}
 \Big)\stackrel{\eqref{eq-132}}{=} \\
%%%
=\Big(\Delta^{-1}\ast i_{g^2\circ r^2\circ t'}\Big)\odot\Big\{i_{d\circ e}\ast\Big[\Big(i_{g^2}\ast
 \Gamma^2\Big)\odot\Big(\Omega\ast i_t\Big)\Big]\Big\}\odot\Big(\Delta\ast i_{g^1\circ r^1\circ t}
 \Big)= \\
%%%
=\Big(i_{g^2}\ast\Gamma^2\Big)\odot\Big(\Omega\ast i_t\Big).
\end{gather*}

Since \eqref{eq-60} satisfies condition \textbf{\hyperref[A2]{A2}}$(D)$, then
there is a unique invertible $2$-morphism
$\Gamma:t\Rightarrow t'$, such that $i_{r^m}\ast\Gamma=\Gamma^m$ for each $m=1,2$. This proves that
condition \textbf{\hyperref[A2]{A2}}$(D)$ holds also for diagram \eqref{eq-134}.
\end{proof}

\begin{proof}[Proof of Lemma~\ref{lem-03}.]
As usual, we give the proof in the case when $\CATD$ is a $2$-category.
Since $\Omega^1$ and $\Omega^2$ are invertible, then the roles of $(g^1,g^2)$ and $(\overline{g}^1,
\overline{g}^2)$ are interchangeable. Hence, we will only prove that if \eqref{eq-60} is
a weak fiber product, then \eqref{eq-21} is also a weak fiber product.\\

So let us fix any object $D$ in $\CATD$ and let us prove condition \textbf{\hyperref[A1]{A1}}$(D)$
for \eqref{eq-21}, so let us consider any set of data $(s^1,s^2,\overline{\Lambda})$
in $\CATD$ as follows, with $\overline{\Lambda}$ invertible:

\[
\begin{tikzpicture}[xscale=1.5,yscale=-0.8]
    \node (A0_0) at (0, 0) {$D$};
    \node (A0_2) at (2, 0) {$B^1$};
    \node (A2_0) at (0, 2) {$B^2$};
    \node (A2_2) at (2, 2) {$A$.};
    
    \node (A1_1) [rotate=225] at (0.9, 1) {$\Longrightarrow$};
    \node (A1_2) at (1.2, 1) {$\overline{\Lambda}$};
    
    \path (A0_0) edge [->]node [auto,swap] {$\scriptstyle{s^2}$} (A2_0);
    \path (A0_0) edge [->]node [auto] {$\scriptstyle{s^1}$} (A0_2);
    \path (A0_2) edge [->]node [auto] {$\scriptstyle{\overline{g}^1}$} (A2_2);
    \path (A2_0) edge [->]node [auto,swap] {$\scriptstyle{\overline{g}^2}$} (A2_2);
\end{tikzpicture}
\]

Then we define an invertible $2$-morphism

\begin{equation}\label{eq-44}
\Lambda:=\Big(\left(\Omega^2\right)^{-1}\ast i_{s^2}\Big)\odot\overline{\Lambda}\odot
\Big(\Omega^1\ast i_{s^1}\Big):\,\,g^1\circ s^1\Longrightarrow g^2
\circ s^2.
\end{equation}

Since \eqref{eq-60} is a weak fiber product, then by \textbf{\hyperref[A1]{A1}}$(D)$
for \eqref{eq-60} there are a morphism $s:D
\rightarrow C$ and a pair of invertible $2$-morphisms $\Lambda^m:s^m\Rightarrow
r^m\circ s$ for $m=1,2$, such that:

\begin{equation}\label{eq-18}
\Big(\Omega\ast i_s\Big)\odot\Big(i_{g^1}\ast\Lambda^1\Big)=
\Big(i_{g^2}\ast\Lambda^2\Big)\odot\Lambda.
\end{equation}

Therefore, by interchange law we have:

\begin{gather*}
\Big(\overline{\Omega}\ast i_s\Big)\odot\Big(i_{\overline{g}^1}\ast\Lambda^1
 \Big)\stackrel{\eqref{eq-21}}{=} \\
%%%
\stackrel{\eqref{eq-21}}{=}\Big(\Omega^2\ast i_{r^2\circ s}\Big)
 \odot\Big(\Omega\ast i_s\Big)\odot\Big(\left(\Omega^1\right)^{-1}\ast
 i_{r^1\circ s}\Big)\odot\Big(i_{\overline{g}^1}\ast\Lambda^1\Big)= \\ 
%%%
  =\Big(\Omega^2\ast i_{r^2\circ s}\Big)\odot\Big(\Omega\ast
 i_s\Big)\odot\Big(i_{g^1}\ast\Lambda^1\Big)
 \odot\Big(\left(\Omega^1\right)^{-1}\ast i_{s^1}\Big)
 \stackrel{\eqref{eq-18}}{=} \\ 
%%%
  \stackrel{\eqref{eq-18}}{=}\Big(\Omega^2\ast i_{r^2\circ s}\Big)\odot
 \Big(i_{g^2}\ast\Lambda^2\Big)\odot\Lambda\odot\Big(\left(\Omega^1\right)^{-1}\ast
 i_{s^1}\Big)= \\
%%%
  =\Big(i_{\overline{g}^2}\ast\Lambda^2\Big)\odot\Big(\Omega^2\ast
 i_{s^2}\Big)\odot\Lambda\odot\Big(\left(\Omega^1\right)^{-1}\ast
 i_{s^1}\Big)\stackrel{\eqref{eq-44}}{=}\Big(i_{\overline{g}^2}\ast
 \Lambda^2\Big)\odot\overline{\Lambda}.  
\end{gather*}

Therefore, property \textbf{\hyperref[A1]{A1}}$(D)$ holds for diagram \eqref{eq-21}. Let
us prove also condition \textbf{\hyperref[A2]{A2}}$(D)$ for \eqref{eq-21}, so let us fix any
pair of morphisms $t,t':D\rightarrow C$ and any
pair of invertible $2$-morphisms $\Gamma^m:r^m\circ t
\Rightarrow r^m\circ t'$ for $m=1,2$, such that:

\begin{equation}\label{eq-22}
\Big(\overline{\Omega}\ast i_{t'}\Big)\odot\Big(i_{\overline{g}^1}\ast
\Gamma^1\Big)=\Big(i_{\overline{g}^2}\ast\Gamma^2\Big)\odot\Big(
\overline{\Omega}\ast i_t\Big).
\end{equation}

By interchange law we have:

\begin{gather*}
\Big(\Omega\ast i_{t'}\Big)\odot\Big(i_{g^1}\ast\Gamma^1\Big)= \\
%%%
=\Big(\Omega\ast i_{t'}\Big)\odot\Big(\left(\Omega^1\right)^{-1}\ast i_{r^1\circ
 t'}\Big)\odot\Big(i_{\overline{g}^1}\ast\Gamma^1\Big)\odot
 \Big(\Omega^1\ast i_{r^1\circ t}\Big)\stackrel{\eqref{eq-21}}{=} \\
%%%
\stackrel{\eqref{eq-21}}{=}\Big(\left(\Omega^2\right)^{-1}\ast i_{r^2\circ t'}
 \Big)\odot\Big(\overline{\Omega}
 \ast i_{t'}\Big)\odot\Big(i_{\overline{g}^1}\ast\Gamma^1\Big)
 \odot\Big(\Omega^1\ast i_{r^1\circ t}\Big)\stackrel{\eqref{eq-22}}{=} \\
%%%
\stackrel{\eqref{eq-22}}{=}\Big(\left(\Omega^2\right)^{-1}\ast i_{r^2\circ t'}\Big)
 \odot\Big(i_{\overline{g}^2}\ast\Gamma^2\Big)\odot\Big(\overline{\Omega}
 \ast i_t\Big)\odot\Big(\Omega^1\ast i_{r^1
 \circ t}\Big)= \\
%%%
=\Big(i_{g^2}\ast\Gamma^2\Big)\odot\Big(\left(\Omega^2\right)^{-1}\ast
 i_{r^2\circ t}\Big)\odot\Big(\overline{\Omega}\ast i_t\Big)\odot\Big(\Omega^1\ast
 i_{r^1\circ t}\Big)\stackrel{\eqref{eq-21}}{=}\\
%%%
\stackrel{\eqref{eq-21}}{=}\Big(i_{g^2}\ast\Gamma^2\Big)\odot\Big(\Omega
 \ast i_t\Big).  
\end{gather*}

Since \eqref{eq-60} is a weak fiber product, then by \textbf{\hyperref[A2]{A2}}$(D)$ for
\eqref{eq-60} there is a unique invertible
$2$-morphism $\Gamma:t\Rightarrow t'$, such that
$i_{r^m}\ast\Gamma=\Gamma^m$ for each $m=1,2$.
Therefore, property \textbf{\hyperref[A2]{A2}}$(D)$ holds also for diagram \eqref{eq-21}.
\end{proof}

\begin{proof}[Proof of Lemma~\ref{lem-04}.]
We give a proof in the case when $\CATD$ is a $2$-category; this implies that $\overline{\Omega}=
\Omega\ast i_e$. Since $e$ is an internal equivalence, we can choose a morphism
$d:C\rightarrow\overline{C}$ and invertible
$2$-morphisms $\Delta:\id_{\overline{C}}\Rightarrow d\circ e$ and $\Xi:e\circ d\Rightarrow
\id_C$, such that

\begin{equation}\label{eq-25}
\Big(\Xi\ast i_e\Big)\odot\Big(i_e\ast\Delta\Big)=i_e\quad\textrm{and}\quad\Big(i_d\ast\Xi\Big)\odot
\Big(\Delta\ast i_d\Big)=i_d.
\end{equation}

We fix any object $\overline{D}$ in $\CATD$ and we prove conditions
\textbf{\hyperref[A1]{A1}}$(\overline{D})$ 
and \textbf{\hyperref[A2]{A2}}$(\overline{D})$
for diagram \eqref{eq-24}. In order to prove the first property, let us
fix any set of data
$(\overline{s}^1,\overline{s}^2,\overline{\Lambda})$ in $\CATD$ as in the following
diagram, with $\overline{\Lambda}$ invertible:

\[
\begin{tikzpicture}[xscale=1.5,yscale=-0.8]
    \node (A0_0) at (0, 0) {$\overline{D}$};
    \node (A0_2) at (2, 0) {$B^1$};
    \node (A2_0) at (0, 2) {$B^2$};
    \node (A2_2) at (2, 2) {$A$.};
    
    \node (A1_1) [rotate=225] at (0.9, 1) {$\Longrightarrow$};
    \node (A1_2) at (1.2, 1) {$\overline{\Lambda}$};
    
    \path (A0_0) edge [->]node [auto,swap] {$\scriptstyle{\overline{s}^2}$} (A2_0);
    \path (A0_0) edge [->]node [auto] {$\scriptstyle{\overline{s}^1}$} (A0_2);
    \path (A0_2) edge [->]node [auto] {$\scriptstyle{g^1}$} (A2_2);
    \path (A2_0) edge [->]node [auto,swap] {$\scriptstyle{g^2}$} (A2_2);
\end{tikzpicture}
\]

Since \eqref{eq-60} is a weak fiber product, then by \textbf{\hyperref[A1]{A1}}$(\overline{D})$
for \eqref{eq-60} there are a morphism $s:
\overline{D}\rightarrow C$ and a pair of 
invertible $2$-morphisms $\Lambda^m:\overline{s}^m\Rightarrow r^m\circ s$ for $m=1,2$, such that:

\begin{equation}\label{eq-27}
\Big(\Omega\ast i_s\Big)\odot\Big(i_{g^1}\ast\Lambda^1\Big)=\Big(i_{g^2}\ast\Lambda^2\Big)\odot\
\overline{\Lambda}.
\end{equation}

Then we set $\overline{s}:=d\circ s:\overline{D}\rightarrow\overline{C}$; for each $m=1,2$ we define

\begin{equation}\label{eq-45}
\overline{\Lambda}^m:=\Big(i_{r^m}\ast\Xi^{-1}\ast i_s\Big)\odot\Lambda^m:\,\,\overline{s}^m
\Longrightarrow r^m\circ e\circ d\circ s=r^m\circ e\circ\overline{s}.
\end{equation}

By definition of $\overline{s}$ and $\overline{\Omega}$ and interchange law, we have:

\begin{gather*}
\Big(\overline{\Omega}\ast i_{\overline{s}}\Big)\odot\Big(i_{g^1}\ast\overline{\Lambda}^1
 \Big)\stackrel{\eqref{eq-45}}{=} \\
%%%
\stackrel{\eqref{eq-45}}{=}\Big(\Omega\ast i_{e\circ d\circ s}\Big)\odot\Big(i_{g^1\circ r^1}
 \ast\Xi^{-1}\ast i_s\Big)\odot\Big(i_{g^1}\ast\Lambda^1\Big)= \\
%%%
=\Big(i_{g^2\circ r^2}\ast\Xi^{-1}\ast i_s\Big)\odot\Big(\Omega\ast i_s\Big)\odot\Big(i_{g^1}\ast
 \Lambda^1\Big)\stackrel{\eqref{eq-27}}{=} \\
%%%
  \stackrel{\eqref{eq-27}}{=}\Big(i_{g^2\circ r^2}\ast\Xi^{-1}\ast i_s\Big)\odot\Big(i_{g^2}\ast
 \Lambda^2\Big)\odot\overline{\Lambda}\stackrel{\eqref{eq-45}}{=}
 \Big(i_{g^2}\ast\overline{\Lambda}^2\Big)\odot\overline{\Lambda}.  
\end{gather*}

Therefore diagram \eqref{eq-24} satisfies property \textbf{\hyperref[A1]{A1}}$(\overline{D})$.\\

Let us prove also property \textbf{\hyperref[A2]{A2}}$(\overline{D})$
for \eqref{eq-24}, so let us fix any pair of
morphisms $\overline{t},\overline{t}':\overline{D}\rightarrow\overline{C}$ and any pair of invertible 
$2$-morphisms $\overline{\Gamma}^m:(r^m\circ e)\circ\overline{t}\Rightarrow(r^m\circ e)\circ
\overline{t}'$ for $m=1,2$, such that

\begin{equation}\label{eq-28}
\Big(\overline{\Omega}\ast i_{\overline{t}'}\Big)\odot\Big(i_{g^1}\ast\overline{\Gamma}^1\Big)
=\Big(i_{g^2}\ast\overline{\Gamma}^2\Big)\odot\Big(\overline{\Omega}\ast i_{\overline{t}}\Big).
\end{equation}

Let us consider the morphisms $t:=e\circ\overline{t}$ and $t':=e\circ\overline{t}'$, both defined
from $\overline{D}$ to $C$. Then by \eqref{eq-28} we have $(\Omega\ast i_{t'})\odot(i_{g^1}\ast
\overline{\Gamma}^1)=(i_{g^2}\ast\overline{\Gamma}^2)\odot(\Omega\ast i_t)$. Since \eqref{eq-60} is a
weak fiber product, then by \textbf{\hyperref[A2]{A2}}$(\overline{D}$)
there is a unique invertible $2$-morphism $\Gamma:t
\Rightarrow t'$, such that $i_{r^m}\ast\Gamma=\overline{\Gamma}^m$ for each $m=1,2$. Then we define

\[\overline{\Gamma}:=\Big(\Delta^{-1}\ast i_{\overline{t}'}\Big)\odot\Big(i_d\ast\Gamma\Big)
\odot\Big(\Delta\ast i_{\overline{t}}\Big):\,\,\overline{t}\Longrightarrow\overline{t}'.\]

A direct computation using \eqref{eq-25} and the interchange law shows that $i_e\ast\overline{\Gamma}
=\Gamma$. Therefore

\begin{equation}\label{eq-31}
i_{r^m\circ e}\ast\overline{\Gamma}=i_{r^m}\ast\Gamma=\overline{\Gamma}^m\quad\textrm{for }m=1,2.
\end{equation}

So in order to conclude we need only to prove that $\overline{\Gamma}$ is the unique invertible
$2$-morphism $\overline{t}\Rightarrow\overline{t}'$ such that \eqref{eq-31} holds. So let us fix
another invertible $2$-morphism $\overline{\Gamma}':\overline{t}\Rightarrow\overline{t}'$, such that
$i_{r^m\circ e}\ast\overline{\Gamma}'=\overline{\Gamma}^m$ for each $m=1,2$. Then we have $i_{r^m}
\ast(i_e\ast\overline{\Gamma}')=\overline{\Gamma}^m$ for each $m=1,2$; by uniqueness of $\Gamma$ we
conclude that $i_e\ast\overline{\Gamma}'=\Gamma$. So we get $i_e\ast\overline{\Gamma}=\Gamma=i_e\ast
\overline{\Gamma}'$, hence by interchange law we have:

\begin{gather*}
\overline{\Gamma}=\Big(\Delta^{-1}\ast i_{\overline{t}'}\Big)\odot\Big(i_{d\circ e}\ast
 \overline{\Gamma}\Big)\odot\Big(\Delta\ast i_{\overline{t}}\Big)= \\
%%%
=\Big(\Delta^{-1}\ast i_{\overline{t}'}\Big)\odot\Big(i_{d\circ e}\ast\overline{\Gamma}'\Big)
 \odot\Big(\Delta\ast i_{\overline{t}}\Big)=\overline{\Gamma}'.  
\end{gather*}

This suffices to conclude.
\end{proof}

\begin{proof}[Proof of Lemma~\ref{lem-14}.]
For simplicity of exposition, in this proof we assume that both $\CATA$ and $\CATB$ are $2$-categories
and that $\functor{F}$ is a strict pseudofunctor between them, i.e.\ a $2$-functor (in other
terms, we assume that $\functor{F}$ preserves compositions and identities). So in particular
$\Omega_{\CATB}=\functor{F}_2(\Omega_{\CATA})$. In the more general case
the proof is analogous: it suffices to add unitors and associators for $\functor{F}$ wherever it
is necessary and to use the coherence conditions on the pseudofunctor $\functor{F}$.\\

So let us fix any object $D_{\CATB}$ in $\CATB$ and let us prove conditions
\textbf{\hyperref[A1]{A1}}$(D_{\CATB})$ and \textbf{\hyperref[A2]{A2}}$(D_{\CATB})$
for diagram \eqref{eq-94}.
By property (\hyperref[X1]{X1}) for $\functor{F}$, there are an object $D_{\CATA}$ and an internal
equivalence $e_{\CATB}:\functor{F}_0(D_{\CATA})\rightarrow D_{\CATB}$. Since $e_{\CATB}$ is an
internal equivalence then there are an internal equivalence $d_{\CATB}:D_{\CATB}\rightarrow
\functor{F}_0(D_{\CATA})$ and a pair of invertible $2$-morphisms $\Xi_{\CATB}:e_{\CATB}
\circ d_{\CATB}\Rightarrow\id_{D_{\CATB}}$ and $\Delta_{\CATB}:
\id_{\functor{F}_0(D_{\CATA})}\Rightarrow d_{\CATB}\circ e_{\CATB}$, such that

\begin{equation}\label{eq-67}
\Big(\Xi_{\CATB}\ast i_{e_{\CATB}}\Big)\odot\Big(i_{e_{\CATB}}\ast\Delta_{\CATB}\Big)=i_{e_{\CATB}}
\quad\textrm{and}\quad\Big(i_{d_{\CATB}}\ast\Xi_{\CATB}\Big)\odot\Big(\Delta_{\CATB}\ast
i_{d_{\CATB}}\Big)=i_{d_{\CATB}}.
\end{equation}

In order to prove condition \textbf{\hyperref[A1]{A1}}$(D_{\CATB})$ for \eqref{eq-94},
let us fix any set $(s^1_{\CATB},s^2_{\CATB},\Lambda_{\CATB})$ in $\CATB$ as
follows, with $\Lambda_{\CATB}$ invertible

\[
\begin{tikzpicture}[xscale=1.5,yscale=-0.8]
    \node (A0_0) at (0, 0) {$D_{\CATB}$};
    \node (A0_2) at (2, 0) {$\functor{F}_0(B^1_{\CATA})$};
    \node (A2_0) at (0, 2) {$\functor{F}_0(B^2_{\CATA})$};
    \node (A2_2) at (2, 2) {$\functor{F}_0(A_{\CATA})$.};

    \node (A0_1) [rotate=225] at (0.9, 1) {$\Longrightarrow$};
    \node (A1_1) at (1.3, 1) {$\Lambda_{\CATB}$};

    \path (A2_0) edge [->]node [auto,swap] {$\scriptstyle{\functor{F}_1(g^2_{\CATA})}$} (A2_2);
    \path (A0_0) edge [->]node [auto,swap] {$\scriptstyle{s^2_{\CATB}}$} (A2_0);
    \path (A0_2) edge [->]node [auto] {$\scriptstyle{\functor{F}_1(g^1_{\CATA})}$} (A2_2);
    \path (A0_0) edge [->]node [auto] {$\scriptstyle{s^1_{\CATB}}$} (A0_2);
\end{tikzpicture}
\]

By property (\hyperref[X2]{X2}) for $\functor{F}$, for each $m=1,2$ there are a morphism
$s^m_{\CATA}:D_{\CATA}\rightarrow B^m_{\CATA}$ and an invertible $2$-morphism $\chi^m_{\CATB}:
\functor{F}_1(s^m_{\CATA})\Rightarrow s^m_{\CATB}\circ e_{\CATB}$. Again by (\hyperref[X2]{X2})
there is a (unique) invertible $2$-morphism $\Lambda_{\CATA}:g^1_{\CATA}\circ s^1_{\CATA}\Rightarrow
g^2_{\CATA}\circ s^2_{\CATA}$, such that

\begin{gather}
\nonumber \functor{F}_2(\Lambda_{\CATA})=\Big(i_{\functor{F}_1(g^2_{\CATA})}\ast\left(
 \chi^2_{\CATB}\right)^{-1}\Big)
 \odot\Big(\Lambda_{\CATB}\ast i_{e_{\CATB}}\Big)\odot\Big(
 i_{\functor{F}_1(g^1_{\CATA})}\ast\chi^1_{\CATB}\Big): \\
\label{eq-95} \functor{F}_1(g^1_{\CATA}\circ
 s^1_{\CATA})\Longrightarrow\functor{F}_1(g^2_{\CATA}\circ s^2_{\CATA}).
\end{gather}

Since \eqref{eq-93} satisfies property \textbf{\hyperref[A1]{A1}}$(D_{\CATA})$,
then there are a morphism $s_{\CATA}:
D_{\CATA}\rightarrow C_{\CATA}$ and a pair of invertible $2$-morphisms $\Lambda^m_{\CATA}:s^m_{\CATA}
\Rightarrow r^m_{\CATA}\circ s_{\CATA}$ for $m=1,2$, such that

\begin{equation}\label{eq-101}
\Big(\Omega_{\CATA}\ast i_{s_{\CATA}}\Big)\odot\Big(i_{g^1_{\CATA}}\ast\Lambda^1_{\CATA}\Big)=
\Big(i_{g^2_{\CATA}}\ast\Lambda^2_{\CATA}\Big)\odot\Lambda_{\CATA}.
\end{equation}

Since we are assuming that $\functor{F}$ is a strict pseudofunctor, then
by applying $\functor{F}_2$ to \eqref{eq-101} and using \eqref{eq-95}, we get:

\begin{gather*}
\Big(\functor{F}_2(\Omega_{\CATA})\ast i_{\functor{F}_1(s_{\CATA})}\Big)\odot
 \Big(i_{\functor{F}_1(g^1_{\CATA})}\ast\functor{F}_2(\Lambda^1_{\CATA})\Big)= \\
%%%
=\Big(i_{\functor{F}_1(g^2_{\CATA})}\ast\functor{F}_2(\Lambda^2_{\CATA})\Big)\odot\Big(
 i_{\functor{F}_1(g^2_{\CATA})}\ast\left(\chi^2_{\CATB}\right)^{-1}\Big)\odot
 \Big(\Lambda_{\CATB}\ast
 i_{e_{\CATB}}\Big)\odot\Big(
 i_{\functor{F}_1(g^1_{\CATA})}\ast\chi^1_{\CATB}\Big).  
\end{gather*}

This implies that:

\begin{gather}
\nonumber \Big(\functor{F}_2(\Omega_{\CATA})\ast i_{\functor{F}_1(s_{\CATA})}\Big)\odot 
 \Big(i_{\functor{F}_1(g^1_{\CATA})}\ast\Big(\functor{F}_2(\Lambda^1_{\CATA})\odot
 \left(\chi^1_{\CATB}\right)^{-1}\Big)\Big)= \\
%%%
\label{eq-96} =\Big(i_{\functor{F}_1(g^2_{\CATA})}\ast\Big(\functor{F}_2(\Lambda^2_{\CATA})\odot
 \left(\chi^2_{\CATB}\right)^{-1}\Big)\Big)\odot\Big(\Lambda_{\CATB}\ast i_{e_{\CATB}}\Big).
\end{gather}

Then by interchange law we have:

\begin{gather}
\nonumber \Big(\functor{F}_2(\Omega_{\CATA})\ast i_{\functor{F}_1
 (s_{\CATA})\circ d_{\CATB}}\Big)\odot \\
\nonumber \odot\Big\{i_{\functor{F}_1(g^1_{\CATA})}\ast\Big[
 \Big(\functor{F}_2(\Lambda^1_{\CATA})\ast
 i_{d_{\CATB}}\Big)\odot\Big(\left(\chi^1_{\CATB}\right)^{-1}\ast i_{d_{\CATB}}\Big)\odot
 \Big(i_{\operatorname{s}^1_{\CATB}}\ast
 \Xi_{\CATB}^{-1}\Big)\Big]\Big\}\stackrel{\eqref{eq-96}}{=} \\
%%%
\label{eq-97} \stackrel{\eqref{eq-96}}{=}\Big\{i_{\functor{F}_1(g^2_{\CATA})}\ast
 \Big[\Big(\functor{F}_2(\Lambda^2_{\CATA})\ast i_{d_{\CATB}}\Big)
 \odot\Big(\left(\chi^2_{\CATB}\right)^{-1}\ast i_{d_{\CATB}}\Big)\odot
 \Big(i_{\operatorname{s}^2_{\CATB}}\ast\Xi_{\CATB}^{-1}\Big)\Big]\Big\}\odot\Lambda_{\CATB}.  
\end{gather}

Then we set $s_{\CATB}:=\functor{F}_1(s_{\CATA})\circ d_{\CATB}:D_{\CATB}\rightarrow
\functor{F}_0(C_{\CATA})$ and for each $m=1,2$:

\[\Lambda^m_{\CATB}:=\Big(\functor{F}_2(\Lambda^m_{\CATA})\ast i_{d_{\CATB}}
\Big)\odot\Big(\left(\chi^m_{\CATB}\right)^{-1}\ast i_{d_{\CATB}}\Big)\odot
\Big(i_{\operatorname{s}^m_{\CATB}}
\ast\Xi_{\CATB}^{-1}\Big):\,\,s^m_{\CATB}\Longrightarrow\functor{F}_1(r^m_{\CATA})\circ s_{\CATB},\]
so \eqref{eq-97} reads as follows:

\[\Big(\functor{F}_2(\Omega_{\CATA})\ast i_{s_{\CATB}}\Big)\odot
\Big(i_{\functor{F}_1(g^1_{\CATA})}\ast\Lambda^1_{\CATB}
\Big)=\Big(i_{\functor{F}_1(g^2_{\CATA})}\ast\Lambda^2_{\CATB}\Big)\odot\Lambda_{\CATB}.\]

This shows that condition \textbf{\hyperref[A1]{A1}}$(D_{\CATB})$
holds for diagram \eqref{eq-94}.\\

Now let us also prove \textbf{\hyperref[A2]{A2}}$(D_{\CATB})$ for \eqref{eq-94}, so let us fix
any pair of morphisms $t_{\CATB},t'_{\CATB}:D_{\CATB}\rightarrow\functor{F}_0
(C_{\CATA})$ any pair of invertible $2$-morphisms $\Gamma^m_{\CATB}:\functor{F}_1(r^m_{\CATA})\circ
t_{\CATB}\Rightarrow\functor{F}_1(r^m_{\CATA})\circ t'_{\CATB}$, such that

\begin{equation}\label{eq-98}
\Big(\functor{F}_2(\Omega_{\CATA})\ast i_{t'_{\CATB}}\Big)\odot\Big(i_{\functor{F}_1
(g^1_{\CATA})}\ast\Gamma^1_{\CATB}\Big)=\Big(i_{\functor{F}_1(g^2_{\CATA})}\ast
\Gamma^2_{\CATB}\Big)\odot\Big(\functor{F}_2(\Omega_{\CATA})\ast i_{t_{\CATB}}\Big).
\end{equation}

By property (\hyperref[X2]{X2}) there are a pair of morphisms $t_{\CATA},t'_{\CATA}:D_{\CATA}
\rightarrow C_{\CATA}$ and a pair of invertible $2$-morphisms

\[\Phi_{\CATB}:\functor{F}_1(t_{\CATA})\Longrightarrow t_{\CATB}\circ e_{\CATB}\quad\textrm{and}\quad
\Phi'_{\CATB}:\functor{F}_1(t'_{\CATA})\Longrightarrow t'_{\CATB}\circ e_{\CATB}.\]

Then using the interchange law we get:

\begin{gather}
\nonumber \Big(\functor{F}_2(\Omega_{\CATA})\ast i_{\functor{F}_1(t'_{\CATA})}
 \Big)\odot \\
\nonumber \odot\Big\{i_{\functor{F}_1(g^1_{\CATA})}\ast\Big[\Big(i_{\functor{F}_1(r^1_{\CATA})}\ast
 \left(\Phi'_{\CATB}\right)^{-1}\Big)\odot
 \Big(\Gamma^1_{\CATB}\ast i_{e_{\CATB}}\Big)\odot
 \Big(i_{\functor{F}_1(r^1_{\CATA})}\ast\Phi_{\CATB}\Big)\Big]\Big\}= \\
%%%
\nonumber =\Big(i_{\functor{F}_1(g^2_{\CATA}\circ r^2_{\CATA})}\ast\left(\Phi'_{\CATB}
 \right)^{-1}\Big)\odot\Big\{\Big[\Big(\functor{F}_2(\Omega_{\CATA})\ast i_{t'_{\CATB}}\Big)\odot \\
\nonumber \odot\Big(i_{\functor{F}_1(g^1_{\CATA})}\ast\Gamma^1_{\CATB}\Big)\Big]\ast i_{e_{\CATB}}
 \Big\}\odot\Big(i_{\functor{F}_1(g^1_{\CATA}\circ r^1_{\CATA})}\ast\Phi_{\CATB}\Big)
 \stackrel{\eqref{eq-98}}{=} \\
%%%
\nonumber \stackrel{\eqref{eq-98}}{=}\Big(i_{\functor{F}_1(g^2_{\CATA}\circ r^2_{\CATA})}\ast\left(
 \Phi'_{\CATB}\right)^{-1}\Big)\odot\Big\{\Big[\Big(i_{\functor{F}_1(g^2_{\CATA})}\ast
  \Gamma^2_{\CATB}\Big)\odot \\
\nonumber \odot\Big(\functor{F}_2(\Omega_{\CATA})\ast i_{t_{\CATB}}\Big)\Big]
 \ast i_{e_{\CATB}}\Big\}\odot\Big(i_{\functor{F}_1(g^1_{\CATA}\circ r^1_{\CATA})}\ast\Phi_{\CATB}
 \Big)= \\
%%%
\nonumber =\Big\{i_{\functor{F}_1(g^2_{\CATA})}\ast\Big[\Big(i_{\functor{F}_1(r^2_{\CATA})}\ast
 \left(\Phi'_{\CATB}\right)^{-1}\Big)\odot\Big(\Gamma^2_{\CATB}\ast i_{e_{\CATB}}\Big)\odot
 \Big(i_{\functor{F}_1(r^2_{\CATA})}\ast\Phi_{\CATB}\Big)\Big]\Big\}\odot \\
\label{eq-99} \odot\Big(\functor{F}_2(\Omega_{\CATA})\ast i_{\functor{F}_1(t_{\CATA})}\Big). 
\end{gather}

Again by property (\hyperref[X2]{X2}) for $\functor{F}$, for each $m=1,2$ there is a (unique)
invertible $2$-morphism $\Gamma^m_{\CATA}:r^m_{\CATA}\circ t_{\CATA}\Rightarrow r^m_{\CATA}\circ
t'_{\CATA}$, such that

\begin{gather}
\nonumber \functor{F}_2(\Gamma^m_{\CATA})=\Big(i_{\functor{F}_1(r^m_{\CATA})}\ast\left(\Phi'_{\CATB}
 \right)^{-1}\Big)\odot\Big(\Gamma^m_{\CATB}\ast i_{e_{\CATB}}\Big)\odot\Big(i_{\functor{F}_1
 (r^m_{\CATA})}\ast\Phi_{\CATB}\Big): \\
%%%
\label{eq-119}\functor{F}_1(r^m_{\CATA}\circ t_{\CATA})\Longrightarrow\functor{F}_1(r^m_{\CATA}
 \circ t'_{\CATA}).
\end{gather}

Then \eqref{eq-99} reads as follows:

\[\functor{F}_2\Big[\Big(\Omega_{\CATA}\ast i_{t'_{\CATA}}\Big)\odot\Big(i_{g^1_{\CATA}}\ast
\Gamma^1_{\CATA}\Big)\Big]=\functor{F}_2\Big[\Big(i_{g^2_{\CATA}}\ast\Gamma^2_{\CATA}\Big)\odot
\Big(\Omega_{\CATA}\ast i_{t_{\CATA}}\Big)\Big].\]

Then again by property (\hyperref[X2]{X2}) we conclude that

\[\Big(\Omega_{\CATA}\ast i_{t'_{\CATA}}\Big)\odot\Big(i_{g^1_{\CATA}}\ast\Gamma^1_{\CATA}\Big)=
\Big(i_{g^2_{\CATA}}\ast\Gamma^2_{\CATA}\Big)\odot\Big(\Omega_{\CATA}\ast i_{t_{\CATA}}\Big).\]

Since diagram \eqref{eq-93} satisfies property \textbf{\hyperref[A2]{A2}}$(D_{\CATA})$,
then the previous identity
implies that there is a unique invertible $2$-morphism $\Gamma_{\CATA}:t_{\CATA}\Rightarrow
t'_{\CATA}$, such that $i_{r^m_{\CATA}}\ast\Gamma_{\CATA}=\Gamma^m_{\CATA}$ for each $m=1,2$.
So by interchange law, for each $m=1,2$ we have:

\begin{gather}
\nonumber \Gamma^m_{\CATB}=\Big(i_{\functor{F}_1(r^m_{\CATA})\circ t'_{\CATB}}\ast\Xi_{\CATB}\Big)
 \odot\Big(\Gamma^m_{\CATB}\ast i_{e_{\CATB}\circ d_{\CATB}}\Big)\odot
 \Big(i_{\functor{F}_1(r^m_{\CATA})\circ t_{\CATB}}\ast\Xi_{\CATB}^{-1}\Big)
 \stackrel{\eqref{eq-119}}{=} \\
%%%
\nonumber \stackrel{\eqref{eq-119}}{=}\Big(i_{\functor{F}_1(r^m_{\CATA})\circ t'_{\CATB}}\ast
 \Xi_{\CATB}\Big)\odot\Big(i_{\functor{F}_1(r^m_{\CATA})}\ast\Phi'_{\CATB}\ast i_{d_{\CATB}}\Big)
 \odot\Big(\functor{F}_2(\Gamma^m_{\CATA})\ast i_{d_{\CATB}}\Big)\odot \\
\nonumber \odot\Big(i_{\functor{F}_1(r^m_{\CATA})}\ast\Phi_{\CATB}^{-1}\ast i_{d_{\CATB}}\Big)\odot
 \Big(i_{\functor{F}_1(r^m_{\CATA})\circ t_{\CATB}}\ast\Xi^{-1}_{\CATB}\Big)= \\
%%%
\label{eq-54} =i_{\functor{F}_1(r^m_{\CATA})}\ast\Big\{\Big(i_{t'_{\CATB}}\ast\Xi_{\CATB}\Big)
 \odot\Big[\Big(\Phi'_{\CATB}\odot\functor{F}_2(\Gamma_{\CATA})\odot
 \Phi^{-1}_{\CATB}\Big)\ast i_{d_{\CATB}}\Big]\odot\Big(i_{t_{\CATB}}\ast\Xi_{\CATB}^{-1}\Big)
 \Big\}.
\end{gather}

Hence, if we set

\begin{equation}\label{eq-55}
\Gamma_{\CATB}:=\Big(i_{t'_{\CATB}}\ast\Xi_{\CATB}\Big)\odot
\Big[\Big(\Phi'_{\CATB}\odot\functor{F}_2(\Gamma_{\CATA})\odot
\Phi^{-1}_{\CATB}\Big)\ast i_{d_{\CATB}}\Big]\odot\Big(i_{t_{\CATB}}\ast
\Xi_{\CATB}^{-1}\Big):\,\,t_{\CATB}\Rightarrow t'_{\CATB},  
\end{equation}
then \eqref{eq-54} implies that $i_{\functor{F}_1(r^m_{\CATA})}\ast\Gamma_{\CATB}=\Gamma^m_{\CATB}$
for each $m=1,2$.\\

Now in order to conclude that \eqref{eq-94} satisfies property
\textbf{\hyperref[A2]{A2}}$(D_{\CATB})$, we need
only to prove that $\Gamma_{\CATB}$ is the unique invertible $2$-morphism with such a property. So
let us fix another invertible $2$-morphism $\Gamma'_{\CATB}:t_{\CATB}\Rightarrow t'_{\CATB}$ such
that $i_{\functor{F}_1(r^m_{\CATA})}\ast\Gamma'_{\CATB}=\Gamma^m_{\CATB}$ for each $m=1,2$. By
(\hyperref[X2]{X2}) there is a unique invertible $2$-morphism $\Gamma'_{\CATA}:t_{\CATA}
\Rightarrow t'_{\CATA}$, such that

\begin{equation}\label{eq-56}
\functor{F}_2(\Gamma'_{\CATA})=\left(\Phi'_{\CATB}\right)^{-1}\odot\Big(\Gamma'_{\CATB}\ast
 i_{e_{\CATB}}\Big)\odot\Phi_{\CATB}:\,\,\functor{F}_1(t_{\CATA})\Longrightarrow\functor{F}_1
(t'_{\CATA}).
\end{equation}

Now by interchange law, we have:

\begin{gather}
\nonumber \Gamma_{\CATB}\ast i_{e_{\CATB}}\stackrel{\eqref{eq-55},\eqref{eq-67}}{=}
 \Big(i_{t'_{\CATB}\circ e_{\CATB}}\ast\Delta^{-1}_{\CATB}\Big)\odot
 \Big(\Big(\Phi'_{\CATB}\odot\functor{F}_2(\Gamma_{\CATA})\odot\Phi^{-1}_{\CATB}\Big)\ast
 i_{d_{\CATB}\circ e_{\CATB}}\Big)\odot \\
\label{eq-47} \odot\Big(i_{t_{\CATB}\circ e_{\CATB}}\ast\Delta_{\CATB}\Big)=
 \Phi'_{\CATB}\odot\functor{F}_2(\Gamma_{\CATA})\odot\Phi^{-1}_{\CATB}.  
\end{gather}

Therefore, for each $m=1,2$ we have:

\begin{gather*}
\functor{F}_2\Big(i_{r^m_{\CATA}}\ast\Gamma'_{\CATA}\Big)=i_{\functor{F}_1(r^m_{\CATA})}\ast
 \functor{F}_2(\Gamma'_{\CATA})\stackrel{\eqref{eq-56}}{=} \\
%%%
\stackrel{\eqref{eq-56}}{=}\Big(i_{\functor{F}_1(r^m_{\CATA})}\ast\left(\Phi'_{\CATB}\right)^{-1}
 \Big)\odot\Big(i_{\functor{F}_1(r^m_{\CATA})}\ast\Gamma'_{\CATB}\ast i_{e_{\CATB}}\Big)
 \odot\Big(i_{\functor{F}_1(r^m_{\CATA})}\ast\Phi_{\CATB}\Big)= \\
%%%
=\Big(i_{\functor{F}_1(r^m_{\CATA})}\ast\left(\Phi'_{\CATB}\right)^{-1}\Big)\odot
 \Big(\Gamma^m_{\CATB}\ast i_{e_{\CATB}}\Big)\odot\Big(i_{\functor{F}_1
 (r^m_{\CATA})}\ast\Phi_{\CATB}\Big)= \\
%%%
=\Big(i_{\functor{F}_1(r^m_{\CATA})}\ast\left(\Phi'_{\CATB}\right)^{-1}\Big)\odot
 \Big(i_{\functor{F}_1(r^m_{\CATA})}\ast\Gamma_{\CATB}\ast i_{e_{\CATB}}\Big)\odot
 \Big(i_{\functor{F}_1(r^m_{\CATA})}\ast\Phi_{\CATB}\Big)= \\
%%%
=i_{\functor{F}_1(r^m_{\CATA})}\ast\Big(\left(\Phi'_{\CATB}\right)^{-1}\odot\Big(\Gamma_{\CATB}\ast
 i_{e_{\CATB}}\Big)\odot\Phi_{\CATB}\Big)\stackrel{\eqref{eq-47}}{=} \\
%%%
\stackrel{\eqref{eq-47}}{=}i_{\functor{F}_1(r^m_{\CATA})}\ast\functor{F}_2(\Gamma_{\CATA})=
 \functor{F}_2\Big(i_{r^m_{\CATA}}\ast\Gamma_{\CATA}\Big)=\functor{F}_2(\Gamma^m_{\CATA}).  
\end{gather*}

By (\hyperref[X2]{X2}), this implies that $i_{r^m_{\CATA}}\ast\Gamma'_{\CATA}=\Gamma^m_{\CATA}$
for each $m=1,2$. By construction, $\Gamma_{\CATA}$ is the unique invertible $2$-morphism from
$t_{\CATA}$ to $t'_{\CATA}$ such that $i_{r^m_{\CATA}}\ast\Gamma_{\CATA}=\Gamma^m_{\CATA}$ for
each $m=1,2$, hence $\Gamma'_{\CATA}=\Gamma_{\CATA}$. Therefore,

\begin{gather}
\nonumber \Gamma'_{\CATB}\ast i_{e_{\CATB}}\stackrel{\eqref{eq-56}}{=}\Phi'_{\CATB}\odot
 \functor{F}_2(\Gamma'_{\CATA})\odot\Phi^{-1}_{\CATB}= \\
%%%
\label{eq-154} =\Phi'_{\CATB}\odot\functor{F}_2(\Gamma_{\CATA})\odot\Phi^{-1}_{\CATB}
 \stackrel{\eqref{eq-47}}{=}\Gamma_{\CATB}\ast i_{e_{\CATB}}.  
\end{gather}

Hence by interchange law we have:

\begin{gather*}
\Gamma'_{\CATB}=\Big(i_{t'_{\CATB}}\ast\Xi_{\CATB}\Big)\odot\Big(\Gamma'_{\CATB}\ast i_{e_{\CATB}
 \circ d_{\CATB}}\Big)\odot\Big(i_{t_{\CATB}}\ast\Xi_{\CATB}^{-1}
 \Big)\stackrel{\eqref{eq-154}}{=} \\
%%%
\stackrel{\eqref{eq-154}}{=}\Big(i_{t'_{\CATB}}\ast\Xi_{\CATB}\Big)\odot\Big(\Gamma_{\CATB}\ast
 i_{e_{\CATB}\circ d_{\CATB}}\Big)\odot\Big(i_{t_{\CATB}}\ast\Xi_{\CATB}^{-1}\Big)=\Gamma_{\CATB},
\end{gather*}
so we have proved also the uniqueness part of condition \textbf{\hyperref[A2]{A2}}$(D_{\CATB})$.
Therefore diagram \eqref{eq-94} is a weak fiber product in $\CATB$.
\end{proof}

%%%%% BIBLIOGRAPHY %%%%%%%%%%%%%%%%%%%%%%%%%%%%%%%%%%%%%%%%%%%%%%%%%%%%%%%%%%%%%%%%%%%%%%%%%%%%%%%


\begin{thebibliography}{90}
\bibitem[AMMV]{AMMV} O. Abbad, S. Mantovani, G. Metere, E. M. Vitale, \emph{Butterflies in a
 semi-abelian context}, Advances in Mathematics, 238, 140--183 (2013),
 \href{http://arxiv.org/abs/math/1104.4275v1}{arXiv: math.CT 1104.4275v1};

\bibitem[B]{B} Francis Borceux, \emph{Handbook of Categorical Algebra 1 - Basic Category Theory},
 Cambridge University Press (1994);

\bibitem[GZ]{GZ} Pierre Gabriel, Michel Zisman, \emph{Calculus of Fractions and Homotopy Theory},
 Springer-Verlag, New York, (1967);

\bibitem[J]{J} Dominic Joyce, \emph{Algebraic geometry over $C^{\infty}$-rings} (2012),
 \href{http://arxiv.org/abs/1001.0023v4}{arXiv: math.DG 1001.0023v4};
 
\bibitem[L]{L} Tom Leinster, \emph{Higher Operads, Higher Categories}, London Mathematical
 Society Lecture note, Series 298, Cambridge University Press (2004),
 \href{http://arxiv.org/abs/math/0305049}{arXiv: math.CT 0305049v1};

\bibitem[LMB]{LMB} G\'erard Laumon, Laurent Moret-Bailly, \emph{Champs alg\'ebriques}, Ergebnisse der
 Mathematik und ihrer Grenzgebiete 39, Springer-Verlag (2000);

\bibitem[Mac]{Mac} Saunders Mac Lane, \emph{Categories for the working mathematicians},
 Springer-Verlag (1978);

\bibitem[MM]{MM} Ieke Moerdijk, Janez Mr\u{c}un, \emph{Introduction to foliations and Lie
 groupoids}, Cambridge University Press (2003);

\bibitem[MMV]{MMV} S. Mantovani, G. Metere, E. M. Vitale, \emph{Profunctors in Mal'tsev categories and
 fractions of functors}, Journal of Pure and Applied Algebra, 217, 1173 -- 1186 (2013),
 \href{http://arxiv.org/abs/math/1104.4275v1}{arXiv: math.CT 1104.4275v1};
 
\bibitem[N]{N} Behrang Noohi, \emph{Foundations of topological stacks I},
 \href{http://arxiv.org/abs/math/0503247v1}{arXiv: math.AG 0503247v1};
 
\bibitem[Pr]{Pr} Dorette A. Pronk, \emph{\'Etendues and stacks as bicategories of fractions},
 Compositio Mathematica 102, 243--303 (1996), available at
 \url{http://www.numdam.org/item?id=CM_1996__102_3_243_0};

\bibitem[PW]{PW} Dorette A. Pronk, Michael A. Warren, \emph{Bicategorical fibration structures
 and stacks} (2013), \href{http://arxiv.org/abs/1303.0340v1}{arXiv: math.CT 1303.0340v1};

\bibitem[R]{R} David Michael Roberts, \emph{Internal categories, anafunctors and localisations}
 (2012), \href{http://arxiv.org/abs/1101.2363v3}{arXiv: math.CT 1101.2363v3};

\bibitem[St]{St} Ross Street, \emph{Fibrations in bicategories}, Cahier de topologie et
 g\'eom\'etrie diff\'erentielle cat\'egoriques, 21 (2), 111--160 (1980), available at
 \url{http://archive.numdam.org/article/CTGDC_1980__21_2_111_0.pdf};
 
\bibitem[T1]{T3} Matteo Tommasini, \emph{Some insights on bicategories of fractions - I} (2014),
 \href{http://arxiv.org/abs/1410.3990v2}{arXiv: math.CT 1410.3990v2};

\bibitem[T2]{T4} Matteo Tommasini, \emph{Some insights on bicategories of fractions - II} (2014),
 \href{http://arxiv.org/abs/1410.5075v2}{arXiv: math.CT 1410.5075v2}.

%\bibitem[T3]{T7} Matteo Tommasini, \emph{Weak fiber products of differentiable stacks} (2014),
% to appear.
\end{thebibliography}
\end{document}